\documentclass[12pt, a4paper]{article}
\usepackage{afterpage}
\usepackage{algorithm}
\usepackage{algorithmicx}
\usepackage{algpseudocode}
\usepackage{authblk}
\usepackage{amsmath}
\usepackage[title]{appendix}
\usepackage{amsfonts}
\usepackage{booktabs} 
\usepackage[font=small,labelfont=bf, hypcap=false]{caption} 
\usepackage{chngcntr}
\usepackage{comment} 
\usepackage{csquotes} 
\usepackage{csvsimple} 
\usepackage{diagbox}
\usepackage{enumitem}
\usepackage{eqnarray}
\usepackage{float} 
\usepackage[T1]{fontenc} 
\usepackage{geometry}
\usepackage{graphicx}
\usepackage[colorlinks=true, allcolors=blue]{hyperref}
\usepackage[none]{hyphenat} 
\usepackage{lineno} 
\usepackage{listings}
\usepackage{lipsum}
\usepackage{longtable}
\usepackage{lscape}
\usepackage{mathrsfs}
\usepackage{multirow}
\usepackage{orcidlink}
\usepackage{pdflscape}
\usepackage[]{ragged2e}
\usepackage{subcaption}
\usepackage{setspace} 
\usepackage{threeparttable} 
\usepackage{tabularray}
\usepackage{titlesec} 

\titleformat{\section} 
{\normalfont\Large\bfseries}{\thesection.}{1em}{}



\graphicspath{{fig/}{fig/2d/}{fig/3d}}

\usepackage{amssymb}
\usepackage{amsmath}
\usepackage{amsthm}
\usepackage{amsbsy}
\usepackage{color}
\newtheorem{theorem}{Theorem}[section]
  
   \newtheorem{definition}[theorem]{Definition}

    \newtheorem{remark}[theorem]{Remark}

\numberwithin{equation}{section}
\usepackage{subcaption}

\usepackage{bm}
\usepackage{stmaryrd}
\usepackage{xparse}

\NewDocumentCommand{\dgal}{sO{}m}{%
  \IfBooleanTF{#1}
    {\dgalext{#3}}
    {\dgalx[#2]{#3}}%
}

\NewDocumentCommand{\dgalext}{m}{%
  \sbox0{%
    \mathsurround=0pt 
    $\left\{\vphantom{#1}\right.\kern-\nulldelimiterspace$%
  }%
  \sbox2{\{}%
  \ifdim\ht0=\ht2
    \{\kern-.45\wd2 \{#1\}\kern-.45\wd2 \}%
  \else
    \left\{\kern-.5\wd0\left\{#1\right\}\kern-.5\wd0\right\}%
  \fi
}

\NewDocumentCommand{\dgalx}{om}{%
  \sbox0{\mathsurround=0pt$#1\{$}%
  \sbox2{\{}%
  \ifdim\ht0=\ht2
    \{\kern-.45\wd2 \{#2\}\kern-.45\wd2 \}%
  \else
    \mathopen{#1\{\kern-.5\wd0 #1\{}
    #2
    \mathclose{#1\}\kern-.5\wd0 #1\}}
  \fi
}

\title{Recovery techniques for ﬁnite element methods}
\author[1]{\small Hailong Guo}
\author[2]{\small Zhimin Zhang}
\affil[1]{\small{School of Mathematics and Statistics, The University of Melbourne, Parkville, VIC,  3052, Australia (\href{mailto:hailong.guo@unimelb.edu.au}{hailong.guo@unimelb.edu.au})}}
\affil[2]{\small{Department of Mathematics, Wayne State University, Detroit, 48202, MI, USA(\href{mailto:ag7761@wayne.edu}{ag7761@wayne.edu})}}
\date{}

\begin{document}
\newpage
\maketitle
\begin{abstract}
Post-processing techniques are essential tools for enhancing the accuracy of finite element approximations and achieving superconvergence. Among these, recovery techniques stand out as vital methods, playing significant roles in both post-processing and pre-processing. This paper provides an overview of recent developments in recovery techniques and their applications in adaptive computations. The discussion encompasses both gradient recovery and Hessian recovery methods. To establish the superconvergence properties of these techniques, two theoretical frameworks are introduced. Applications of these methods are demonstrated in constructing asymptotically exact {\it a posteriori} error estimators for second-order elliptic equations, fourth-order elliptic equations, and interface problems. Numerical experiments are performed to evaluate the asymptotic exactness of recovery type {\it a posteriori} error estimators.
\end{abstract}

ultraconvergence, asymptotically exact, gradient recovery, Hessian recovery, unfitted\\

\section{Introduction}
\label{sec:int}
In modern scientific computing, adaptive computation is one of the most important techniques for large-scale scientific computing, enabling problems to be solved in an optimal way. Roughly speaking, adaptive computation consists of a loop of four steps: Solve, Estimate, Mark, and Refine.

One of the key steps is how to obtain a reliable and computable error estimator. In the pioneering work by the late mathematicians Ivo Babu{\v{s}}ka and Werner C. Rheinboldt \cite{babuska1978aadaptive, babuska1978badaptive}, they proposed using the finite element solution to construct a computable  \textit{a posteriori} error estimator based on the residual of the partial differential equation. Since then, there has been extensive work in both mathematics and engineering on  \textit{a posteriori} error estimators, which can be roughly categorized into two different classes: residual type and recovery type. Other types of {\it a posteriori} error estimators include hierarchical error estimators \cite{bank1996hierarchical} and the equilibrated flux method \cite{ainsworth1993flux}, which can also be viewed as residual based.  Interested readers are referred to the monographs on adaptive finite element methods \cite{oden2000posterioribook, bangerth2003adaptivebook, verfurth2013aposteribook, babushka2001fembook, repin2008posteriorbook} and the review papers \cite{bonito2024adaptivereview, becker2001review, cc2002average, chamoin2023postererrorreview, nochetto2009adaptive, edtmayer2024adaptive} for the latest advances in this field.

Independently, there are another groups of computational mathematicians who found that in finite element methods, there are some special points that have better convergence than the optimal convergence of polynomial interpolation. These special points are the so-called superconvergent points \cite{wahlbin1995superconvergencebook, deboor1973super}. In the seminal work of Douglas and Dupont \cite{douglas1973superconvergenceproof, douglas1973super}, they proved that the finite element solution at nodes can achieve $2k$-order superconvergence for two-point boundary value problems when a $k$th order Lagrange element is used. After that, there has been tremendous pursuit of the theoretical analysis of superconvergence phenomena, with famous methods like the quasi-projection method \cite{douglas1974super}, local averaging method \cite{bramble1977averaging, thomee1977highorder}, element analysis method \cite{chen20122kconjecture, chen1998elementanalysis, cao2014superdg, cao2015fvm2k, zlamal1977somesuper, zlamal1978super}, computer-based method \cite{babushka1996computerassistant, lin2008naturalsuper}, and local symmetry theory \cite{schatz1996super, wahlbin1998generalprinciple}. See the monographs \cite{chen1995high, wahlbin1995superconvergencebook, zhu1989superconvergencebook, chen2001structure} for comprehensive results in superconvergence.

Those two ideas are magically connected when Zienkiewicz and Zhu tried to construct an easy-to-use  \textit{a posteriori} error estimate for the engineering community \cite{zhu1987simpleerror}, known as the ZZ estimator. They proposed the superconvergent patch recovery (SPR) method \cite{zhu1992spr1, zhu1992spr2}. The key idea of the SPR is to locally use the optimal sampling points of stress to reconstruct a more accurate gradient. The optimal sampling points turn out to be the superconvergent points, at least in the one-dimensional case \cite{zienkiewicz2013fembook}. Computational results have even shown that the recovered gradients can be superconvergent to the exact gradient. Since then, there have been extensive efforts to improve SPR and extend it from continuous finite element methods to other numerical methods, as seen in \cite{boroomand1997rep, boroomand1997irep, ubertini2004patchenergy, rodenas2007improvedspr, gonzalez2013recovery, bordas2007recovery, picasso2003anisotropic, maisano2006recovery, wei2010sprsurface, feischl2014bemzz, skumar2017spr4iga, bartels2024flux,fortin2024recovery}, just to name a few. To date, SPR has become a standard tool in practical science and engineering computation, as evidenced by its implementation in commercial finite element software such as Abaqus and ANSYS. To uncover the mechanism behind its extraordinary numerical performance, extensive theoretical investigations have been conducted, revealing that the recovered gradient using SPR possesses the superconvergence property\cite{zhang1998analysisofspr1, zhang1998analysisofspr2, libo1999analysisofspr, xu2004analysisofrecovery} or even the ultraconvergence property  \cite{zhang1996ultra, zhang2003ultranconvergencezz, zhang2000ultraconvergenceofpatchrecovery}.  The term "ultraconvergence" refers to a convergence rate that is two orders higher than the optimal rate, a concept first introduced by Zienkiewicz and Zhu in \cite{zhu1992spr1}.

Although SPR adopts the stress points as the sampling points, it loses the magic of superconvergence, even on some uniform meshes, such as uniform meshes with a chevron pattern. To achieve better superconvergence properties, Zhang and Naga proposed a new gradient recovery technique called polynomial preserving recovery (PPR) \cite{zhang2005ppr, naga2004pprposteriori, naga2005ppr2d3d}. Unlike SPR, PPR locally fits a higher-order polynomial using nodal points as sampling points, and then defines the recovered gradient by differentiation. Due to the reproducibility of the linear least-square procedure, PPR has been proven to preserve the exact gradient when the exact solution is a polynomial of degree $k+1$ and a Lagrange element of degree $k$ is used. This, in turn, implies the unconditional consistency of PPR on arbitrary meshes. As a result of  its improved superconvergent property, PPR has attracted considerable attention from both theoretical and practical perspectives, as seen in \cite{wu2007superconvergenceonadaptivemesh, li2024fvmppr, guo2016pprboundary, guo2017hessianrecovery, picasso2011hessianrecovery,  reusken2013levelsetrecovery, huang2014nonconvergenthessian, huang2005meshquality, caoweiming2015granisotropic, chen2014super3d, guo2017grifem, guo2018grcutfem}. It has become the superconvergent tool in the commercial finite element software COMSOL Multiphysics.

To demonstrate the main difference between SPR and PPR, we consider the following one-dimensional continuous piecewise linear finite element function $u_h$ on the mesh as plotted in Figure \ref{fig:vissol}. The gradient is a piecewise constant function. It is well known from classical approximation theory \cite{brenner2008fembook, ciarlet2002fembook} that it is a linear approximation to its continuous counterpart. To achieve a better approximation, the most straightforward way is to construct a piecewise linear function, denoted by $G_hu_h$, which can be completely defined through its value at each nodal point $x_i$. There are several different choices to define $G_hu_h(x_i)$. Different choices correspond to different gradient recovery methods in the literature. The popular choices are:
\begin{enumerate}
	\item Simple (weighted) averaging method \cite{wilson1963sa, cc2002average}: Define $G_hu_h(x_i)$ as
	 \begin{equation}\label{equ:sa}
	 	(G_hu_h)(x_i) = \frac{1}{2}( u_h'|_{K_{i-1}} +  u_h'|_{K_{i}}) \quad 
	 	\left(\text{ or  } \frac{|K_{i-1}|}{|K_{i-1}|+|K_{i}|} u_h'|_{K_{i-1}} +
	 	\frac{|K_{i}|}{|K_{i-1}|+|K_{i}|} u_h'|_{K_{i}}\right).
	 \end{equation}
	\item Superconvergent patch recovery \cite{zhu1992spr1, zhu1992spr2}: Firstly, fit a linear polynomial on $K_{i-1} \cup K_{i}$ in the sense of 
	\begin{equation}
	p_1 = \arg\min_{p\in\mathbb{P}_1(K_{i-1}\cup K_{i})}( (p(u_{i-1/2}) -  u_h'|_{K_{i-1}} )^2
		+ (p(u_{i+1/2}) -  u_h'|_{K_{i}} )^2);
	\end{equation}
	where $\mathbb{P}_1(K_{i-1}\cup K_{i})$ is the space of linear polynomials on $K_{i-1}\cup K_{i}$. The midpoints are defined as  $x_{i-1/2} = \frac{1}{2}(x_{i-1}+x_i)$ and $x_{i+1/2} = \frac{1}{2}(x_i+x_{i+1})$. Then, define $(G_hu_h)(x_i) = p_1(x_i)$.
	\item  Polynomial preserving recovery \cite{zhang2005ppr, naga2005ppr2d3d}: Firstly, fit a quadratic polynomial on $K_{i-1} \cup K_{i}$ in the sense of
	\begin{equation}
		p_2 = \arg\min_{p\in\mathbb{P}_2(K_{i-1}\cup K_{i})}( (p(x_{i-1}) -  u_h(x_{i-1}))^2
		+  (p(x_{i}) -  u_h(x_{i}))^2 +  (p(x_{i+1}) -  u_h(x_{i+1}))^2);
	\end{equation}
	where $\mathbb{P}_2(K_{i-1}\cup K_{i})$ is the space of quadratic polynomials on $K_{i-1}\cup K_{i}$.  Then, define $(G_hu_h)(x_i) = p_2'(x_i)$.
\end{enumerate}
When the meshes are uniform, the $(G_hu_h)(x_i)$ given by all the above three methods are the same, which is a second-order central finite difference scheme. Regarding the recovery operator as a finite difference stencil is also one of the key ideas to establish its superconvergence (or ultraconvergence) on translation invariant meshes \cite{zhang2005ppr, naga2005ppr2d3d, guo2017hessianrecovery}.  

\begin{figure}[!h]
   \centering
  \includegraphics[width=\textwidth]{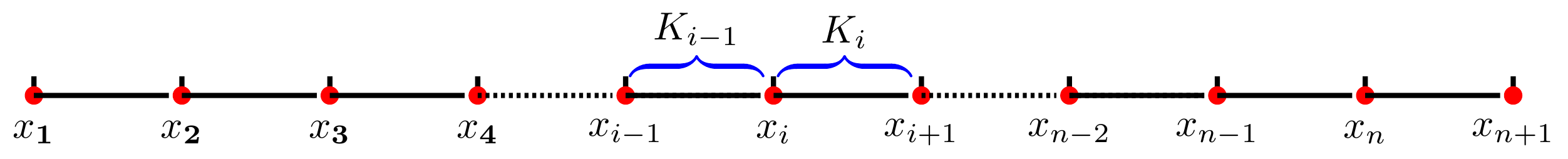}
   \caption{Illustration of the setup for recovery techniques in one-dimensional space. }
   \label{fig:vissol}
\end{figure}

The main purpose of this paper is to survey the recent developments in polynomial preserving recovery and its applications in adaptive computation. PPR offers a universal, meshfree framework for post-processing numerical solutions, allowing various numerical methods to be unified under the same conceptual umbrella. In this work, we focus on three distinct recovery techniques: gradient recovery, Hessian recovery, and recovery methods for non-smooth problems. For each technique, we discuss its mathematical properties, including the relationship with differentiation matrices.
To analyze the superconvergence properties of these recovery techniques, we present two general frameworks. The first framework leverages supercloseness results on mildly structured meshes. The second framework is based on the argument of superconvergence by different quotients on translation invariant meshes.
We also highlight the applications of these recovery techniques in the construction of easy-to-use {\it a posteriori} error estimators. A key distinguishing feature of the error estimators derived from these recovery techniques is their asymptotic exactness.

This review article is organized as follows. In Section \ref{sec:not}, we introduce the relevant notations for finite element methods. In Section \ref{sec:ppr}, we present the definition of three recovery techniques and discuss their polynomial preserving properties. In Section \ref{sec:super}, we introduce two frameworks to analyze the superconvergence properties of the recovery techniques. In Section \ref{sec:estimator}, we explore applications in constructing asymptotically exact {\it a posteriori} error estimators for three exemplary classes of partial differential equations. In Section \ref{sec:ne}, we provide extensive numerical examples to illustrate the performance of recovery type {\it a posteriori} error estimators. Finally, in Section \ref{sec:con}, we draw some conclusions and discuss other applications of recovery techniques.

\section{Notations}\label{sec:not}

For the sake of simplicity, we consider $\Omega$ to be a polygonal domain in $\mathbb{R}^2$ with Lipschitz continuous boundary. For any subdomain $A\subset \Omega$,  let $W^{k,p}(A)$ be the Sobolev space \cite{brenner2008fembook, ciarlet2002fembook,evans2010pdebook}  on $A$ with norms $\|\cdot\|_{k,p, A}$ and semi-norms $|\cdot|_{k,p, A}$, where $1\le p \le \infty$. When $p=2$, $W^{k, p}(A)$ is the $H^{k}(A)$ and the subindex $p$ is omitted in the corresponding (semi-)norms. The set of polynomials on $A$ with degree less than or equal to $k$ is denoted  as $\mathbb{P}_k(A)$. The dimensionality of  $\mathbb{P}_k(A)$ is $n_k = \frac{1}{2}(k+1)(k+2)$.

Let $\alpha = (\alpha_1, \alpha_2)$ be a 2-index, and its length is given by $|\alpha| = \alpha_1 + \alpha_2$. With the 2-index notation, the weak derivative of a function $u$ can be denoted as
\begin{equation}\label{equ:multindex}
	D^{\alpha}u := \frac{\partial^{|\alpha|}u}{\partial x^{\alpha_1}\partial y^{\alpha_2}}.
\end{equation}
In particular, the Hessian operator is denoted as
\begin{equation}
	H=\left(\begin{array}{cc}
\displaystyle\frac{\partial^2}{\partial x^2} & \displaystyle\frac{\partial^2}{\partial x \partial y} \\
\displaystyle\frac{\partial^2}{\partial y \partial x} & \displaystyle\frac{\partial^2}{\partial y^2}
\end{array}\right) .
\end{equation}

Let $\mathcal{T}_h$ be a shape-regular triangulation of $\Omega$ into triangles \cite{brenner2008fembook, ciarlet2002fembook}, i.e. $\overline{\Omega} = \bigcup\limits_{K\in\mathcal{T}_h}K$.  Denote by $h_K$ the diameter of an element $K$ in $\mathcal{T}_h$, and let $h = \max\limits_{K\in\mathcal{T}_h}h_K$.  In this paper, we  focus on the continuous Lagrange finite element method. We define the finite element space $S_{h}$ as
\begin{equation}\label{equ:femspace}
	S_{h} = \left\{v_h \in C^0(\Omega): v_h|_{K} \in \mathbb{P}_k(K) \quad \forall K \in \mathcal{T}_h \right\}. 
\end{equation}
Let $\mathcal{N}_h$ denote the set of nodal points, and let $\mathcal{E}_h$ denote the set of edges in the triangulation $\mathcal{T}_h$.  The nodal basis functions of $S_{h}$ are
denoted by $\{\phi_{i}\}_{i =1}^{|\mathcal{N}_h|}$ with $\phi_z(z') = \delta_{zz'}$, where $\delta_{zz'}$ is the Kronecker delta and $|\mathcal{N}_h|$ is the cardinality of $\mathcal{N}_h$. For any function $u\in C^0(\bar\Omega)$, let $u_I$ be the interpolation of $u$ in $S_h$, i.e.  
\begin{equation}\label{equ:interp}
	u_I = \sum_{i=1}^{|\mathcal{N}_h|}u(z_i) \phi_i.
\end{equation}

For $A \subset \Omega$, denote by $S_{h}(A)$ the restriction of functions  $S_{h}$ in $A$ and by $S_{h}^{\text{comp}}(A)$ the  set of  functions in $S_{h}(A)$ with  compact support in the interior of $A$.  Let  $ \Omega_0 \subset\subset \Omega_1 \subset\subset \Omega_2 \subset\subset \Omega $ be  the nested subdomains separated by a distance $ d \geq ch $.  Suppose $\tau$ is a parameter which is typically a constant times $h$.  Similar to \cite{wahlbin1995superconvergencebook},  the translation operator $T^{\ell}_{\tau}$ is defined as 
  \begin{equation} \label{equ:translateoperator}
  	T^{\ell}_{\tau}v(x) = v(x+\tau\ell), 
  \end{equation}
and 
  \begin{equation} \label{equ:translate}
  	T^{\ell}_{\nu\tau}v(x) = v(x+\nu\tau\ell), 
  \end{equation}
 for some integer $\nu$.  The finite element space $S_{h}$ is called translation invariant  by $\tau$ in the direction $\ell$ \cite{wahlbin1995superconvergencebook} if 
 \begin{equation}\label{equ:translationinvariant}
     T^{\ell}_{\nu\tau}v \in S_{h}^{\text{comp}}(\Omega_2), \quad  \forall
     v \in S_{h}^{\text{comp}}(\Omega_1).
 \end{equation}
In that case, we also say the mesh $\mathcal{T}_h$ is translation invariant.  As the demonstrated in \cite{guo2017hessianrecovery}, the uniform meshes of regular, chevron, criss-cross,  and unionjack patterns are translation invariant  as demonstrated in  Figure \ref{fig:timesh}. For example, the uniform mesh of the regular pattern is translation invariant by $h$ in the directions $(1,0)$ and $(0,1)$,
by $2\sqrt{2}h$ in the directions $\left( \pm \frac{\sqrt{2}}{2}, \frac{\sqrt{2}}{2} \right)$,
and by $\sqrt{5}h$ in the directions
$\left( \frac{2\sqrt{5}}{5}, \pm \frac{\sqrt{5}}{5} \right)$ and $\left( \pm \frac{\sqrt{5}}{5}, \frac{2\sqrt{5}}{5} \right)$, etc.

\begin{figure}[!h]
\captionsetup[subfigure]{}
   \centering
   \subcaptionbox{ \label{fig:regular}}
  {\includegraphics[width=0.24\textwidth]{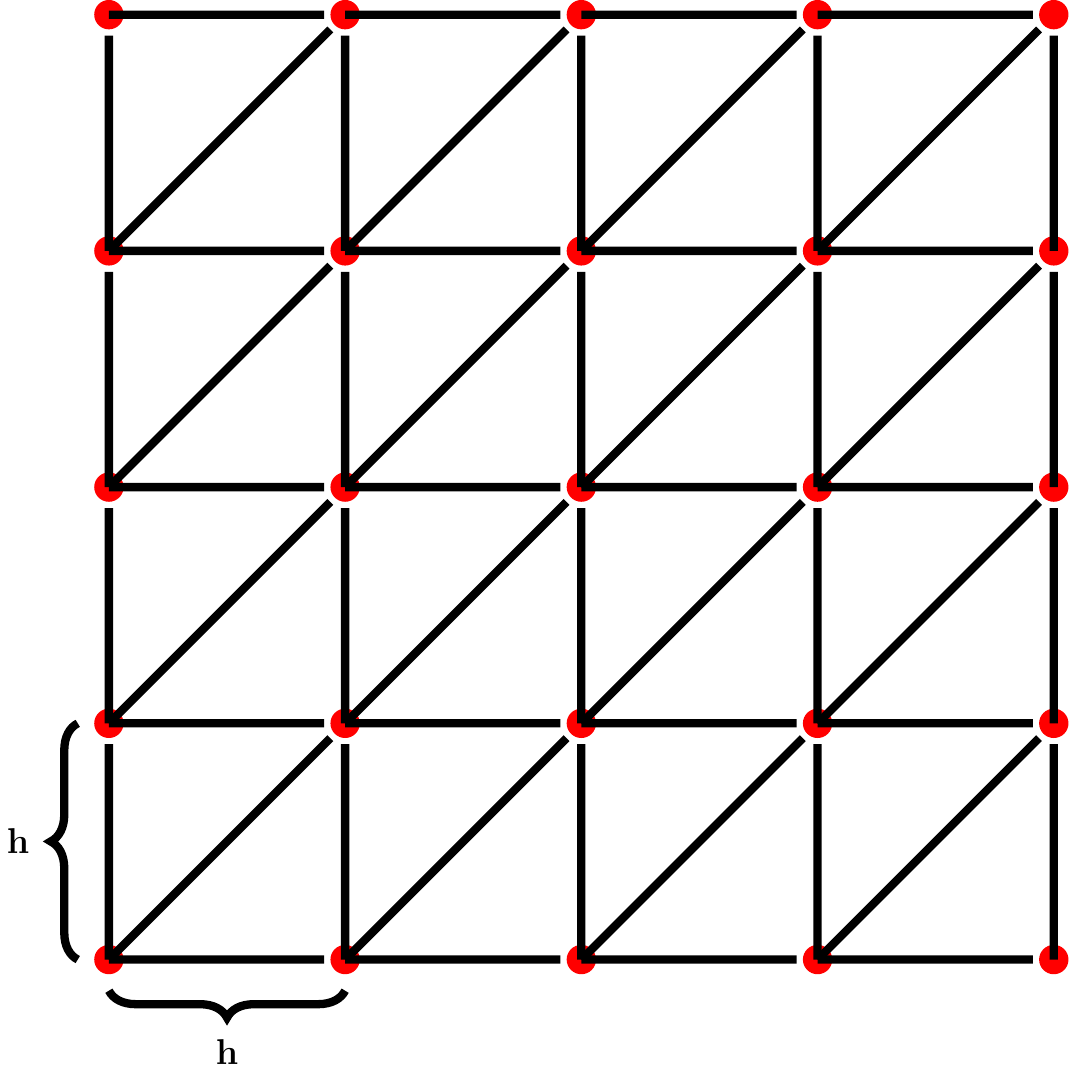}}
  \subcaptionbox{\label{fig:chevron}}
   {\includegraphics[width=0.24\textwidth]{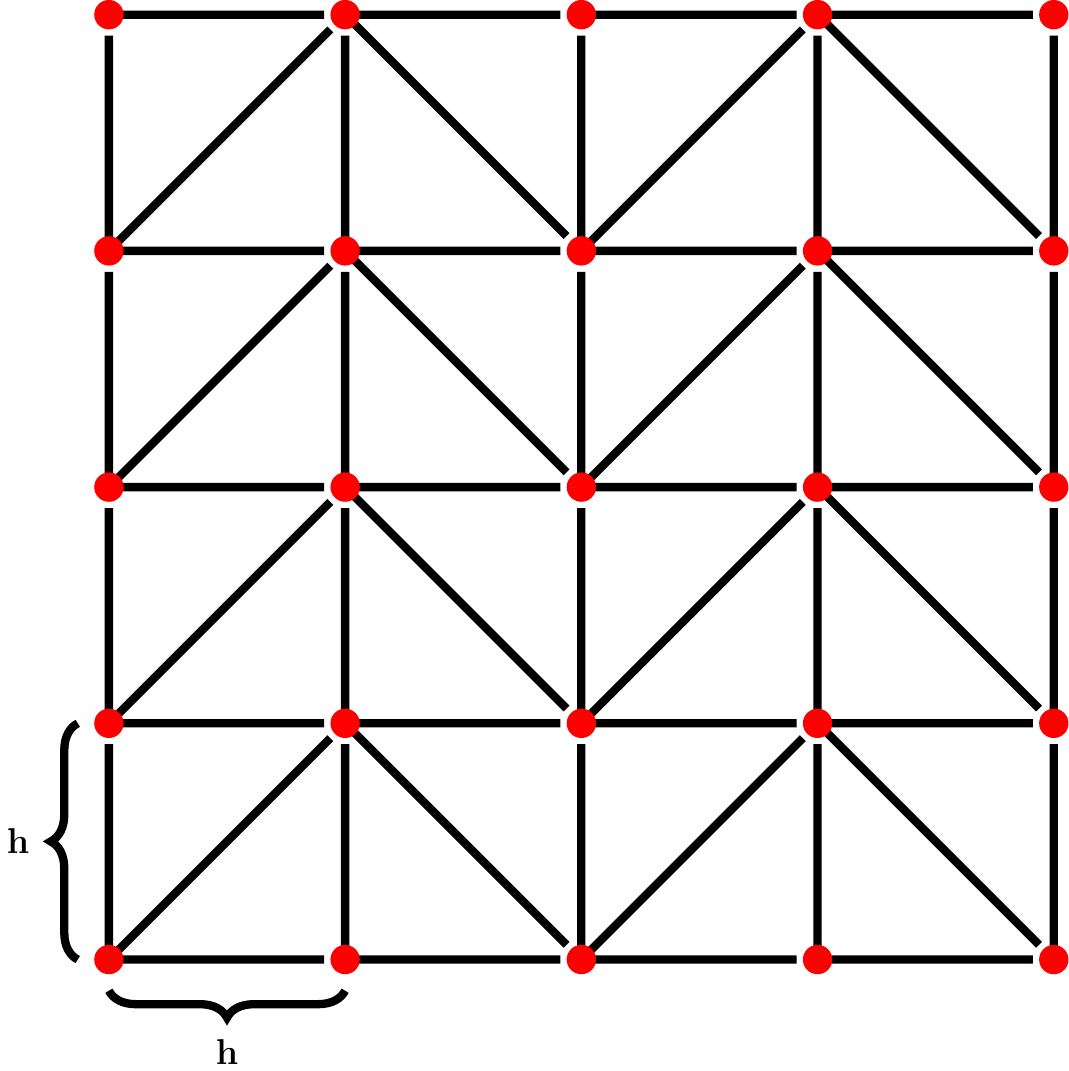}}
      \subcaptionbox{ \label{fig:crisscross}}
  {\includegraphics[width=0.25\textwidth]{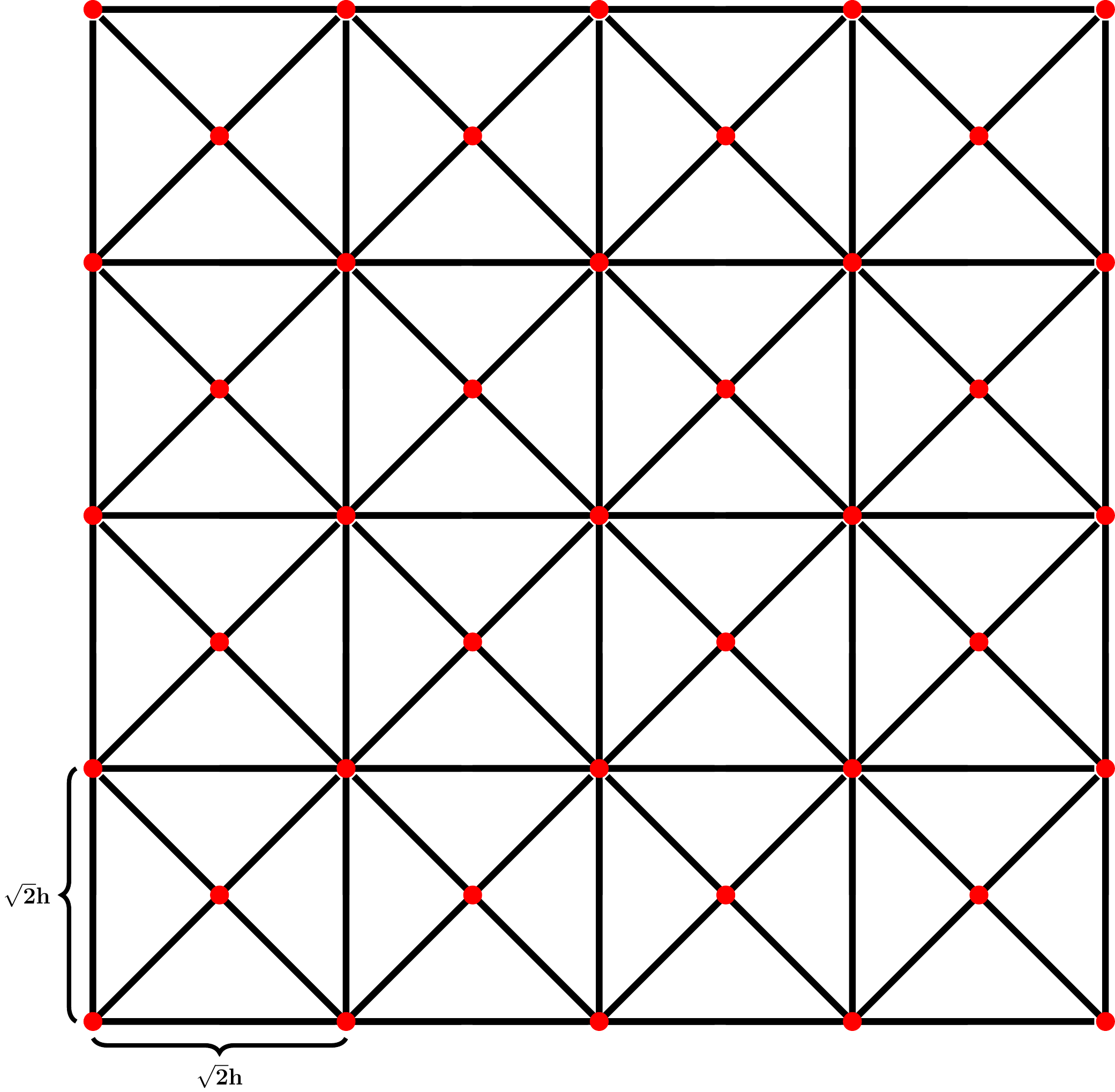}}
  \subcaptionbox{\label{fig:unionjack}}
   {\includegraphics[width=0.24\textwidth]{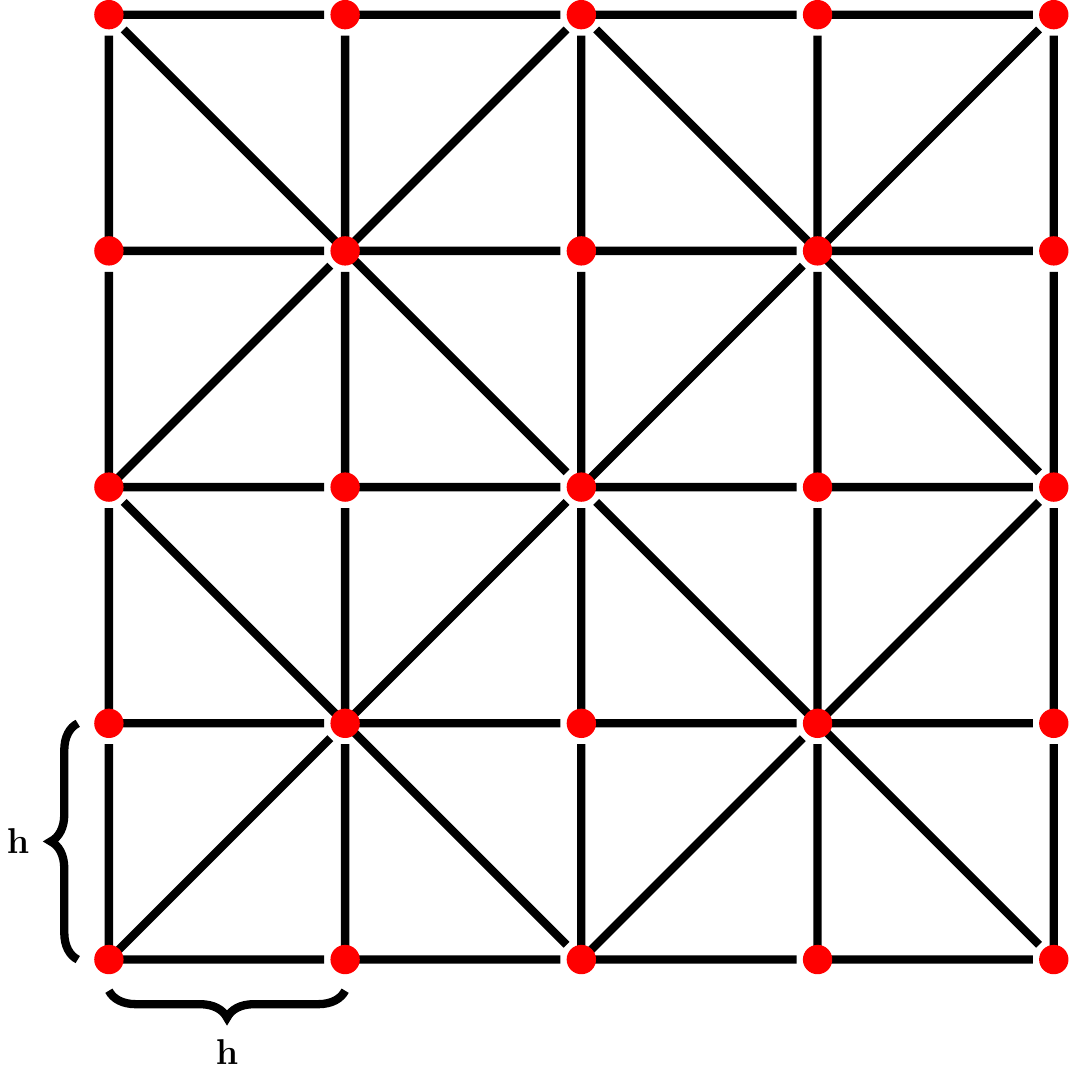}}
   \caption{Translation invariant meshes on the unit square: (a) uniform mesh with a regular pattern; (b) uniform mesh with a chevron pattern; (c) uniform mesh with a crisscross pattern; (d) uniform mesh with a Union Jack pattern. }\label{fig:timesh}
\end{figure}

To facilitate the  superconvergence analysis of the recovery operator on unstructured meshes,  we need to put some conditions on the mesh $\mathcal{T}_h$. Two adjacent triangles are said to form an $\mathcal{O}(h^{1+\alpha})$ approximate parallelogram if the lengths of any two opposite edges differ only by $\mathcal{O}(h^{1+\alpha})$. Based on this, we can define the following Condition $(\sigma, \alpha)$ for mildly structured meshes \cite{xu2004analysisofrecovery, guo2017hessianrecovery}:
\begin{definition}\label{def:cond}
	The triangulation $\mathcal{T}_h$ is said to satisfy the Condition $(\sigma, \alpha)$ if there exists a partition $\mathcal{T}_{1,h} \cup \mathcal{T}_{2,h}$ of $\mathcal{T}_h$ and positive constants $\alpha$ and $\sigma$ such that any two adjacent triangles in $\mathcal{T}_{1,h}$ form an $O\left(h^{1+\alpha}\right)$ parallelogram and
\begin{equation}
	\sum_{K \in \mathcal{T}_{2,h}}|K|=O\left(h^\sigma\right).
\end{equation}
\end{definition}

\begin{remark}
	We emphasize that Condition $(\sigma, \alpha)$ is not an artificial requirement. It is naturally satisfied by meshes generated by automated mesh generators, such as the Delaunay mesh generator \cite{liseikin1999meshgeneration}. This generator employs a \emph{diagonal-swapping} procedure, which adjusts the direction of certain diagonal edges to align adjacent element edges more closely and maximize the number of nodes with six attached triangles.
\end{remark}

Throughout this article, the letters $C$ or $c$, with or without subscripts, denote a generic constant that is independent of the mesh size $h$ and may vary with each occurrence. For convenience, we denote the relation $x \leq Cy$ as $x \lesssim y$.

\section{Polynomial preserving recovery}\label{sec:ppr}
In this section, we outline the construction of the polynomial preserving recovery method. We begin in Subsection \ref{ssec:gr} with an introduction to gradient recovery techniques. In the subsequent subsection, we focus on the definition of Hessian recovery techniques. In the final subsection, we explain gradient recovery techniques for both body-fitted and unfitted numerical methods for solving interface problems.

\subsection{Gradient recovery technique}\label{ssec:gr} Let $S_h$ be the continuous finite element space consisting of piecewise polynomials of degree $k$, as defined in \eqref{equ:femspace}. The PPR gradient recovery operator $G_h$ is defined as a linear operator mapping from $S_h$ to $S_h \times S_h$. A key observation is that the recovered gradient $G_h u_h$ is a function in the finite element space $S_h \times S_h$. Therefore, it is sufficient to define its value at each nodal point. The process can be formulated in three steps: (1) construct local patches of elements; (2) perform local recovery procedures; and (3) assemble the recovered data into a global expression.

\begin{figure}[!h]
   \centering
   \subcaptionbox{\label{fig:firstlaye}}
  {\includegraphics[width=0.32\textwidth]{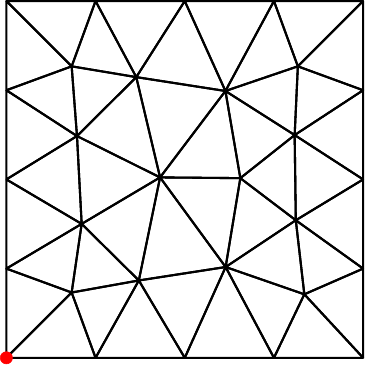}}
  \subcaptionbox{\label{fig:secondlayer}}
   {\includegraphics[width=0.32\textwidth]{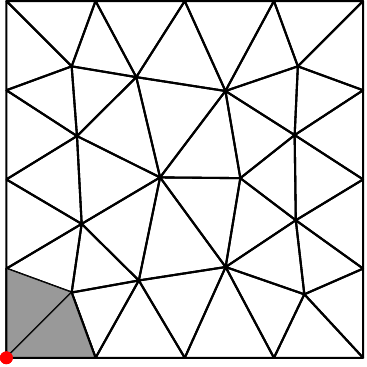}}
  \subcaptionbox{\label{fig:thirdlayer}}
  {\includegraphics[width=0.32\textwidth]{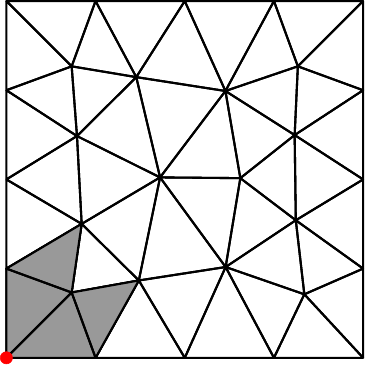}}
   \caption{Illustration of the definition of $L(z_i, n)$: (a) Plot of $L(z_i, 0)$; (b) Plot of $L(z_i, 1)$;
   (c) Plot of $L(z_i, 2)$.}\label{fig:patchlayer}
\end{figure}

To facilitate the construction of the local element patch, we first introduce the concept of a mesh layer around a vertex.  Let $z_i \in \mathcal{N}_h$ be a vertex of $\mathcal{T}_h$. For any positive integer $n$, we define $\mathcal{L}(z_i, n)$ as follows:
\begin{equation}
	\mathcal{L}\left(z_i, n\right)=
	 \begin{cases}z_i, & \text { if } n=0, \\
	 \bigcup\left\{K: K \in \mathcal{T}_h, K\cap \mathcal{L}\left(z_i, 0\right) \neq \emptyset\right\}, & \text { if } n=1, \\
	  \bigcup\left\{K: K \in \mathcal{T}_h, K \cap \mathcal{L}\left(z_i, n-1\right) \text { is an edge in } \mathcal{E}_h\right\}, & \text { if } n \geq 2 .
	 \end{cases}
\end{equation}
In Figure \ref{fig:patchlayer}, we illustrate $\mathcal{L}(z_i, n)$, where $z_i$ is represented as the red dotted point. As shown, $\mathcal{L}(z_i, 0)$ contains only the vertex $z_i$, while $\mathcal{L}(z_i, 1)$ consists of all elements that have $z_i$ as a vertex. Furthermore, $\mathcal{L}(z_i, 2)$ includes all elements in $\mathcal{L}(z_i, 1)$ along with their neighboring elements.
The local element patch $\Omega_{z_i}$ at $z_i$ is defined as
\begin{equation}
	\Omega_{z_i} = \mathcal{L}\left(z_i, n_i\right),
\end{equation}
where $n_i$ is the smallest integer that ensures a unique $(k+1)$th-order polynomial can be fitted in the least-squares sense using the nodal points in $\Omega_{z_i}$ as sampling points. For a vertex $z_i$ located on the boundary $\partial\Omega$, the local element patch $\Omega_{z_i}$ is defined as \cite{guo2016pprboundary}
\begin{equation}
\Omega_{z_i} = \bigcup_{j=0}^{J}\Omega_{z_{i_j},
}
\end{equation}
where $z_{i_j}$ ($j = 0, \ldots, J$) denotes the adjacent interior vertices of $z_i$.

For each $z_i \in \mathcal{N}_h$, it belongs to one of three categories: vertices, edge nodes, and internal nodes. First, we consider the case where the nodal point $z_i$ is a vertex. Let $B_{z_i}$ be the set of nodal points in $\Omega_{z_i}$ and $I_i$ be the set of indices of $B_{z_i}$. Using the nodal points $B_{z_i}$ as sampling points, we fit a polynomial of degree $(k+1)$ at $z_i$ in the following least-squares sense:
\begin{equation}\label{equ:localls}
    p_{z_i}(z) = \arg\min_{p \in \mathbb{P}_{k+1}(\Omega_{z_i})} \sum_{j \in I_i} \left| p(z_{i_j}) - u_h(z_{i_j}) \right|^2.
\end{equation}

To ensure numerical stability in practical computations, all calculations are performed in the local coordinate system. For such purpose, let
\begin{equation}
h_{z_i} := \max\{|z_{i_j} - z_{i_{\ell}}|: 
z_{i_j}, z_{i_{\ell}}\in B_{z_i} \} ;
\end{equation}
and define the coordinate transform map $F$ as
\begin{equation}\label{equ:coordtran}
	F: (x, y) \rightarrow (\xi, \eta) = \frac{(x,y)-(x_i, y_i)}{h_{z_i}},
\end{equation}
where $z = (x,y)$ and $\hat{z} = (\xi, \eta)$.  In term of the new local coordinate system, we can write the fitted polynomial $p_{z_i}$ as 
\begin{equation}\label{equ:rel}
	p_{z_i} = \mathbf{p}^T\mathbf{a} = \hat{\mathbf{p}}^T\hat{\mathbf{a}},
\end{equation}
where 
\begin{align*}
	&\mathbf{p}^T = (1, x, y, x^2, xy, y^2, \cdots, x^{k+1}, x^ky, \cdots, xy^k, y^{k+1}), \\
	&\hat{\mathbf{p}}^T = (1, \xi, \eta, \xi^2, \xi\eta, \eta^2, \cdots, \xi^{k+1}, \xi^k\eta, \cdots, \xi\eta^k, \eta^{k+1}),\\
	&\mathbf{a}^T = (a_1, a_2, a_3, a_4, a_5, a_6, \cdots, a_{n_{k+1}-5},  \cdots, a_{n_{k+1}}), \\
	&\hat{\mathbf{a}}^T = (a_1, h_{z_i}a_2, h_{z_i}a_3, h_{z_i}^2a_4, h_{z_i}^2a_5, h_{z_i}^2a_6, \cdots, h_{z_i}^{k+1}a_{n_{k+1}-5},  \cdots, h_{z_i}^{k+1}a_{n_{k+1}}).  
\end{align*}
Let $\hat{z}_{i_j} = F(z_{i_j})$ for all $j\in I_i$. The coefficients $\hat{\mathbf{a}}$ can be obtained by solving the normal equation
\begin{equation}\label{equ:normalequation}
	\hat{\mathbf{A}}^T\hat{\mathbf{A}}\hat{\mathbf{a}} = \hat{\mathbf{A}}^T\hat{\mathbf{b}},
\end{equation}
where 
\begin{equation*}
	\hat{\mathbf{A}}
	= \begin{pmatrix}
		1 & \xi_{i_0} & \eta_{i_0} & \xi_{i_0}^2 & \xi_{i_0}\eta_{i_0} & \eta_{i_0}^2 
		& \cdots & \xi_{i_0}^{k+1}& \cdots & \eta_{i_0}^{k+1} \\
		1 & \xi_{i_1} & \eta_{i_1} & \xi_{i_1}^2 & \xi_{i_1}\eta_{i_1} & \eta_{i_1}^2 
		& \cdots & \xi_{i_1}^{k+1}& \cdots & \eta_{i_1}^{k+1} \\
		\vdots & \vdots & \vdots & \vdots & \vdots & \vdots &  \ddots & \vdots & \ddots & \vdots \\
		1 & \xi_{i_{|I_i|-1}} & \eta_{i_{|I_i|-1}} & \xi_{i_{|I_i|-1}}^2 & \xi_{i_{|I_i|-1}}\eta_{i_{|I_i|-1}} & \eta_{i_{|I_i|-1}}^2 
		& \cdots & \xi_{i_{|I_i|-1}}^{k+1}& \cdots & \eta_{i_{|I_i|-1}}^{k+1}
	\end{pmatrix},
\end{equation*}
and 
\begin{equation*}
		\hat{\mathbf{b}} = 
	\begin{pmatrix}
		u_h(z_{i_0})\\
		u_h(z_{i_1})\\
		\vdots \\
		u_h(z_{i_{|I_i|-1}})
	\end{pmatrix},
\end{equation*}
with $|I_i|$ being the cardinality of the set $I_i$.  The recovered gradient $G_hu_h$ at the vertex $z_i$ is defined as 
\begin{equation}\label{equ:recoveredgrad}
	(G_hu_h)(z_i) :=  
	\frac{1}{h_{z_i}}
	\begin{pmatrix}
		\hat{a}_2\\
		\hat{a}_3 \\
	\end{pmatrix}.
\end{equation}

Second, if $z_i$ is a nodal point located on an edge connecting two vertices $z_{i_1}$ and $z_{i_2}$, the recovered gradient at $z_i$ is defined as
\begin{equation}\label{equ:edgegr}
    (G_h u_h)(z_i) = \beta p_{z_{i_1}}(z_i) + (1 - \beta) p_{z_{i_2}}(z_i),
\end{equation}
where $p_{z_{i_j}}$ is the least-squares fitted polynomial at $z_{i_j}$ for $j = 1, 2$, and $\beta$ is determined by the ratio of the distances of $z_i$ to $z_{i_1}$ and $z_{i_2}$.

Similarly, if $z_i$ is an interior nodal point located in a triangle formed by three vertices $z_{i_1}$, $z_{i_2}$, and $z_{i_3}$, the recovered gradient at $z_i$ is defined as
\begin{equation}\label{equ:elemgr}
    (G_h u_h)(z_i) = \sum_{j=1}^3 \beta_j p_{z_{i_j}}(z_i),
\end{equation}
where $\beta_j$ is the barycentric coordinate of $z_i$.

Once the recovered gradient $(G_hu_h)(z_i)$ is defined for all $z_i\in\mathcal{N}_h$,  the global recovered gradient can be interpolated as
\begin{equation}
	G_hu_h =  \sum_{i=1}^{|\mathcal{N}_h|}(G_hu_h)(z_i)\phi_{i}. 
\end{equation}

One key observation is that the gradient recovery operator $G_h$ is a linear operator from $S_h$ to $S_h \times S_h$. From linear algebra, it is known  that a linear operator from one finite-dimensional vector space to itself can be represented as a matrix linear transformation \cite{axler2024labook}. Let $\Phi$ be
the vector of basis functions, i.e. $\Phi = (\phi_{1}, \cdots, \phi_{|\mathcal{N}_h|})^T$.  Then for any function $u_h \in S_h$, it can be rewritten as 
\begin{equation}\label{equ:vecform}
u_h = \sum_{i=1}^{|\mathcal{N}_h|}u_{i}\phi_{i} =  \mathbf{u}^{T}\Phi, 
\end{equation}
where $\mathbf{u} = (u_1, \cdots, u_{|\mathcal{N}_h|})^T$ and $u_i$ is the value of $u_h$ at nodal point $z_i$.   Similarly,  the recovered gradient 
$G_hu_h$ can also be expressed as 
\begin{equation}
G_hu_h =
\begin{pmatrix}
	 G^x_hu_h\\
	 G^y_hu_h 
\end{pmatrix}
= 
\begin{pmatrix}
	 \mathbf{u}_x^{T}\Phi\\
	  \mathbf{u}_y^{T}\Phi
\end{pmatrix},
\label{equ:grad}
\end{equation}
where $\mathbf{u}_x$ and $\mathbf{u}_y$ are the vectors containing the recovered gradient values at the nodal points.  Since $G_h^x$ and $G_h^y$ are two linear  operator  from $S_{h}$ to $S_{h}$,  there exist two matrices ${\mbox{B}_h^x} \in \mathbb{R}^{|\mathcal{N}_h|\times |\mathcal{N}_h|}$ and ${\mbox{B}_h^y} \in \mathbb{R}^{|\mathcal{N}_h|\times |\mathcal{N}_h|}$ such that
\begin{equation}
\mathbf{u}_x = \mbox{B}_h^x \mathbf{u} \quad \text{            and           }\quad \mathbf{u}_y = \mbox{B}_h^y\mathbf{u}.
\label{equ:gradmatrix}
\end{equation}
The matrices $\mbox{B}_h^x$ and $\mbox{B}_h^y$ are referred to as the first-order differentiation matrices and play a similar role to the differentiation matrices in the spectral collocation method \cite{shen2011spectralbook}. By the definition of polynomial preserving recovery, both $\mbox{B}_h^x$ and $\mbox{B}_h^y$ are sparse matrices. In other words, the sparse matrices $\mbox{B}_h^x$ and $\mbox{B}_h^y$ can be constructed and reused. Consequently, the recovery procedure reduces to sparse matrix-vector multiplications, which can be performed in $\mathcal{O}(|\mathcal{N}_h|)$ operations.

\begin{remark}
	We emphasize that the PPR framework can be viewed as a unified approach for post-processing a wide class of numerical methods, encompassing not only classical finite element methods on simplicial, quadrilateral, or hexahedral meshes but also extending to modern methods based on polygonal and polyhedral meshes, such as mimetic finite difference methods \cite{beirao2014mfdmbook, shashkov1996mfdmbook} and virtual element methods \cite{beirao2013vem}. Additionally, PPR can be interpreted as a meshfree method, where the approach remains applicable as long as values of the numerical solution at specific nodal points are directly available or can be indirectly obtained via averaging procedures. In fact, we can alternatively define $B_{z_i}$ as
\begin{equation*}
    B_{z_i} = \left\{ z \in \mathcal{N}_h : |z - z_i| \le h_{z_i} \right\},
\end{equation*}
where $h_{z_i}$ is chosen to ensure that $B_{z_i}$ contains at least $n_{k+1}$ nodal points, thus guaranteeing the uniqueness of solutions to the least-squares problem \eqref{equ:localls}.  Geometric conditions for the uniqueness of the linear least-squares problem \eqref{equ:localls} were discussed by Zhang and Naga in \cite{naga2004pprposteriori}. For recent advancements in PPR applications to other numerical methods, we refer interested readers to \cite{du2019dgppr, song2015ipdgppr, song2015superdg} for discontinuous Galerkin methods, \cite{li2024fvmppr} for finite volume methods, \cite{guo2015crelement, guo2020scr} for (surface) Crouzeix-Raviart elements, \cite{dong2020pppr} for surface finite element methods, \cite{wang2019ppredgeelement} for N\'ed\'elec edge elements, and \cite{guo2019grvem, chi2019vemppr} for virtual element methods.
\end{remark}

\begin{remark}
We can make several improvements to the least-squares process. First, as observed in \cite{dong2020pppr}, the recovery procedure does not improve the accuracy of function approximation. We can remove one degree of freedom (DOF) in the least-squares fitting procedure by assuming
\begin{equation}
    \tilde{p}_{z_i}(\xi, \eta) = u_h(z_i) + \sum_{0 \le i, j \le k+1, i+j \le k+1} \hat{a}_{ij} \xi^i \eta^j,
\end{equation}
which reduces the size of the least-squares problem by 1. In particular, for the linear element case, we only need to solve a $5 \times 5$ rather than a $6 \times 6$ linear system. Second, formulating the linear least-squares problem using monomial basis functions on uniformly distributed sampling points is ill-conditioned for high-order elements.  We can address this by using orthogonal polynomials, such as the Legendre polynomials \cite{shen2011spectralbook}, as basis functions in the polynomial space of degree $k+1$.
\end{remark}

\begin{remark}
Anithworth and Oden stated in \cite{oden2000posterioribook} that a good recovery operator should satisfy three conditions: the consistency condition, the localization condition, and the boundedness \& linearity conditions. The polynomial preserving recovery satisfies all three conditions. In fact, the second condition can be restated as the requirement that the first-order differentiation matrices $\mathrm{B}_h^x$ and $\mathrm{B}_h^y$ are sparse. In contrast, the first-order differentiation matrices of the global $L^2$-projection \cite{hinton1974globall2} involve the inverse of a sparse matrix that is dense.   We would also like to emphasize that sparsity is a crucial component when the recovery operator is employed for pre-processing purposes, such as constructing new finite element methods for fourth-order elliptic problems, as demonstrated in \cite{guo2018linearfem, chen2017rbfemeigenvalue, xu2019hbfem}.
\end{remark}

\begin{remark}
In general, recovery techniques deteriorate near the boundary, necessitating special treatment in these regions. In \cite{guo2016pprboundary}, Guo et al. proposed two strategies to improve the performance of PPR near the boundary. The first strategy involves initially applying the standard PPR local least-squares fitting procedure for nearby interior vertices, followed by averaging the gradients of the obtained polynomials at the target boundary vertex to compute the recovered gradient. The second strategy constructs a relatively large element patch by merging the element patches of selected interior vertices near the target boundary vertex. The key idea is to replace the boundary patch with interior patches, which contain more information and exhibit certain symmetric properties.
\end{remark}

\begin{remark}
 The gradient recovery operator $G_h$ can also be interpreted as  a "smoothing" operator. It smooths a piecewise defined (discontinuous) function into a piecewise defined continuous function, making further differentiation possible. Such an observation is crucial for designing linear finite element methods for higher-order partial differential equations \cite{ cai2024surfacebiharmonic, guo2018linearfem, chen2017rbfemeigenvalue}.
\end{remark}

\begin{figure}[!h]
\captionsetup[subfigure]{}
   \centering
   \subcaptionbox{ \label{fig:gr_regular}}
  {\includegraphics[width=0.38\textwidth]{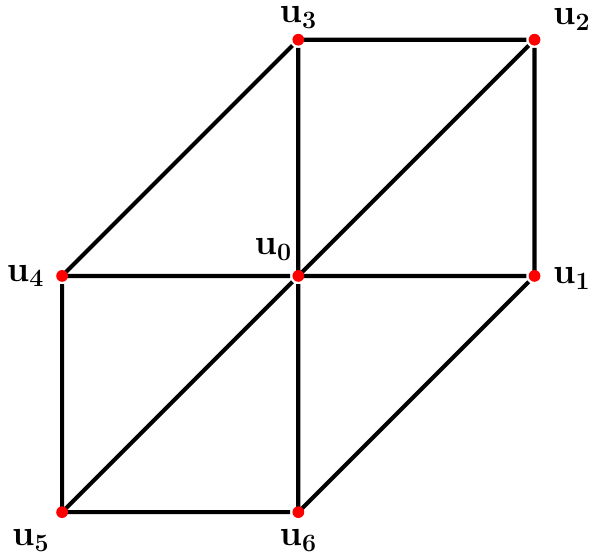}}
  \subcaptionbox{\label{fig:gr_chevron}}
   {\includegraphics[width=0.38\textwidth]{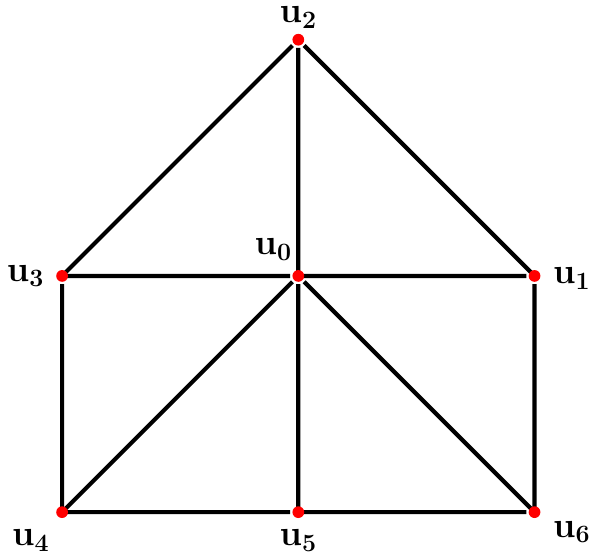}}
   \caption{Gradient recovery on uniform meshes: (a) regular pattern; (b) chevron pattern. }\label{fig:gr}
\end{figure}

To illustrate the main idea of PPR, we present two examples using continuous linear elements on uniform meshes.

\textbf{Example 1}: In this example, we consider the uniform mesh with a regular pattern. The local numbering of nodal points is from 0 to 6, as shown in Figure \ref{fig:gr_regular}. The local coordinates are defined as
\begin{equation}
	\boldsymbol{\xi}^T=(0, 1, 1, 0, -1, -1, 0) \quad \text{and} \quad 
	\boldsymbol{\eta}^T = (0, 0, 1, 1, 0, -1, -1).
\end{equation}
Let $\mathbf{e}^T = (1, 1, 1, 1, 1, 1, 1)$. 
 Then, we define the matrix  $A = (\mathbf{e}, \boldsymbol{\xi}, \boldsymbol{\eta}, \boldsymbol{\xi}^2, \boldsymbol{\xi}\boldsymbol{\eta}, \boldsymbol{\eta}^2)$.  We can compute 
\begin{equation}
(\hat{\mathbf{A}}^T\hat{\mathbf{A}})^{-1}\hat{\mathbf{A}}^T = \frac{1}{6}
	\begin{pmatrix}
		    6&     0  &    0  &  0 & 0  & 0 & 0 \\        
		    0 & 2 & 1 & -1 & -2 & -1 &1 \\
		    0 & -1 & 1 & 2 & 1 & -1 & -2 \\
		    -6 & 3 & 0 & 0 & 3 & 0 & 0 \\
		    6 & -3 & 3 & -3 & -3 & 3 & -3\\
		    -6 & 0 & 0 & 3 & 0 & 0 & 3
			\end{pmatrix}.
\end{equation}
Next, consider the values of an analytic function  $u$  at the sample points instead of the finite element solution  $u_h$ . The solution of the normal equation \eqref{equ:normalequation} and the relationship \eqref{equ:rel} allow us to express 
\begin{equation}
\begin{aligned}
 p_{z_0}(x, y)&=u_0+\frac{1}{6 h}\left[2\left(u_1-u_4\right)+u_2-u_3+u_6-u_5\right] x \\
& +\frac{1}{6 h}\left[2\left(u_3-u_6\right)+u_2-u_1+u_4-u_5\right] y+\frac{1}{2 h^2}\left(u_1-2 u_0+u_4\right) x^2 \\
& +\frac{1}{2 h^2}\left(2u_0 + u_2 -  u_1-u_3-u_4+u_5-u_6\right) x y+\frac{1}{2 h^2}\left(u_3-2 u_0+u_6\right) y^2 .
\end{aligned}
\end{equation}
From the definition \eqref{equ:recoveredgrad}, the recovered gradient at  $z_0 = (0, 0)$ is 
\begin{equation}\label{equ:regularstencil}
	\left(G_h u\right)\left(z_0\right)=\frac{1}{6 h}\left(\binom{2}{-1} u_1+\binom{1}{1} u_2+\binom{-1}{2} u_3+\binom{-2}{1} u_4+\binom{-1}{-1} u_5+\binom{1}{-2} u_6\right) . 
\end{equation}
Using Taylor expansion, we can show that
\begin{equation}
G_hu(z_0) = 
\left(
\begin{array}{c}
  u_x(z_0)+\frac{1}{6}(u_{xxx}(z_0)+u_{xxy}(z_0)+u_{xyy}(z_0))h^2 + O(h^3)\\
    u_y(z_0)+\frac{1}{6}(u_{xxy}(z_0)+u_{xyy}(z_0)+u_{yyy}(z_0))h^2 + O(h^3)
  \end{array}
\right),
\label{equ:regulartayor}
\end{equation}
which means \eqref{equ:regularstencil} is a second order finite difference stencil.

\textbf{Example 2}: In this example, we consider a uniform mesh with a chevron pattern. Similar to the previous example, the local numbering of nodal points is shown in Figure \ref{fig:gr_chevron}. The local coordinates are defined as
\begin{equation}
	\boldsymbol{\xi}^T=(0, 1, 0, -1, -1, 0, 1) \quad \text{and} \quad 
	\boldsymbol{\eta}^T = (0, 0, 1, 0, -1, -1, -1).
\end{equation}
Repeating the same procedure as in example 1,  we obtain that 
\begin{equation}\label{equ:chevronsprcasea}
G_hu(z_0) = \frac{1}{12h}
\left(
\begin{array}{c}
  6u_1-6u_3\\
 -2u_0+u_1+6u_2+u_3-u_4-4u_5-u_6
\end{array}
\right).
\end{equation}
Taylor expansion around  $z_0$  reveals that
\begin{equation}\label{equ:chevronsprtaylorcasea}
G_hu(z_0) = 
\left(
\begin{array}{l}
  u_x(z_0)+\frac{1}{6}u_{xxx}(z_0)h^2 + O(h^3)\\
    u_y(z_0)+\frac{1}{12}(2u_{yyy}(z_0)+u_{xxy}(z_0))h^2 + O(h^3)
  \end{array}
\right), 
\end{equation}
which again indicates that PPR provides a second-order approximation of $\nabla u$ at $z_0$.

\begin{remark}
On uniform meshes with a  regular pattern, the recovered gradients obtained via simple (weighted) averaging, superconvergent patch recovery, and polynomial preserving recovery methods are equivalent. They form a second-order finite difference scheme. However, on uniform meshes with a chevron pattern, the gradients obtained through simple (weighted) averaging and SPR exhibit only first-order convergence. This outcome explains the loss of superconvergence in SPR on certain mesh configurations. In other words, SPR is conditionally second-order consistent, whereas PPR is unconditionally second-order consistent.
\end{remark}

From the above two examples, we can see that the gradient recovery operator acts as a second-order finite difference scheme. For general unstructured meshes, we can show that the gradient recovery operator  $G_h$  has the following polynomial preserving property:
\begin{theorem}\label{thm:pp}
	The gradient recovery operator $G_h$ preserves polynomials of degree $k + 1$ for an arbitrary mesh. Furthermore, if $z_i$ is a center of symmetry for the involved nodal points and $k = 2r$, then $G_h$ preserves polynomials of degree up to $k + 2$ at $z_i$.
\end{theorem}

\begin{proof} (i). 
First, we consider the case where $z_i$ is a vertex. From the definition \eqref{equ:recoveredgrad}, it suffices to show that
	\begin{equation}\label{equ:ppidentity}
		\nabla p_{z_i}(z_i) = \nabla u(z_i),
	\end{equation}
	if $u$ is a polynomial of degree $k+1$. For simplicity, we consider the least-squares procedure on the computational domain $\Omega_{z_i}$ instead of the reference domain $\hat{\Omega}_{z_i}$. 
	Let $\{q_1(z), q_2(z), \ldots, q_{n_{k+1}}(z)  \}$ be the monomial basis functions for  $\mathbb{P}_{k+1}$ and 
	$$\mathbf{p}^T = (q_1(z), q_2(z), \ldots, q_{n_{k+1}}(z)).$$ The polynomial can be represented as 
	\begin{equation}
		\mathbf{p}_{z_i} = \mathbf{p}^T\mathbf{a},
	\end{equation}
	where the coefficients $\mathbf{a}$ are determined by the normal equation \eqref{equ:normalequation}. To establish the polynomial preserving property, it is sufficient to show that \eqref{equ:ppidentity} holds for all $q_j(z)$, $j = 1, \ldots, n_{k+1}$. For this purpose, let $u = q_j(z)$, which implies that
	\begin{equation}
		\mathbf{b}^T = \left( q_j(z_{i_0}), q_j(z_{i_1}), \ldots, q_j(z_{i_{|I_i|-1}})\right).
	\end{equation}
	Let $\mathbf{e}_j$ be the $j$th canonical basis of $\mathbb{R}^{n_{k+1}}$. It is straightforward to see that $\mathbf{e}_j$ is a solution to $\mathbf{A} \mathbf{a} = \mathbf{b}$. By the uniqueness of the least-squares problem, $\mathbf{e}_j$ must be the unique solution. Consequently, we have
\begin{equation*}
	p_{z_i} = \mathbf{p}^T\mathbf{e}_j = q_j,
\end{equation*} 
and hence we have
\begin{equation*}
	\nabla p_{z_i}(z_i) = \nabla u(z_i).
\end{equation*}
Thus, $G_h$ preserves polynomials of degree $k + 1$.
Second, if $z_i$ is an edge nodal point located on an edge formed by $z_{i_1}$ and $z_{i_2}$, the above proof shows that $p_{z_{i_1}}(z) = p_{z_{i_2}}(z) = u(z)$ if $u$ is a polynomial of degree $k+1$. By \eqref{equ:edgegr}, we have
\begin{equation*}
	\nabla p_{z_i}(z_i) = \nabla u(z_i).
\end{equation*}
Third, when $z_i$ is an interior nodal point, the same argument shows that $G_h$ preserves polynomials of degree $k + 1$. 
(ii). The proof can be found in \cite[Theorem 2.1]{zhang2005ppr}.
\end{proof}

Using the polynomial preserving property of $G_h$, we can establish the following unconditionally consistent result of $G_h$.
\begin{theorem}\label{thm:consistency}
	Suppose $u\in W^{k+2}_{\infty}(\Omega_{z_i})$, then we have 
	\begin{equation}
	\|\nabla u-G_hu\|_{0, \infty, \Omega_{z_i}} \lesssim h^{k+1}|u|_{k+2, \infty, \Omega_{z_i}}.
	\end{equation}
Furthermore, if $z_i$ is a center of symmetry for the involved nodal points and $u\in W^{k+3}_{\infty}(\Omega_{z_i})$ with  $k = 2r$, then we have 
\begin{equation}
	|(\nabla u-G_hu)(z_i)| \lesssim h^{k+2}|u|_{k+3, \infty, \Omega_{z_i}}.
	\end{equation}
\end{theorem}

\begin{proof}
This conclusion is a direct consequence of Theorem \ref{thm:pp} and the Bramble-Hilbert lemma \cite{brenner2008fembook, ciarlet2002fembook}.
\end{proof}

For the gradient recovery operator $G_h$, we can also demonstrate the boundedness property as follows:
\begin{theorem}[\cite{naga2004pprposteriori}]\label{thm:boundedness}
	For $K \in\mathcal{T}_h$, let $\Omega_{K} = \bigcup\limits_{z_j\in K \cap\mathcal{N}_h}\Omega_{z_j}$. Then, we have 
	\begin{equation}
		\|G_hv_h\|_{0, K} \lesssim \|\nabla v_h\|_{0, \Omega_{K}}. 
	\end{equation}
\end{theorem}

\subsection{Hessian recovery technique}\label{ssec:hr}
Let  $G_h$  denote the gradient recovery operator as defined in the preceding subsection. For any function $u_h\in S_h$, the recovered gradient  $G_h u_h$  is expressed in component form as:
\begin{equation}
	G_hu_h=
	\begin{pmatrix}
	 G^x_hu_h\\
	 G^y_hu_h 
\end{pmatrix}.
\end{equation}
The key idea of Hessian recovery is to apply the gradient recovery operator $G_h$ to $G_h^xu_h$ and $G_h^yu_h$. 
Thus, the recovered Hessian matrix of  $u_h$  is given by: 
\begin{equation}\label{equ:hessianrecovery}
	H_h u_h :=
	\begin{pmatrix}
		\left(G_h\left(G_h^x u_h\right)\right)^T \\ \left(G_h\left(G_h^y u_h\right)\right)^T
		\end{pmatrix}
	=
	\begin{pmatrix}
		G_h^x\left(G_h^x u_h\right) & G_h^x\left(G_h^y u_h\right) \\
G_h^y\left(G_h^x u_h\right) & G_h^y\left(G_h^y u_h\right)
	\end{pmatrix}. 
\end{equation}

Let us represent $u_h = \mathbf{u}^{T}\Phi$ as in \eqref{equ:vecform} and express the recovered Hessian matrix as 
\begin{equation}
	H_h u_h :=
	\begin{pmatrix}
		H_h^{xx}u_h & H_h^{xy}u_h \\
H_h^{yx}u_h & H_h^{yy}u_h
\end{pmatrix}
= 
\begin{pmatrix}
		 \mathbf{u}_{xx}^{T}\Phi&  \mathbf{u}_{xy}^{T}\Phi \\
\mathbf{u}_{yx}^{T}\Phi&  \mathbf{u}_{yy}^{T}\Phi
\end{pmatrix},
\end{equation}
where $\mathbf{u}_{xx}$, $\mathbf{u}_{xy}$, $\mathbf{u}_{yx}$, and $\mathbf{u}_{yy}$ are the vectors containing the recovered Hessian matrix values at the nodal points. Similar to gradient recovery case, we have the following relationship:
\begin{equation}
	\mathbf{u}_{xx} =  \mbox{B}_h^x \mbox{B}_h^x \mathbf{u} , \quad \mathbf{u}_{xy} =  \mbox{B}_h^x \mbox{B}_h^y \mathbf{u}, \quad 
	\mathbf{u}_{yx} =  \mbox{B}_h^y \mbox{B}_h^x \mathbf{u} , \quad \mathbf{u}_{yy} =  \mbox{B}_h^y \mbox{B}_h^y \mathbf{u}.
\end{equation}
It is basically just additional sparse matrix-vector multiplication, which can also be done in $\mathcal{O}(|\mathcal{N}_h|)$ operations. We can again see that the differentiation matrices $\mbox{B}_h^x$ and $\mbox{B}_h^y$ play the role of differentiation matrices in spectral methods \cite{shen2011spectralbook}. The second-order differentiation matrices are simply products of the first-order differentiation matrices.

\begin{figure}[!h]
   \centering
   {\includegraphics[width=0.451\textwidth]{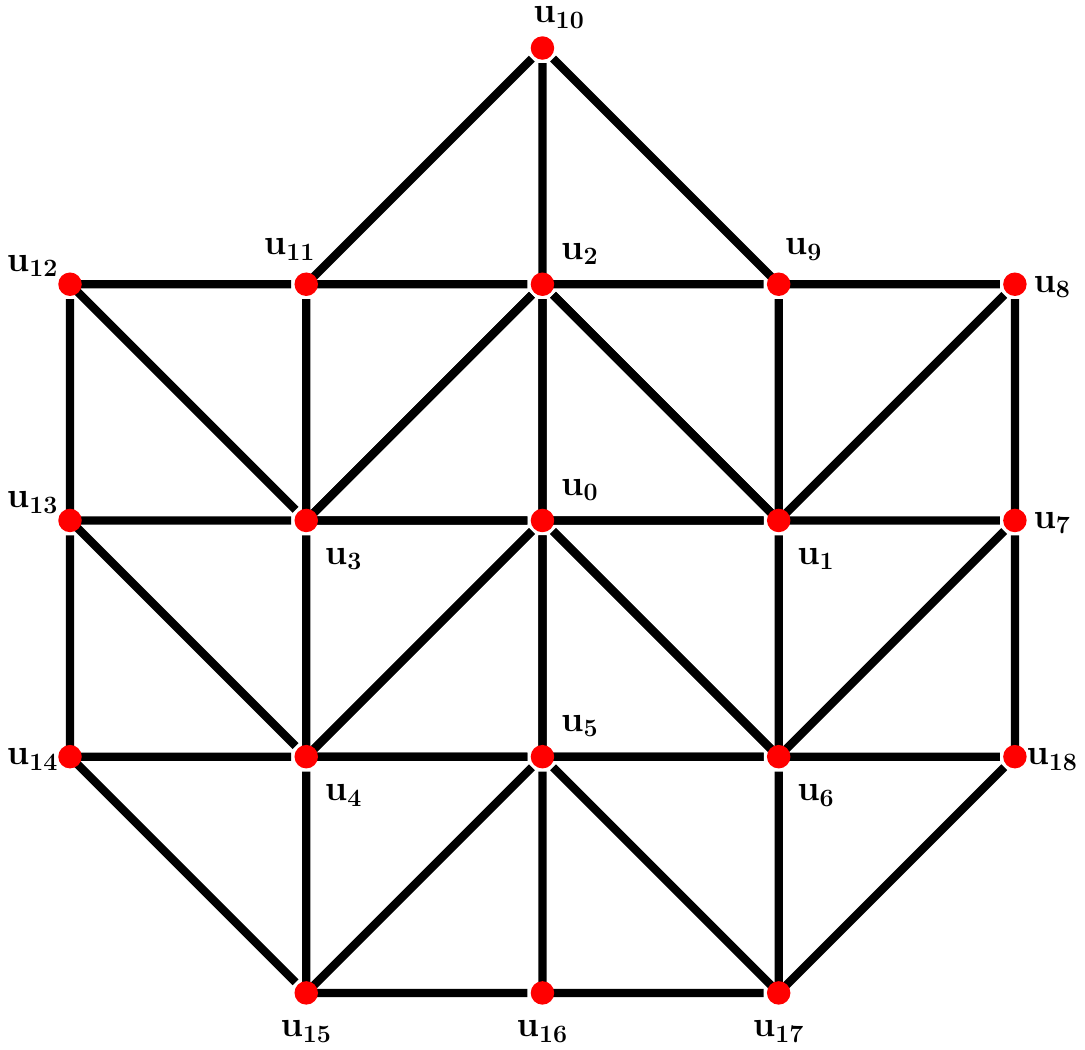}}
   \caption{Hessian recovery on uniform meshes with a chevron pattern.}\label{fig:hr}
\end{figure}

To demonstrate the idea, we consider the Hessian recovery on uniform meshes of chevron pattern using linear element.

\textbf{Example 3}: Consider the uniform mesh of chevron pattern as in Figure \ref{fig:hr}. From example 2, we know 
\begin{equation}\label{equ:ch}
G_hu(z_0) = \frac{1}{12h}
\left(
\begin{array}{c}
  6u_1-6u_3\\
 -2u_0+u_1+6u_2+u_3-u_4-4u_5-u_6
\end{array}
\right).
\end{equation}
By the definition of Hessian recovery \eqref{equ:hessianrecovery},  we have 
\begin{equation}\label{equ:chevron_hessian1}
	\begin{pmatrix}
		H_h^{xx}u \\
		H_h^{xy}u
	\end{pmatrix}(z_0)
	= 
	\frac{1}{12h}
	 \left( 6(G_hu)(z_1) - 6(G_hu)(z_3) \right),
\end{equation}
and 
\begin{equation}\label{equ:chevron_hessian2}
\begin{aligned}
\begin{pmatrix}
		H_h^{yx}u \\
		H_h^{yy}u
	\end{pmatrix}(z_0)
	= 
	\frac{1}{12h}
	 (&  -2(G_hu)(z_0) + (G_hu)(z_1) + 6(G_hu)(z_2) + (G_hu)(z_3)   \\
	 &-(G_hu)(z_4)-4(G_hu)(z_5)-(G_hu)(z_6) ) .
\end{aligned}
\end{equation}
Repeating a similar procedure as in Example 2, we obtain
\begin{equation*}
	\begin{aligned}
		& G_hu(z_1) = \frac{1}{12h}
\left(
\begin{array}{c}
  -6u_0+6u_7\\
 -u_0+2u_1+u_2-6u_6-u_7+u_8+4u_9
\end{array}
\right),\\
		& G_hu(z_2) = \frac{1}{12h}
\left(
\begin{array}{c}
  6u_9-6u_{11}\\
 -4u_0-u_1-2u_2-u_3+u_9+6u_{10}+u_{11}
\end{array}
\right),\\
		& G_hu(z_3) = \frac{1}{12h}
\left(
\begin{array}{c}
  6u_0-6u_{13}\\
 -u_0+u_2+2u_3-6u_4+4u_{11}+u_{12}-u_{13}
\end{array}
\right),\\
		& G_hu(z_4) = \frac{1}{12h}
\left(
\begin{array}{c}
  6u_5-6u_{14}\\
 u_0+4u_3+2u_4-u_5+u_{13}-u_{14}-6u_{15}
\end{array}
\right),\\
		& G_hu(z_5) = \frac{1}{12h}
\left(
\begin{array}{c}
  -6u_4+6u_6\\
 6u_0+u_4-2u_5+u_6-u_{15}-4u_{16}-u_{17}
\end{array}
\right),\\
		& G_hu(z_6) = \frac{1}{12h}
\left(
\begin{array}{c}
  -6u_5+6u_{18}\\
 u_0+4u_1-u_5+2u_6+u_7-6u_{17}-u_{18}
\end{array}
\right).\\
	\end{aligned}
\end{equation*}
Plugging these formulations of $(G_h u)(z_i)$ into \eqref{equ:chevron_hessian1} and \eqref{equ:chevron_hessian2}, we obtain
\begin{equation}
	\begin{aligned}
\left(H_h^{x x} u\right)\left(z_0\right)=&\frac{1}{144 h^2}\left(-72 u_0+36 u_{13}+36 u_7\right), \\
 \left(H_h^{x y} u\right)\left(z_0\right)=&\frac{1}{144 h^2}\left(-12 u_1+12 u_3+24 u_4-24 u_6+6 u_7\right. \\
& \left.+36 u_9-36 u_{11}-6 u_{13}+6 u_{14}-6 u_{18}\right), \\
 \left(H_h^{y x} u\right)\left(z_0\right)=&\frac{1}{144 h^2}\left(12 u_1-12 u_3+36 u_4-36 u_6-6 u_7\right. \\
& \left.+6 u_8+24 u_9-24 u_{11}-6 u_{12}+6 u_{13}\right), \\
 \left(H_h^{y y} u\right)\left(z_0\right)=&\frac{1}{144 h^2}\left(-48 u_0-10 u_1-22 u_2-10 u_3-10 u_4+18 u_5\right. \\
& -10 u_6-2 u_7+u_8+10 u_9+36 u_{10}+10 u_{11}+u_{12} \\
& \left.-2 u_{13}+u_{14}+10 u_{15}+16 u_{16}+10 u_{17}+u_{18}\right).
\end{aligned}
\end{equation}
Expanding using Taylor series, we can get 
\begin{equation*}
	\begin{aligned}
& \left(H_h^{x x} u\right)\left(z_0\right)=u_{x x}\left(z_0\right)+\frac{h^2}{3} u_{x x x x}\left(z_0\right)+\frac{2 h^4}{45} u_{x x x x x x}\left(z_0\right)+O\left(h^5\right), \\
& \left(H_h^{x y} u\right)\left(z_0\right)=u_{x y}\left(z_0\right)+\frac{h^2}{12}\left(3 u_{x x x y}\left(z_0\right)+2 u_{x y y y}\left(z_0\right)\right)-\frac{h^3}{24} u_{x x x y y}\left(z_0\right)+O\left(h^4\right), \\
& \left(H_h^{y x} u\right)\left(z_0\right)=u_{y x}\left(z_0\right)+\frac{h^2}{12}\left(3 u_{x x x y}\left(z_0\right)+2 u_{x y y y}\left(z_0\right)\right)+\frac{h^3}{24} u_{x x x y y}\left(z_0\right)+O\left(h^4\right), \\
& \left(H_h^{y y} u\right)\left(z_0\right)=u_{y y}\left(z_0\right)+\frac{h^2}{6}\left(u_{x x y y}\left(z_0\right)+2 u_{y y y y}\left(z_0\right)\right)-\frac{5 h^3}{72} u_{x x y y y}\left(z_0\right)+O\left(h^4\right) .
\end{aligned}
\end{equation*}
Just as in the gradient recovery case, we can see that the Hessian recovery operator $H_h$ is also a second-order finite difference operator, though $H_h^{xy} \neq H_h^{yx}$.

\begin{remark}
As we observed in the uniform mesh of chevron pattern case, the Hessian recovery operator is not generally symmetric. However, we can obtain a symmetric Hessian recovery operator by symmetrizing, i.e.,
\begin{equation*}
	H_h \leftarrow \frac{H_h + H_h^T}{2}.
\end{equation*}
This symmetrization process is straightforward to implement in practice and maintains the accuracy of the approximation.
\end{remark}

\begin{remark}
	In the literature, various other Hessian recovery methods have been proposed, such as double linear fitting, simple linear fitting, and double $L^2$ projection \cite{agouzal2022hessianrecovery}. For numerical comparisons of these methods, see \cite{picasso2003anisotropic, vallet2007hessianrecovery}. However, as reported in \cite{picasso2003anisotropic}, these methods fail to achieve convergence even on some uniform meshes.  
It is also worth noting that a Hessian recovery approach based on local least-squares fitting followed by second-order differentiation has been applied to solve the Cahn-Hilliard equation \cite{xu2004analysisofrecovery}. In particular, the Hessian recovery method reduces to the standard five-point stencil on a uniform mesh of  regular pattern.
\end{remark}

On general unstructured meshes, we can show the Hessian recovery operator $H_h$ satisfies the following polynomial preserving property.
\begin{theorem}\label{thm:hrpp}
The Hessian recovery operator $H_h$ preserves polynomials of degree $k + 1$ on arbitrary unstructured meshes.
\end{theorem}
\begin{proof}
	By Theorem \ref{thm:pp}, $G_h u = \nabla u$ if $u$ is a polynomial of degree $k+1$. Noting that $\nabla u$ is a polynomial of degree $k$, we obtain $G_h(G_h u) = G_h(\nabla u) = D^2 u$. This confirms that the conclusion holds.
\end{proof}

When the underlying mesh is translation invariant as defined in Section \ref{sec:not}, we can obtain an improved polynomial preserving result for $H_h$ \cite{guo2017hessianrecovery}:

\begin{theorem}\label{thm:hripp}
If $z_i$ is a node of a translation invariant mesh and serves as a symmetric center for the involved nodes, then $H_h$ preserves polynomials of degree $k + 2$ for odd $k$ and of degree $k + 3$ for even $k$. Additionally, $H_h$ is symmetric.
\end{theorem}

With the polynomial preserving properties of the Hessian recovery operator $H_h$, we can establish its consistency results as follows:

\begin{theorem}\label{thm:hessianconsistency}
	Let $u \in W_{\infty}^{k+2}(\Omega_{z_i})$; then
\begin{equation*}
    \|H u - H_h u\|_{0, \infty, \Omega_{z_i}} \lesssim h^{k} |u|_{k+2, \infty, \Omega_{z_i}}.
\end{equation*}
If $z_i$ is a node of a translation invariant mesh and serves as a symmetric center for the involved nodes, and $u \in W_{\infty}^{k+3}(\Omega_{z_i})$, then
\begin{equation*}
    |(H u - H_h u)(z_i)| \lesssim h^{k+1} |u|_{k+3, \infty, \Omega_{z_i}}.
\end{equation*}
Moreover, if $u \in W_{\infty}^{k+4}(\Omega_{z_i})$ and $k$ is an even number, then
\begin{equation*}
    |(H u - H_h u)(z_i)| \lesssim h^{k+2} |u|_{k+4, \infty, \Omega_{z_i}}.
\end{equation*}
\end{theorem}

\begin{proof}
	The results are a direct consequence of the polynomial preserving properties in Theorems \ref{thm:hrpp} and \ref{thm:hripp}, along with the Bramble-Hilbert Lemma \cite{brenner2008fembook, ciarlet2002fembook}.
\end{proof}

\subsection{Recovery techniques for nonstandard numerical methods}\label{ssec:rt}
The recovery techniques in the previous two subsections aim to obtain a globally continuous gradient or Hessian matrix. However, in practice, many problems lack global regularity in the gradient. One typical example is the interface problem, where the parameters of the partial differential equation are discontinuous across the domain, such as in a composite material consisting of two different materials with distinct physical properties. In this subsection, we consider the recovery techniques for interface problems.

In the setting of interface problem, we consider a domain $\Omega$ that is divided into two disjoint subdomains, $\Omega^-$ and $\Omega^+$, by a smooth curve $\Gamma$, which is typically described by the zero level set of a level set function \cite{osher2003levelsetbook, sethian1996levelsetbook}. Let $W^{k,p}(\Omega^- \cup \Omega^+)$ denote the function space consisting of piecewise Sobolev functions $w$ such that $w|_{\Omega^-} \in W^{k,p}(\Omega^-)$ and $w|_{\Omega^+} \in W^{k,p}(\Omega^+)$. The norm on this space is defined as
\begin{equation}\label{eq:pnorm}
\|w\|_{k,p, \Omega^- \cup \Omega^+} = \left( \|w\|_{k,p, \Omega^-}^p + \|w\|_{k,p, \Omega^+}^p \right)^{1/p},
\end{equation}
and the seminorm is defined as
\begin{equation}\label{eq:psnorm}
|w|_{k,p, \Omega^- \cup \Omega^+} = \left( |w|_{k,p, \Omega^-}^p + |w|_{k,p, \Omega^+}^p \right)^{1/p}.
\end{equation}

In this paper, we consider a unified framework that includes both body-fitted finite element methods \cite{babuska1970bodyfitted, chen1998bffem} and unfitted finite element methods, such as immersed finite element methods \cite{hou2005petrov, ji2014scifem, lin2015ppifem} and cut finite element methods \cite{hansbo2002nitsche, burman2015cutfem}. Let $\mathcal{T}_h$ be a shape-regular triangulation of $\Omega$, which can represent either a body-fitted or an unfitted mesh. Define the set of all elements covering the subdomain $\Omega^{\pm}$ as
\begin{equation}\label{equ:fictmesh}
	\mathcal{T}_{h}^{\pm} = \left\{K : K \cap \Omega^{\pm} \neq \emptyset \right\},
\end{equation}
and let
\begin{equation}\label{equ:fictdomain}
	\Omega_{h}^{\pm} = \bigcup_{K \in \mathcal{T}_{h}^{\pm}} K.
\end{equation}
For a body-fitted mesh, we have $\Omega_{h}^{\pm} = \Omega^{\pm}$; however, this does not hold for unfitted finite element methods. In that case, the domain $\Omega^{\pm}_h$ is called a fictitious subdomain.  Let $S_{h}^{\pm}$ be the continuous finite element space with piecewise polynomial degree $k$ on $\mathcal{T}_{h}^{\pm}$, defined as
\begin{equation}
	S_h^{\pm} = \{ v \in C^0(\Omega_{h}^{\pm}) : v|_{K} \in \mathbb{P}_k(K), \forall K\in \mathcal{T}_h^{\pm} \}.
\end{equation}
Let $\mathcal{N}_h^{\pm}$ denote the nodal points for $\mathcal{T}_h^{\pm}$, and $\phi_i^{\pm}$ denote the nodal basis for $S_h^{\pm}$.
For body-fitted finite element methods, we may also consider $S_h^{\pm}$ as the restriction of $S_h$ onto $\Omega^{\pm}$. In the case of cut finite element methods, the cut finite element space $S_h$ is defined as the direct product of $S_h^{\pm}$.

The key observation for interface problems is that their solutions are generally piecewise smooth on each subdomain, even though they may exhibit low global regularity. This motivates the development of a new gradient recovery method by applying the PPR gradient operator to each subdomain, as PPR is a local gradient recovery method.
To this end, let $G_h^{\pm}$ denote the gradient recovery operator defined on the (fictitious) subdomain $\Omega_{h}^{\pm}$. In this case, $G_h^{\pm}$ is a linear operator from $S_h^{\pm}$ to $S_h^{\pm} \times S_h^{\pm}$, with values at each nodal point obtained via local least squares fitting using sampling points located only within $\Omega_{h}^{\pm}$.  Let $u_h = (u_h^-, u_h^+) \in S_h^- \times S_h^+$. We define the recovered gradient of $u_h$ as  \cite{guo2018grbfem, guo2018grcutfem, guo2017grifem, guo2017superconvergentppifem}
\begin{equation}
	R_h u_h = \left(G_h^- u_h^-, G_h^+ u_h^+\right).
\end{equation}

Similar to the recovery techniques discussed in the previous two sections, the gradient recovery procedure can also be implemented using sparse matrices and vector multiplication. In this case, we define first-order differentiation matrices for each $G_h^{\pm}$.

By the polynomial preserving property of $G_h^{\pm}$, we can easily obtain the following piecewise polynomial preserving property:
\begin{theorem}\label{thm:ppp}
	The gradient recovery operator $R_h$ preserves piecewise polynomials defined on $\Omega^{\pm}$ of degree $k + 1$ for an arbitrary mesh.
\end{theorem}

\section{Superconvergence analysis of recovery techniques}\label{sec:super}
There are two frameworks for analyzing the superconvergence of recovery techniques. One approach involves superconvergence analysis using the supercloseness argument on mildly structured meshes \cite{lakhnay2000superconvergence, bank2003asymptotic1}, while the other relies on superconvergence analysis using difference quotients on translation invariant meshes \cite{wahlbin1995superconvergencebook, guo2017hessianrecovery}.

\subsection{Superconvergence analysis on mildly structured meshes}\label{ssec:mildlystructured}
Thought the recovery technique in section \ref{sec:ppr} is problem independent, to do superconvergence analysis, we shall  consider the following model problem in a general setting: find  $ u \in H^1(\Omega) $ such that
\begin{equation}\label{equ:elliptic}
    a(u, v) = \int_{\Omega} (\mathcal{D} \nabla u + \mathbf{b} u) \cdot \nabla v + c u v \, dx 
    =  f(v), \quad \forall v \in H^1(\Omega),
\end{equation}
we assume that the coefficient matrix $ \mathcal{D} $ is a symmetric, positive definite $ 2 \times 2 $ matrix, $ \mathbf{b} $ is a vector, $ c $ is a scalar constant, and $ f(\cdot) $ is a linear functional on $ H^1(\Omega) $. All coefficient functions are taken to be sufficiently smooth.

To ensure the well-posedness of \eqref{equ:elliptic}, we impose that the bilinear form $ a(\cdot, \cdot) $ satisfies a continuity condition:
\begin{equation}
    |a(u, v)| \lesssim \|u\|_{1, \Omega} \|v\|_{1, \Omega},
    \label{equ:continous}
\end{equation}
for all $ u, v \in H^1(\Omega) $. Additionally, we assume that the bilinear form $ a(\cdot, \cdot) $ fulfills the inf-sup conditions (cf. \cite{brenner2008fembook, ciarlet2002fembook})
\begin{equation}
    \inf_{u \in H^1(\Omega)} \sup_{v \in H^1(\Omega)} \frac{a(u, v)}{\|u\|_{1, \Omega} \|v\|_{1, \Omega}}
    =
    \sup_{u \in H^1(\Omega)} \inf_{v \in H^1(\Omega)} \frac{a(u, v)}{\|u\|_{1, \Omega} \|v\|_{1, \Omega}}
    \geq \mu > 0.
    \label{equ:infsup}
\end{equation}

The conforming finite element method for the model equation is to find $u_h\in S_h$ such that
\begin{equation}\label{equ:elliptic_fem}
	a(u_h, v_h) = f(v_h), \quad \forall v_h \in S_h. 
\end{equation}
To guarantee the existence of a unique solution to \eqref{equ:elliptic_fem}, we assume the inf-sup conditions\begin{equation}
  \inf_{u\in S_h}\sup_{v\in S_h}
    \frac{a(u, v)}{\|u\|_{1, \Omega}\|v\|_{1, \Omega}}
    =
  \sup_{u\in S_h}\inf_{v\in S_h}
    \frac{a(u, v)}{\|u\|_{1, \Omega}\|v\|_{1, \Omega}}
    \ge  \mu > 0.
  \label{equ:disinfsup}
\end{equation}
From \eqref{equ:elliptic} and \eqref{equ:elliptic_fem}, it follows that
\begin{equation}
  a(u-u_h, v) = 0,
  \label{equ:orth}
\end{equation}
for any $v \in S_h$. In particular, \eqref{equ:orth} holds for any
$v \in S^{\text{comp}}_h(\Omega)$.

\subsubsection{Linear element} In this subsubsection, we consider $S_h$ to be the continuous linear finite element space. Furthermore, we assume the mesh $\mathcal{T}_h$ satisfies the  Condition $(\sigma, \alpha)$ as defined in Definition \ref{def:cond}.

For the positive $\sigma$ and $\alpha$, Xu and Zhang \cite{xu2004analysisofrecovery} showed the following  supercloseness results:
\begin{theorem}\label{thm:linearsuperclose}
 Let $u$ be the solution of \eqref{equ:elliptic}, let $u_h\in S_h$ be the finite element
solution of \eqref{equ:elliptic_fem}, and let $u_I \in S_h$ be the linear interpolation of $u$ as in \eqref{equ:interp}. If the triangulation
$\mathcal{T}_h$ satisfies Condition $(\sigma,\alpha)$  and $u \in H^3(\Omega) \cap W^{2}_{\infty} (\Omega)$, then
\begin{equation}\label{equ:linearsuperclose}
	|u_h-u_I|_{1,\Omega} \lesssim  h^{1+\rho} (|u|_{3,\Omega} + |u|_{2,\infty,\Omega}),
\end{equation}
where $\rho = \min(\alpha,\sigma/2,1/2)$.
\end{theorem}

Building on the supercloseness result, we first present the superconvergence analysis for gradient recovery.

\begin{theorem}\label{thm:grsuperconvergence}
	Suppose the solution $u$ belongs to $ W^{3}_{\infty} (\Omega)$ and the mesh $\mathcal{T}_h$  satisfies the 
	Condition $(\sigma,\alpha)$ . Then, the following inequality holds
	 \begin{equation}
	 	\|\nabla u - G_h u_h\|_{0, \Omega} \lesssim h^{1+\rho}\|u\|_{3,\infty,\Omega},
	 \end{equation}
	 where $\rho = \min(\alpha,\sigma/2,1/2)$.
\end{theorem}
\begin{proof}
Using the triangle inequality, we obtain
	\begin{equation}
		\begin{aligned}
			\|\nabla u - G_h u_h\|_{0, \Omega} \le \|\nabla u- (\nabla u)_I\|_{0, \Omega} + \|(\nabla u)_I - G_h u_I \|_{0, \Omega} + \|G_h u_I - G_h u_h\|_{0, \Omega}. 
		\end{aligned}
	\end{equation}
By the standard approximation results \cite{brenner2008fembook, ciarlet2002fembook}, we have 
 \begin{equation}
   	 \|\nabla u - (\nabla u)_I \|_{0, \Omega} \lesssim   h^2 \|u\|_{3, \Omega}.
   \end{equation}
  For the second term, we have 
  \begin{equation}
  	\begin{aligned}
  		\|(\nabla u)_I - G_h u_I \|_{0, \Omega}
  		\lesssim &\|(\nabla u)_I - G_h u_I \|_{0, \infty, \Omega} \\
  		\le &\|\sum_{i = 1}^{| \mathcal{N}_h|}((\nabla  u)(z_i) - (G_hu)(z_i))\phi_{i}\|
    _{0, \infty, \Omega}\\
    \lesssim&
    \max_{z_i\in \mathcal{N}_h \cap \Omega} |(\nabla u)(z_i) - (G_hu)(z_i)|\\
    \lesssim&
    h^{2}|u|_{3, \infty, \Omega},
  	\end{aligned}
  \end{equation}
where we have used  Theorem \ref{thm:consistency} in the last inequality. 
To estimate the last term, we first apply the boundedness property of $ G_h $ from Theorem \ref{thm:boundedness}, yielding
\begin{equation*}
	\begin{aligned}
		\|G_hu_I - G_hu_h\|_{0, K} \lesssim  \| \nabla u_I - \nabla u_h \|_{0, \Omega_{K}}, \quad
		\forall K \in \mathcal{T}_h. 
	\end{aligned}
\end{equation*}
Thus,  we have 
\begin{equation}\label{equ:intermediate}
\begin{aligned}
	 \|G_h u_I - G_h u_h\|_{0, \Omega}^2 = & \sum_{K\in\mathcal{T}_h}  \|G_h u_I - G_h u_h\|_{0, K}^2 \\
	   \lesssim	 & \sum_{K\in\mathcal{T}_h }\| \nabla u_I - \nabla u_h \|_{0, \Omega_{K}}^2 \\
	    \lesssim & \|\nabla u_I - \nabla u_h\|_{0, \Omega}^2.
\end{aligned}
\end{equation}
Since the mesh $\mathcal{T}_h$ satisfies the Condition  $(\sigma,\alpha)$, it follows that 
\begin{equation}
	 \|\nabla u_I - \nabla u_h\|_{0, \Omega} \lesssim h^{1+\rho} (|u|_{3,\Omega} + |u|_{2,\infty,\Omega}),
\end{equation}
where $\rho = \min(\alpha,\sigma/2,1/2)$. The proof is concluded by combining all the above estimates.
\end{proof}

Next, we shall present the superconvergence results for Hessian recovery. 
\begin{theorem}\label{thm:hrlinearsuper}
	Suppose the mesh $\mathcal{T}_h$ is quasi-uniform and satisfies the Condition $(\sigma, \alpha)$. If $u \in W^3_{\infty}(\Omega)$, then the following estimate holds:
\begin{equation}
	\|H u - H_h u_h\|_{0, \Omega} \le h^{\rho} \|u\|_{3, \infty, \Omega}, 
\end{equation}
	 where $\rho = \min(\alpha,\sigma/2,1/2)$.
\end{theorem}
\begin{proof}
Using the triangle inequality and the definition of $H_h$, we obtain 
   \begin{align*}
       	\|H u - H_h u_h\|_{0, \Omega}&\le
	\|Hu-(Hu)_I\|_{0, \Omega} + \|(Hu)_I-H_hu_I\|_{0, \Omega}+\|H_h(u_I-u_h)\|_{0, \Omega}\\
	&=\|Hu-(Hu)_I\|_{0, \Omega} + \|(Hu)_I-H_hu_I\|_{0, \Omega}+\|G_h(G_h(u_I-u_h))\|_{0, \Omega}.
   \end{align*}
   By the standard approximation approximation theory \cite{brenner2008fembook, ciarlet2002fembook},  the first term is bounded by 
   \begin{equation}
   	\|Hu-(Hu)_I\|_{0, \Omega}  \lesssim h|u|_{3, \Omega}. 
   \end{equation}
  For the second term, we have 
  \begin{equation}
  	\begin{aligned}
  		\|(H u)_I - H_h u_I \|_{0, \Omega}
  		\lesssim &\|(H u)_I - H_h u_I \|_{0, \infty, \Omega} \\
  		\le &\|\sum_{i = 1}^{| \mathcal{N}_h|}((H  u)(z_i) - (H_hu)(z_i))\phi_{i}\|
    _{0, \infty, \Omega}\\
    \lesssim&
    \max_{z_i\in \mathcal{N}_h \cap \Omega} |(H u)(z_i) - (H_hu)(z_i)|\\
    \lesssim&
    h|u|_{3, \infty, \Omega},
  	\end{aligned}
  \end{equation}
   where we have used the  Theorem \ref{thm:hessianconsistency} in the last inequality.  
Using  \eqref{equ:intermediate}, we have 
\begin{equation*}
	\|H_h (u_I - u_h)\|_{0, \Omega} \lesssim \|\nabla (G_h (u_I - u_h))\|_{0, \Omega}.
\end{equation*}
Since $\mathcal{T}_h$ is quasi-uniform and $G_h (u_I - u_h) \in S_h \times S_h$, the inverse inequality \cite{brenner2008fembook, ciarlet2002fembook} implies 
   \begin{align*}
       \|H_h(u_I-u_h)\|_{0, \Omega}\lesssim h^{-1}\|G_h(u_I-u_h)\|_{0, \Omega}
       \lesssim h^{-1}\|\nabla(u_I-u_h)\|_{0, \Omega}
       \lesssim h^{\rho}|u|_{3, \infty, \Omega},
   \end{align*}
  where we have used the supercloseness result \eqref{equ:linearsuperclose}.
Combining the estimates above completes the proof.
\end{proof}

\subsubsection{Quadratic element}  In this subsubsection, we perform a superconvergence analysis for recovery techniques applied to   quadratic elements. Analogous to the linear case, certain assumptions on the mesh $\mathcal{T}_h$ are required. For quadratic elements, however, slightly stronger regularity conditions are necessary. We define a mesh $\mathcal{T}_h$ as strongly regular if any two adjacent triangles within $\mathcal{T}_h$ form an approximate parallelogram with an error of $O(h^2)$.

When the mesh $\mathcal{T}_h$ satisfies the stronger regularity, Huang and Xu \cite{huang2008quadraticsupercloseness} proved the following supercloseness results:
\begin{theorem}\label{thm:quadraticsuperclose}
Assume that the triangulation $\mathcal{T}_h$ is strongly regular. If the solution $u \in H^4(\Omega)$, then the following estimate holds:
    \begin{equation*}
	|u_h-u_I|_{1, \Omega}\lesssim h^3|u|_{4,  \Omega}.
    \end{equation*}
    \label{quadsuperclose}
\end{theorem}

Using the supercloseness results, we can show the superconvergence of the recovered gradient with quadratic elements.
\begin{theorem}\label{equ:grquadsuper}
	Suppose the solution $u$ belongs to $ W^{4}_{\infty} (\Omega)$ and the mesh $\mathcal{T}_h$ is  strongly regular. Then, the following inequality holds
	 \begin{equation}
	 	\|\nabla u - G_h u_h\|_{0, \Omega} \lesssim h^{3}\|u\|_{4,\infty,\Omega}.
	 \end{equation}
\end{theorem}
\begin{proof}
The proof follows similarly to that of Theorem \ref{thm:grsuperconvergence} by applying Theorem \ref{thm:quadraticsuperclose}.
\end{proof}

Similarly, we can establish the superconvergence results for the Hessian recovery operator $H_h$ when using quadratic elements.
\begin{theorem}\label{thm:hrquadsuper}
	Suppose the solution $u$ belongs to $ W^{4}_{\infty} (\Omega)$ and the mesh $\mathcal{T}_h$ is  quasi-uniform and strongly regular. Then, the following inequality holds
	 \begin{equation}
	 	\|Hu - H_h u_h\|_{0, \Omega} \lesssim h^{2}\|u\|_{4,\infty,\Omega}.
	 \end{equation}
\end{theorem}
\begin{proof}
The proof is analogous to that of Theorem \ref{thm:hrlinearsuper}, utilizing Theorem \ref{thm:quadraticsuperclose} and the inverse estimate.
\end{proof}

\begin{remark}
The supercloseness result in Theorem \ref{thm:quadraticsuperclose} requires the mesh $\mathcal{T}_h$ to be strongly regular. However, this condition on $\mathcal{T}_h$ can be relaxed by decomposing it into two submeshes, $\mathcal{T}_{1,h}$ and $\mathcal{T}_{2,h}$, where $\mathcal{T}_{1,h}$ satisfies the strongly regular condition, while the total area of the triangles in $\mathcal{T}_{2,h}$ is $\mathcal{O}(h^{\sigma})$. A weaker supercloseness result can then be derived in a similar manner, with the rate depending on $\sigma$ \cite{huang2008quadraticsupercloseness}.
\end{remark}

\begin{remark}
What we want to emphasize is that the approach of using supercloseness to establish superconvergence results for recovery techniques is restricted to linear and quadratic elements when simplicial meshes are employed in the two-dimensional case. Specifically, Li \cite{li2004counterexample} proved that Lagrange interpolation and its associated finite element methods are not superclose in the $H^1$ norm for $P_k$ (simplicial) elements in $d$ dimensions, where $d \ge 2$ and $k \ge d + 1$. In particular, we don't expect the supercloseness results in two-dimensional case if we use cubic or higher-order elements.
\end{remark}

\subsection{Superconvergence analysis on translation invariant meshes}
In this subsection, we introduce a framework for analyzing the superconvergence or ultraconvergence of recovery techniques applicable to arbitrarily high-order elements on translation-invariant meshes. The main analysis tool is superconvergence by a difference quotient \cite{wahlbin1995superconvergencebook, zhang2005ppr, guo2017hessianrecovery}.

\subsubsection{Superconvergence of gradient recovery techniques}
In this subsubsection, we first consider the superconvergence of the gradient recovery operator $G_h$. As demonstrated in the first two examples, the gradient recovery operator is a second-order finite difference scheme. The key observation is that it can be viewed as a difference quotient. Considering the uniform mesh with a regular pattern from Example 1, and substituting $z_0$ with $z$, the recovered gradient $(G_h^x u_h)(z)$ can be reformulated as
\begin{equation*}
	\left(G_h^x u\right)\left(z\right)=\frac{1}{6 h}\left(2 u_1+ u_2 - u_3 -2u_4- u_5+u_6\right) . 
\end{equation*}
Denote $\phi_j$ as the nodal shape functions at $z_j$ for $j = 0, \cdots, 6$.  Then, $(G_h^x u)(z_0)$ can be expressed as 
\begin{equation*}
	\begin{aligned}
		\left(G_h^x u\right)\left(z\right)\phi_0(x,y) = & \frac{1}{6 h}\left(2 u_1\phi_1(x+h, y) + u_2 \phi_1(x+h, y+h) - u_3\phi_1(x, y+h) \right. \\
		&\left. -2u_4\phi_1(x-h, y)- u_5\phi_1(x-h, y-h) + u_6\phi_1(x, y-h) \right).
	\end{aligned}
\end{equation*}
The translations are in the directions of $\ell_1=\pm(1,0)$, $\ell_2=\pm(0,1)$, and $\ell_3 = \pm(1,1)$. In terms of these translation directions, the recovered gradient in the $x$-direction can be represented as
\begin{equation}\label{equ:grdifferencequotient}
	\left(G_h^x u\right)\left(z\right)=\sum_{|\nu|\le M}\sum_{i=1}^{3}C^{(i)}_{\nu, h}u_h(z + \nu h \ell_i).
\end{equation}

Based on this observation, we can prove the following results.
\begin{theorem}\label{thm:gr_translationinvariant_super}
Suppose  all the coefficients in the bilinear operator $a(\cdot, \cdot)$ are constant.  Let $\Omega_1 \subset \subset \Omega$ be separated by $d = O(1)$, and let the finite element space $S_h$, which includes piecewise polynomials of degree $k$, be translation invariant in the directions required by the gradient recovery operator $G_h$ on $\Omega_1$. Additionally, let $u \in W^{k+2}_{\infty}(\Omega)$. Then, for any interior region $\Omega_0 \subset \subset \Omega_1$, we have the following error estimate:
\begin{equation}
    \|G_{h}(u - u_h)\|_{0, \infty, \Omega_0} \lesssim \left(\ln\frac{1}{h}\right)^{\bar{r}} h^{k+1} \|u\|_{k+2, \infty, \Omega} + \|u - u_h\|_{-s, q, \Omega},
    \label{equ:error}
\end{equation}
for some $s \ge 0$ and $q \ge 1$. Here, $\bar{r} = 1$ for linear elements and $\bar{r} = 0$ for higher-order elements.
\end{theorem}
\begin{proof}
By the definition of the translation operator in \eqref{equ:translate}, we have
	\begin{equation}
	 T^{\ell}_{\tau h} u_h = 	u_h(z + \nu h \ell).
	\end{equation}
Since the coefficients in the bilinear form $a(\cdot, \cdot)$ are constant, this implies that
\begin{equation*}
	a( T^{\ell}_{\tau h}(u-u_h), v) = a(u-u_h,  T^{\ell}_{-\tau h}v) =  a(u-u_h,  (T^{\ell}_{\tau h})^*v).
\end{equation*}
As in \eqref{equ:grdifferencequotient}, $G_h^x$ is a linear combination of the translation of difference quotients. Then, we have 
\begin{equation*}
	a(G_h^x(u-u_h), v) = a(u-u_h, (G_h^x)^*v) = 0, \quad \forall  v \in S_{h}^{\text{comp}}(\Omega_1),
\end{equation*}
where we have used the fact that $(G_h^x)^*v  \in S_{h}^{\text{comp}}(\Omega_2)$ to get the last equality. 
Therefore, by \cite[Theorem 5.5.2]{wahlbin1995superconvergencebook} with $F \equiv 0$, we can deduce that
\begin{equation}
	\begin{aligned}
		\|G_h^x(u-u_h)\|_{0, \infty, \Omega_0} \lesssim & \left(\ln d/h \right)^{\bar{r}}\min_{v\in S_h}
		\|G_h^xu - v\|_{0, \infty,  \Omega_1} \\ 
		& +  d^{-s-2/q}\|G_h^x(u-u_h)\|_{-s, q, \Omega_1},
	\end{aligned}
\end{equation}
where $\bar{r} = 1$ for linear elements and $\bar{r} = 0$ for higher-order elements. By the standard approximation theory \cite{brenner2008fembook, ciarlet2002fembook}, the first term in the previous equation can be bounded as
\begin{equation*}
	\min_{v\in S_h}
		\|G_h^xu - v\|_{0, \infty,  \Omega_1} \lesssim h^{k+1}|u|_{k+2, \infty, \Omega_1}. 
\end{equation*}
By the definition of the negative norm, we have
\begin{equation*}
	\|G_h^x(u-u_h)\|_{-s, q, \Omega_1} = \sup_{\phi \in C^{\infty}_0(\Omega_1), \| \phi \|_{s, q', \Omega_1}=1}
	(G_h^x(u-u_h), \phi),
\end{equation*}
where $1/q + 1/q' = 1$. Using the fact that $G_h$ is a linear combination of translations, we have
\begin{equation*}
	\begin{aligned}
		((G_h^x(u-u_h), \phi) & =  (u-u_h, (G_h^x)^*\phi) 
		   \lesssim \|u-u_h\|_{0, \infty, \Omega_1+Mh}\|(G_h^x)^*\phi) \|_{0, 1, \Omega_1+Mh} 
		  \lesssim \|u-u_h\|_{0, \infty, \Omega_1+Mh},
	\end{aligned}
\end{equation*}
where $\Omega_1+Mh$ is the subdomain extending $Mh$ from $\Omega_1$,  and we have used the fact $\|(G_h^x)^*\phi) \|_{0, 1, \Omega_1+Mh}$ is bounded by a constant independent of $h$. 
Applying Theorem 5.5.2 in \cite{wahlbin1995superconvergencebook} again, we have 
\begin{equation*}
	\begin{aligned}
		\|u-u_h\|_{0, \infty, \Omega_1+Mh} &\lesssim \left(\ln d/h \right)^{\bar{r}}\min_{v\in S_h}
		\|u - v\|_{0, \infty,  \Omega} 
		 +  d^{-s-2/q}\|u-u_h\|_{-s, q, \Omega} \\
		 &\lesssim h^{k+1} \left(\ln d/h \right)^{\bar{r}}
		\|u \|_{k+1, \infty,  \Omega} 
		 +  d^{-s-2/q}\|u-u_h\|_{-s, q, \Omega}.
	\end{aligned}
\end{equation*}
If the separation parameter $d = \mathcal{O}(1)$, combining the above estimates gives
\begin{equation}
	\|G_h^x(u-u_h)\|_{0, \infty, \Omega_0}\lesssim h^{k+1} \left(\ln 1/h \right)^{\bar{r}}
		\|u \|_{k+2, \infty,  \Omega} 
		 +  \|u-u_h\|_{-s, q, \Omega}.
\end{equation}
Repeating the same process, we obtain a similar estimate for $G_h^y$. Therefore, our proof is completed by replacing $G^x_h$ with $G_h$.
\end{proof}

With the above preparation, we now present our main superconvergence results for the gradient recovery operator $G_h$:
\begin{theorem}\label{thm:gr_super_final}
Suppose  all the coefficients in the bilinear operator $a(\cdot, \cdot)$ are constant. Let $\Omega_1 \subset \subset \Omega$ be separated by $d = O(1)$, and let the finite element space $S_h$, which includes piecewise polynomials of degree $k$, be translation invariant in the directions required by the gradient recovery operator $G_h$ on $\Omega_1$. Additionally, let $u \in W^{k+2}_{\infty}(\Omega)$. Then, for any interior region $\Omega_0 \subset \subset \Omega_1$, we have the following error estimate:
  \begin{equation}
   \|\nabla u - G_{h}u_h\|_{0, \infty, \Omega_0}
  \lesssim \left(\ln\frac{1}{h}\right)^{\bar{r}}
  h^{k+1}\|u\|_{k+2,\infty,\Omega}
  + \|u-u_h\|_{-s, q, \Omega},
    \label{equ:main}
  \end{equation}
  for some $s\ge 0$ and $q\ge 1$.
  \label{thm:ultra1}
\end{theorem}
\begin{proof}
	We can decompose $\nabla u - G_hu_h$ as 
	\begin{equation*}
		\nabla u - G_hu_h = (\nabla u - (\nabla u)_I) + ((\nabla u)_I - G_hu) + (G_hu - G_hu_h),
	\end{equation*}
	where $ (\nabla u)_I \in S_h \times S_h$ is the interpolation of $\nabla u$ in $S_h$.  By the standard interpolation theory \cite{brenner2008fembook, ciarlet2002fembook}, we have 
	\begin{equation*}
		\|(\nabla u - (\nabla u)_I\|_{0,\infty,\Omega}\lesssim h^{k+1}|\nabla u|_{k+1, \infty, \Omega} \lesssim h^{k+1}|\nabla u|_{k+2, \infty, \Omega}.
	\end{equation*}
	For the second term, we have
	  \begin{equation}
    \begin{split}
    \|(\nabla u)_I - G_h u\|_{0, \infty, \Omega_0}=&
    \|\sum_{i = 1}^{| \mathcal{N}_h|}((\nabla  u)(z_i) - (G_hu)(z_i))\phi_{i}\|
    _{0, \infty, \Omega_0}\\
    \lesssim&
    \max_{z_i\in \mathcal{N}_h \cap \Omega_0} |(\nabla u)(z_i) - (G_hu)(z_i)|\\
    \lesssim&
    h^{k+1}|u|_{k+2, \infty, \Omega},
    \end{split}
    \label{equ:mainsecond}
  \end{equation}
  where we have used Theorem \ref{thm:consistency} in the last inequality.  Combing the above estimates and Theorem \ref{thm:gr_translationinvariant_super}, we finish the proof. 
\end{proof}

\begin{remark}
Theorem \ref{thm:gr_super_final} provides a superconvergence result under the condition
\[
\|u-u_h\|_{-s, q, \Omega} \lesssim h^{k+\sigma}, \quad \sigma > 0.
\]
For negative norm estimates, the reader is referred to \cite{nitsche1974negativenorm}.
\end{remark}

\begin{remark}
	For the SPR, Zienkiewicz and Zhu \cite{zhu1992spr2} demonstrated that SPR achieves $\mathcal{O}(h^4)$ gradient recovery at mesh vertices for a uniform mesh of regular patterns. However, SPR only provides $\mathcal{O}(h^2)$ recovery at edge centers and fails to achieve superconvergence for quadratic elements over the entire patch. Numerical results in \cite{zhang2005ppr} indicate that PPR achieves $\mathcal{O}(h^4)$ gradient recovery at both vertices and edge centers. By applying quadratic interpolation at vertices and edge centers, this approach guarantees $\mathcal{O}(h^3)$ gradient recovery.
\end{remark}

\subsubsection{Ultraconvergence of Hessian recovery techniques}
In this subsubsection, we present the ultraconvergence analysis of the Hessian recovery operator $H_h$ on translation invariant meshes. Similar to the gradient recovery case, the Hessian recovery operator $H_h$ can also be interpreted as a difference quotient, as demonstrated in Example 3. Analogously to the gradient recovery case, the Hessian recovery operator $H_h$ can be represented as 
\begin{equation}
	(H^{ab}_h u_h)(z) = \sum_{|\nu| < M} \sum_{i=1}^6 C_{\nu, h}^i u_h\left(z + \nu h \ell_i\right) = \sum_{|\nu| < M} \sum_{i=1}^6 C_{\nu, h}^i T_{\nu \tau}^{\ell_i} u_h(z),
\end{equation}
where $a, b \in \{x, y\}$, and $\ell_i$ are given directions.

Based on the representation in terms of difference quotients, we first establish the following result for the Hessian recovery operator $H_h$:
\begin{theorem}\label{thm:hr_tr_analysis}
Suppose  all the coefficients in the bilinear operator $a(\cdot, \cdot)$ are constant. Let $\Omega_1 \subset\subset \Omega$ denote a subdomain separated from the boundary by a distance $d = O(1)$. Assume that the finite element space $S_h$, consisting of piecewise polynomials of degree $k$, is translation invariant along the directions necessary for the Hessian recovery operator $H_h$ to function within $\Omega_1$. If $u \in W^{k+3}_{\infty}(\Omega)$, then, on any interior region $\Omega_0 \subset\subset \Omega_1$, we have
\begin{equation}
  \|H_{h}(u-u_h)\|_{0, \infty, \Omega_0}
  \lesssim \left(\ln\frac{1}{h}\right)^{\bar{r}}
  h^{k+1}\|u\|_{k+3, \infty, \Omega}
  + \|u-u_h\|_{-s, q, \Omega},
  \label{equ:error1}
\end{equation}
for some $s \ge 0$ and $q \ge 1$. Here, $\bar{r} = 1$ for linear elements, and $\bar{r} = 0$ for higher-order elements.
\end{theorem}

\begin{proof}
	Since the bilinear form $a(\cdot, \cdot)$ is constant coefficient,  it follows that 
	\begin{equation*}
			a( T^{\ell}_{\tau h}(u-u_h), v) = a(u-u_h,  T^{\ell}_{-\tau h}v) =  a(u-u_h,  (T^{\ell}_{\tau h})^*v).
	\end{equation*}
	Note that $(T^{\ell}_{\tau h})^*v$ is in the finite element space by the definition and $H_h$ is a linear combination of $T^{\ell}_{\tau h}$. Then, we have 
	\begin{equation*}
	a(H_h^{xx}(u-u_h), v) = a(u-u_h, (H_h^{xx})^*v) = 0, \quad \forall  v \in S_{h}^{\text{comp}}(\Omega_1),
\end{equation*}
where the last equality is implied by Galerkin orthogonality. By \cite[Theorem 5.5.2]{wahlbin1995superconvergencebook} with $F \equiv 0$, we can deduce that
\begin{equation}
	\begin{aligned}
		\|H_h^{xx}(u-u_h)\|_{0, \infty, \Omega_0} \lesssim & \left(\ln d/h \right)^{\bar{r}}\min_{v\in S_h}
		\|H_h^{xx}u - v\|_{0, \infty,  \Omega_1} 
		 +  d^{-s-2/q}\|H_h^{xx}(u-u_h)\|_{-s, q, \Omega_1},
	\end{aligned}
\end{equation}
where $\bar{r} = 1$ for linear elements and $\bar{r} = 0$ for higher-order elements. By the standard approximation theory \cite{brenner2008fembook, ciarlet2002fembook}, the first term in the previous equation can be bounded as
\begin{equation*}
	\min_{v\in S_h}
		\|H_h^{xx}u - v\|_{0, \infty,  \Omega_1} \lesssim h^{k+1}|u|_{k+3, \infty, \Omega_1}. 
\end{equation*}
By the definition of the negative norm, we have
\begin{equation*}
	\|H_h^{xx}(u-u_h)\|_{-s, q, \Omega_1} = \sup_{\phi \in C^{\infty}_0(\Omega_1), \| \phi \|_{s, q', \Omega_1}=1}
	(H_h^{xx}(u-u_h), \phi),
\end{equation*}
where $1/q + 1/q’ = 1$. Using the same argument as gradient recovery case, we have 
\begin{equation*}
	\begin{aligned}
		&((H_h^{xx}(u-u_h), \phi)  =  (u-u_h, (H_h^{xx})^*\phi) \\
		   \lesssim  &\|u-u_h\|_{0, \infty, \Omega_1+Mh}\|(H_h^{xx})^*\phi) \|_{0, 1, \Omega_1+Mh} 
		  \lesssim \|u-u_h\|_{0, \infty, \Omega_1+Mh},
	\end{aligned}
\end{equation*}
where $\Omega_1+Mh$ is the subdomain extending $Mh$ from $\Omega_1$,  and we have used the fact $\|(H_h^{xx})^*\phi) \|_{0, 1, \Omega+Mh}$ is bounded by a constant independent of $h$. 
Applying Theorem 5.5.2 in \cite{wahlbin1995superconvergencebook} again, we have 
\begin{equation*}
	\begin{aligned}
		\|u-u_h\|_{0, \infty, \Omega_1+Mh} &\lesssim \left(\ln d/h \right)^{\bar{r}}\min_{v\in S_h}
		\|u - v\|_{0, \infty,  \Omega} 
		 +  d^{-s-2/q}\|u-u_h\|_{-s, q, \Omega} \\
		 &\lesssim h^{k+1} \left(\ln d/h \right)^{\bar{r}}
		\|u\|_{k+1, \infty,  \Omega} 
		 +  d^{-s-2/q}\|u-u_h\|_{-s, q, \Omega}.
	\end{aligned}
\end{equation*}
If the separation parameter $d = \mathcal{O}(1)$, combining the above estimates gives
\begin{equation}
	\|H_h^{xx}(u-u_h)\|_{0, \infty, \Omega_0}\lesssim h^{k+1}\left(\ln 1/h \right)^{\bar{r}}
		\|u\|_{k+3, \infty,  \Omega} 
		 +  \|u-u_h\|_{-s, q, \Omega}.
\end{equation}
Similarly, we can get the same estimate for $H_h^{xy}$, $H_h^{yx}$, and $H_h^{yy}$. Therefore, our proof is completed by replacing $H_h^{xx}$ with $H_h$.
\end{proof}

We are now in a perfect position to present the main ultraconvergence result for the Hessian recovery operator $H_h$.

\begin{theorem}\label{thm:hr_ultraconvergence}
Suppose  all the coefficients in the bilinear operator $a(\cdot, \cdot)$ are constant. Let $ \Omega_1 \subset\subset \Omega $ denote a subdomain separated from the boundary by a distance $ d = O(1) $. Assume that the finite element space $ S_h $, composed of piecewise polynomials of degree $ k $, is translation invariant along the directions necessary for the Hessian recovery operator $ H_h $ to operate effectively within $ \Omega_1 $. If $u \in W^{k+3}_{\infty}(\Omega)$, then, on any interior region $\Omega_0 \subset\subset \Omega_1$, we have
\begin{equation}
   \|Hu - H_{h}u_h\|_{0,\infty,\Omega_0}
  \lesssim  \left(\ln\frac{1}{h}\right)^{\bar{r}}
  h^{k+1}\|u\|_{k+3,\infty,\Omega}
  + \|u-u_h\|_{-s, q, \Omega},
    \label{equ:main1}
\end{equation}
for certain parameters $ s \geq 0 $ and $ q \geq 1 $.
  \label{thm:ultra}
\end{theorem}

\begin{proof}
	Analogue to gradient recovery case, we decompose $Hu - H_hu_h$ as 
	\begin{equation*}
		H u - H_hu_h = (H u - (Hu)_I) + ((Hu)_I - H_hu) + (H_hu - H_hu_h),
	\end{equation*}
	where $ (Hu)_I \in S_h^4$ is the interpolation of $\nabla u$ in $S_h$.  Using  standard interpolation theory \cite{brenner2008fembook, ciarlet2002fembook}, we can obtain 
	\begin{equation*}
		\|Hu - (H u)_I\|_{0,\infty,\Omega}\lesssim h^{k+1}|Hu|_{k+1, \infty, \Omega} \lesssim h^{k+1}| u|_{k+3, \infty, \Omega}.
	\end{equation*}
	For the second term, we have
	  \begin{equation}
    \begin{split}
    \|(Hu)_I - H_h u\|_{0, \infty, \Omega_0}=&
    \|\sum_{i = 1}^{| \mathcal{N}_h|}((H  u)(z_i) - (H_hu)(z_i))\phi_{i}\|
    _{0, \infty, \Omega_0}\\
    \lesssim&
    \max_{z_i\in \mathcal{N}_h \cap \Omega_0} |(H u)(z_i) - (H_hu)(z_i)|\\
    \lesssim&
    h^{k+1}|u|_{k+3, \infty, \Omega},
    \end{split}
    \label{equ:mainsecond3}
  \end{equation}
  where we have used Theorem \ref{thm:hessianconsistency} in the last inequality.  Combing the above estimates and Theorem \ref{thm:hr_tr_analysis}, we conclude the proof.
\end{proof}

\begin{remark}
Theorem \ref{thm:hr_ultraconvergence} presents an ultraconvergence result under the condition
\[
\|u-u_h\|_{-s, q, \Omega} \lesssim h^{k+\sigma}, \quad \sigma > 0.
\]
\end{remark}

\begin{remark}
For a general $k$th-order element, the piecewise Hessian matrix achieves a convergence rate of $\mathcal{O}(h^{k-1})$. Theorems \ref{thm:hrlinearsuper} and \ref{thm:hrquadsuper} demonstrate that the recovered Hessian matrix $H_h u_h$ is superconvergent to the exact Hessian matrix $H u$ at a rate of $\mathcal{O}(h^k)$ on mildly structured meshes. When the mesh is translation invariant, Theorem \ref{thm:hr_ultraconvergence} shows that a convergence rate of $\mathcal{O}(h^{k+1})$ can be achieved. This result constitutes ultraconvergence, that is, two orders higher than the optimal convergence rate.
\end{remark}

\subsection{Superconvergence analysis of finite element methods for interface problems}
In this subsection, we develop the superconvergence theory for gradient recovery methods applied to interface problems. For clarity and conciseness, we restrict our analysis to the use of linear elements. The governing model equation for this subsection is given by
\begin{subequations}\label{equ:interface}
 \begin{align}
  -\nabla \cdot (\beta(z) \nabla u(z)) &= f(z),  \quad \text{in } \Omega^-\cup\Omega^+, \label{equ:interface_equ}\\
   u & = 0, \quad\quad\,\,  \text{on } \partial\Omega, \label{equ:interface_bnd}\\
      \llbracket u\rrbracket &=0,  \quad\quad\,\,  \text{on } \Gamma, \label{equ:interface_valuejump}\\
        \llbracket \beta \partial_n u \rrbracket&= g,   \quad\quad\,\,  \text{on } \Gamma, \label{equ:interface_fluxjump}
\end{align}
\end{subequations}
where $\partial_n u = (\nabla u) \cdot n$, and $ n $ denotes the unit outward normal vector along the interface $\Gamma$. The jump operator $\llbracket w \rrbracket$ along $\Gamma$ is defined by
\begin{equation}\label{eq:jump}
\llbracket w\rrbracket = w^+ - w^-,
\end{equation}
where $ w^{\pm} = w|_{\Omega^{\pm}} $ represents the restriction of $w$ to $\Omega^{\pm}$.
The diffusion coefficient $ \beta(z) \ge \beta_0  > 0$ is a piecewise smooth function expressed as
\begin{equation}
\beta(z) =
\left\{
\begin{array}{ccc}
    \beta^-(z), &  \text{if } z = (x, y) \in \Omega^-, \\
   \beta^+(z),  &   \text{if } z = (x, y) \in \Omega^+,\\
\end{array}
\right.
\end{equation}
exhibiting a finite discontinuity across the interface $ \Gamma $.

The variational formulation of  \eqref{equ:interface_equ}--\eqref{equ:interface_fluxjump} is to find $u \in H^1_0(\Omega)$ such that
\begin{equation}\label{equ:interface_model}
(\beta\nabla u, \nabla v) = (f, v) - \langle g, v \rangle, \quad \forall v \in H^1_0(\Omega),
\end{equation}
where $(\cdot, \cdot)$ and $\langle \cdot, \cdot \rangle$ denote the standard $L_2$-inner products in the spaces $L^2(\Omega)$ and $L^2(\Gamma)$, respectively.  By the positivity of the coefficients, the variational problem \eqref{equ:interface_model} admits a unique solution.

\subsubsection{Superconvergence of body-fitted finite element methods}
Let $\mathcal{T}_h$ denote a body-fitted triangulation of $\Omega$, where each triangle $K \in \mathcal{T}_h$ falls into one of the following
three categories:
\begin{enumerate}
\item $K \subset \overline{\Omega^-}$;
\item $K \subset \overline{\Omega^+}$;
\item $K \cap \Omega^- \neq \emptyset$ and $K \cap \Omega^+ \neq \emptyset$, in which case two vertices of $K$ lie on $\Gamma$.
\end{enumerate}

Let $S_h$ be the continuous finite element space defined on $\mathcal{T}_h$, and let $S_{h,0}$ be the subspace of $S_h$ with homogeneous boundary conditions. The body-fitted finite element approximation of the variational problem \eqref{equ:interface_model} is to find $u_h \in S_{h,0}$ such that
\begin{equation}\label{equ:interface_bffem}
(\beta \nabla u_h, \nabla v_h) = (f, v_h)  - \langle g, v_h \rangle_{\Gamma}, \quad \forall v_h \in S_{h,0}.
\end{equation}
As in Subsection \ref{ssec:rt}, the solution $u_h$ can be considered as comprising two parts: $u_h^- \in S_h^-$ and $u_h^+ \in S_h^+$. 

For the body-fitted finite element method, \cite{guo2018grbfem} proves the following supercloseness results:
\begin{theorem}\label{thm:bffem_supercloseness}
Suppose the body-fitted mesh satisfies the Condition $(\sigma, \alpha)$. If $u \in H^3(\Omega^- \cup \Omega^+) \cap W^{2, \infty}(\Omega^- \cup \Omega^+)$, then for any $v_h \in S_{h, 0}$, the following estimate holds:
\begin{equation}\label{equ:superclose}
 \begin{split}
 \|\nabla(u_I - u_h)\|_{0, \Omega} \lesssim &\, h^{1 + \rho} \left( \|u\|_{3, \Omega^- \cup \Omega^+} + \|u\|_{2, \infty, \Omega^- \cup \Omega^+} \right) 
  + h^{\frac{3}{2}} \left( \|u\|_{2, \infty, \Omega^- \cup \Omega^+} + \|g\|_{0, \infty, \Gamma} \right),
\end{split}
\end{equation}
where $\rho = \min(\alpha, \frac{\sigma}{2}, \frac{1}{2})$ and $u_I \in S_h$ denotes the interpolation of $u$.
\end{theorem}

\begin{remark}
	When the jump function $g\equiv 0$, we can prove an improved result \cite{guo2018grbfem}
	  \begin{equation}\label{equ:improvedsuperclose}
 \begin{split}
 \|\nabla(u_I-u_h)\|_{0, \Omega} \lesssim& h^{1+\rho}(\|u\|_{3, \Omega^-\cup\Omega^+}
+ \|u\|_{2, \infty,  \Omega^-\cup\Omega^+}).
\end{split}
\end{equation} 
This result is the same as the supercloeseness in Theorem \ref{thm:linearsuperclose}. 
\end{remark}

Based on the supercloseness results, we can establish the superconvergence theory for the gradient recovery operator $R_h$. 
\begin{theorem}\label{thm:bffem_superconvergence}
Suppose the body-fitted mesh satisfies the Condition $(\sigma, \alpha)$. If $u \in H^3(\Omega^- \cup \Omega^+) \cap W^{2, \infty}(\Omega^- \cup \Omega^+)$,  then we have
 \begin{equation}\label{equ:super}
 \begin{split}
 \|\nabla u-R_hu_h \|_{0, \Omega} \lesssim& h^{1+\rho}\|u\|_{3, \infty, \Omega^-\cup\Omega^+} + 
h^{\frac{3}{2}}(\|u\|_{2, \infty,  \Omega^-\cup\Omega^+}+\|g\|_{0, \infty, \Gamma}),
\end{split}
\end{equation}
where $\rho = \min(\alpha, \frac{\sigma}{2}, \frac{1}{2})$.
\end{theorem}

\begin{proof}
We decompose $\nabla u - R_h u_h$ as
	\begin{equation*}
		\nabla u - R_hu_h = (\nabla u - (\nabla u)_I) +  ((\nabla u)_I - R_hu_I) + (R_hu_I - R_hu_h). 
	\end{equation*}
Using the triangle inequality, we obtain
	\begin{equation*}
		\begin{aligned}
			&\|\nabla u - R_hu_h\|_{0, \Omega^-\cup\Omega^+}\\
			 \lesssim & \|\nabla u - (\nabla u)_I\|_{0, \Omega^-\cup\Omega^+} +  \|(\nabla u)_I - R_hu_I\|_{0, \Omega^-\cup\Omega^+} + 
			  \|R_hu_I - R_hu_h\|_{0, \Omega^-\cup\Omega^+} \\
			  = & \|\nabla u - (\nabla u)_I\|_{0, \Omega^-\cup\Omega^+} + \\
			  & \|(\nabla u)_I - R_hu_I\|_{0, \Omega^-\cup\Omega^+}  + \|G_h^-u_I^- - G_h^-u_h^-\|_{0, \Omega^-} 
			+ \|G_h^+u_I^+ - G_h^+u_h^+\|_{0, \Omega^+}  \\	
			\lesssim & \|\nabla u - (\nabla u)_I\|_{0, \Omega^-\cup\Omega^+} +  \|(\nabla u)_I - R_hu_I\|_{0, \Omega^-\cup\Omega^+}  + \\
			& \|\nabla u_I^- - \nabla u_h^-\|_{0, \Omega^-} 
			+ \|\nabla u_I^+ - \nabla u_h^+\|_{0, \Omega^+} \\
				\lesssim & \|\nabla u - (\nabla u)_I\|_{0, \Omega^-\cup\Omega^+} +  \|(\nabla u)_I - R_hu_I\|_{0, \Omega^-\cup\Omega^+}   + \|\nabla u_I - \nabla u_h\|_{0, \Omega^-\cup \Omega^+}.
		\end{aligned}
	\end{equation*}
Using the standard approximation theory \cite{brenner2008fembook, ciarlet2002fembook}
piecewisely, we have 
\begin{equation}
	\begin{aligned}
		\|\nabla u - (\nabla u)_I\|_{0, \Omega^-\cup\Omega^+}
		\lesssim h^2 \|u\|_{3, \Omega^-\cup\Omega^+}. 
	\end{aligned}
\end{equation}
For the second term, we have
\begin{equation*}
	\begin{aligned}
		 &\|(\nabla u)_I - R_hu_I\|_{0, \Omega^-\cup\Omega^+} \\
		 \lesssim  &\|(\nabla u)_I - R_hu_I\|_{0,  \Omega^-}  + 
		 \|(\nabla u)_I - R_hu_I\|_{0,  \Omega^+}\\
		 \lesssim & \|(\nabla u)_I - R_hu_I\|_{0, \infty, \Omega^-}  + 
		 \|(\nabla u)_I - R_hu_I\|_{0, \infty, \Omega^+}\\
		 \lesssim & \|\sum_{i=1}^{\mathcal{N}_h^-}((\nabla u)_I - R_hu_I)(z_i^-)\phi^-_i(z_i^-)\|_{0, \infty, \Omega^-} + 
		 \|\sum_{i=1}^{\mathcal{N}_h^+}((\nabla u)_I - R_hu_I)(z_i^+)\phi^+_i(z_i^+)\|_{0, \infty, \Omega^+} \\
		  \lesssim &   \|\sum_{i=1}^{\mathcal{N}_h^-}((\nabla u)_I - R_hu_I)(z_i^-)\|_{0, \infty, \Omega^-} + 
		 \|\sum_{i=1}^{\mathcal{N}_h^+}((\nabla u)_I - R_hu_I)(z_i^+)\|_{0, \infty, \Omega^+} \\
		 \lesssim &  h^2 \|u\|_{3, \infty,  \Omega^-\cup\Omega^+},
	\end{aligned}
\end{equation*}
where we have used the consistency results in Theorem \ref{thm:consistency} for the gradient recovery operator $G^{\pm}_h$. 
Combining the above estimates with Theorem \ref{thm:bffem_supercloseness}, we complete the proof.
\end{proof}

\begin{remark}
Similarly, if $g \equiv 0$, we can establish the improved superconvergence result for $R_h$ as
\begin{equation*}
	\|\nabla u - R_h u_h \|_{0, \Omega} \lesssim h^{1+\rho} \|u\|_{3, \infty, \Omega^- \cup \Omega^+}.
\end{equation*}
This result matches the superconvergence result for problems without an interface.
\end{remark}

\subsubsection{Superconvergence analysis  of cut finite element methods}
In this subsubsection, we consider the superconvergence analysis of the gradient recovery operator $R_h$ within the context of cut finite element methods.  Let $\mathcal{T}_h$ be a triangulation of $\Omega$ that is independent of the location of the interface $\Gamma$. It is  typically chosen as a uniform mesh. Define the submeshes $\mathcal{T}_h^{\pm}$ as in \eqref{equ:fictmesh} and the fictitious subdomains $\Omega_h^{\pm}$ as in \eqref{equ:fictdomain}. Together, $\Omega_h^{\pm}$ form an overlapping domain decomposition of $\Omega$, illustrated in Figure \ref{fig:interface_mesh}.  We also define the interface mesh $\mathcal{T}_{\Gamma, h}$ as 
\begin{equation}
	\mathcal{T}_{\Gamma, h} = \{K \in \mathcal{T}_h : K \cap \Gamma \neq \emptyset\}.  
\end{equation}
Define the standard continuous linear finite element space $S_{h}^{\pm}$ on  $\mathcal{T}_{h}^{\pm}$ as
\begin{equation}
	S_h^{\pm} = \{ v \in C^0(\Omega_{h}^{\pm}) : v|_{K} \in \mathbb{P}_k(K), \forall K \in \mathcal{T}_h^{\pm} \}.
\end{equation}
The cut finite element space is given by $ S_h = S_h^{-} \oplus S_h^{+} $. Additionally, we define $S_{h,0} \subset S_h $ as the subspace of functions that satisfy vanishing Dirichlet boundary conditions.

For each element $K$ in $\mathcal{T}_{\Gamma, h}$, we define the subsets $K^{\pm} := K \cap \Omega^{\pm}$. Next, we introduce the weights \cite{ann2012robustnitsche, barrau2012robust}
\begin{equation}
\omega^-|_{K} = \frac{\beta^+|K^-|}{\beta^+|K^-|+\beta^-|K^+|}, \quad \omega^+|_{K} = \frac{\beta^-|K^+|}{\beta^+|K^-|+\beta^-|K^+|},
\end{equation}
which satisfy the relation $\omega^- + \omega^+ = 1$. Using these weights, we define the weighted averaging of a function $v_h$ in $V_h$ on the interface $\Gamma$ as follows:
\begin{equation}
\dgal{v_h} = \omega^-v_{h}^- +\omega^+v_{h}^+, \quad  \dgal{v_h} ^{\ast}=  \omega^+v_{h}^- +\omega^-v_{h}^+.
\end{equation}
Let $E^{\pm}$ be the $H^3$-extension operator from $H^3(\Omega^{\pm})$ to $H^3(\Omega)$ such that
\begin{equation}
(E^{\pm}w)|_{\Omega^{\pm}} = w,
\end{equation}
and
\begin{equation}
\|E^{\pm}w\|_{s,\infty, \Omega} \le C \|w\|_{s,\infty, \Omega^\pm}, \quad \forall w\in H^s(\Omega^{\pm}), \, s = 0, 1, 2,3.
\end{equation}

\begin{figure}[!h]
   \centering
   \subcaptionbox{\label{fig:meshwhole}}
  {\includegraphics[width=0.32\textwidth]{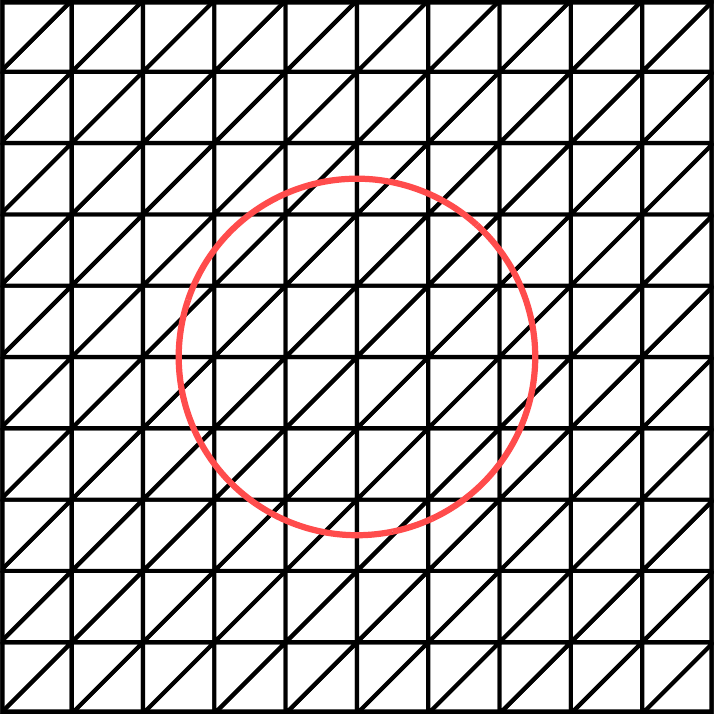}}
  \subcaptionbox{\label{fig:meshone}}
   {\includegraphics[width=0.32\textwidth]{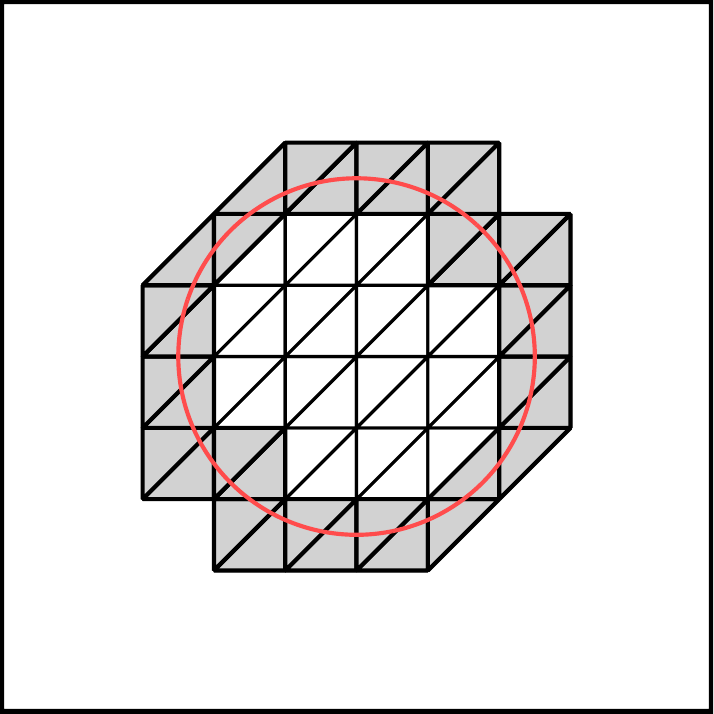}}
  \subcaptionbox{\label{fig:meshtwo}}
  {\includegraphics[width=0.32\textwidth]{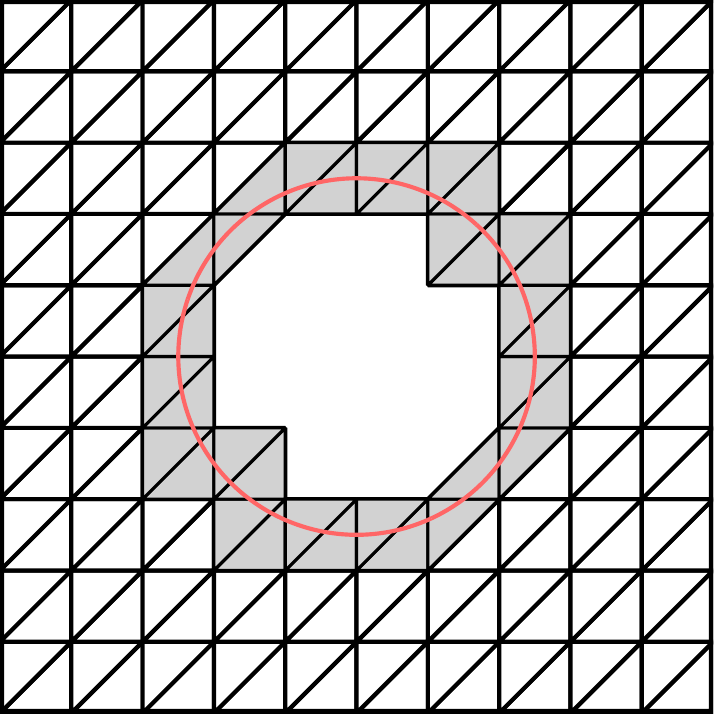}}
   \caption{Triangulation  $\mathcal{T}_h$ on a square domain $\Gamma$ with a circular interface $\Gamma$. (a): Triangulation $\mathcal{T}_h$; (b): Triangulation $\mathcal{T}_{h}^-$ on $\Omega_h^{-}$;
   (c): Triangulation $\mathcal{T}_{h}^+$ on $\Omega_h^{+}$.}\label{fig:interface_mesh}
\end{figure}

The cut finite element approximation \cite{hansbo2002nitsche, burman2015cutfem} of the interface problem \eqref{equ:interface_model} reads as: find $u_h = (u_h^-, u_h^+) \in S_{h,0}$ such that 
\begin{equation}\label{equ:cut_var}
a_h(u_h, v_h) = L_h(v_h), \quad \forall v_h \in S_{h,0},
\end{equation}
where the bilinear form $a_h$ is defined as
\begin{equation}\label{equ:cut_bilinear}
\begin{aligned}
 a_h(u_h,v_h) = &\sum\limits_{s=\pm}\left(\beta^s\nabla u_h^s, \nabla v_h^s\right)_{\Omega^s}
- \left\langle  \llbracket u_h \rrbracket ,\dgal{\beta \partial_nv_{h}}\right\rangle_{\Gamma} \\
&-\left\langle  \llbracket v_h \rrbracket ,\dgal{\beta \partial_nu_{h}}\right\rangle_{\Gamma}
+ h^{-1}\left\langle \gamma \llbracket u_h \rrbracket,  \llbracket v_h \rrbracket\right\rangle_{\Gamma},	
\end{aligned}
\end{equation}
and the linear functional $L_h$ is defined as
\begin{equation}\label{equ:cut_functional}
L_h(v_h) = \sum\limits_{s=\pm}(f, 	v_h^{s})_{\Omega^s} 
+\left\langle g, \dgal{v_h}^{\ast}\right\rangle_{\Gamma},
\end{equation}
with the stability parameter \cite{ann2012robustnitsche, barrau2012robust}
\begin{equation}
\gamma|_{K} = \frac{2h_{K}|\Gamma_T|}{|{K}^-|/\beta^-+|{K}^+|/\beta^+}.
\end{equation}

For the cut finite element solution $u_h$, Guo and Yang \cite{guo2018grcutfem} established the following supercloseness result:
\begin{theorem}\label{thm:cut_supercloseness}
 Suppose the triangulation $\mathcal{T}_h$ satisfies  Condition $(\sigma,\alpha)$.  Let $u$ be the solution of
 the interface problem  \eqref{equ:interface_model} and $u_I$ be the interpolation of $u$ in the finite element space
 $S_{h,0}$. If $u \in H^1(\Omega) \cap H^{3}(\Omega^-\cup\Omega^+)\cap W^{2, \infty}(\Omega^-\cup\Omega^+)$,
 then we have
 \begin{equation}\label{eq:superconvergence}
\|\nabla u_I -\nabla u_h\|_{0,\Omega^-\cup\Omega^+}  \le C\left(h^{1+\rho}(\|u\|_{3, \Omega^-\cup\Omega^+} + \|u\|_{2, \infty, \Omega^-\cup\Omega^+}) + h^{3/2} \|u\|_{2, \infty, \Omega^-\cup\Omega^+}\right),
\end{equation}
where $\rho = \min(\alpha, \frac{\sigma}{2}, \frac{1}{2})$.
\end{theorem}

Building on the supercloseness result, we derive the superconvergence of the gradient recovery operator $R_h$.\begin{theorem}\label{thm:cut_superconvergence}
 Suppose the triangulation $\mathcal{T}_h$ satisfies  Condition $(\sigma,\alpha)$. If $u \in H^1(\Omega) \cap  W^{3, \infty}(\Omega^-\cup\Omega^+)$,
 then we have
 \begin{equation}\label{eq:superconvergence1}
 \|\nabla u -R_hu_h\|_{0, \Omega^-\cup\Omega^+} \lesssim h^{1+\rho}\|u\|_{3, \infty, \Omega^-\cup\Omega^+} + h^{3/2} \|u\|_{2, \infty, \Omega^-\cup\Omega^+},
\end{equation}
where $\rho = \min(\alpha, \frac{\sigma}{2}, \frac{1}{2})$.
\end{theorem}

\begin{proof}
 By the triangle inequality, we obtain
 \begin{equation*}
\begin{aligned}
 \|\nabla u - R_hu_h\|_{0, \Omega^-\cup\Omega^+} \le & \|\nabla u - (\nabla u)_I\|_{0, \Omega^-\cup\Omega^+} +  \|(\nabla u)_I - R_hu_I\|_{0, \Omega^-\cup\Omega^+} + \\
 & \|R_hu_I - R_hu_h\|_{0, \Omega^-\cup\Omega^+} .
\end{aligned}
\end{equation*}
Using the standard approximation theory piecewise, we have
\begin{equation*}
	\|\nabla u - (\nabla u)_I\|_{0, \Omega^-\cup\Omega^+}
	\le h^2 \|u\|_{3,  \Omega^-\cup\Omega^+}. 
\end{equation*}
Next, we estimate the second term:
\begin{equation*}
	\begin{aligned}
		 &\|(\nabla u)_I - R_hu_I\|_{0, \Omega^-\cup\Omega^+} \\
		 \lesssim  &\|(\nabla u)_I - R_hu_I\|_{0,  \Omega^-}  + 
		 \|(\nabla u)_I - R_hu_I\|_{0,  \Omega^-_h}\\
		 \lesssim & \|(\nabla u)_I - R_hu_I\|_{0, \infty, \Omega^-_h}  + 
		 \|(\nabla u)_I - R_hu_I\|_{0, \infty, \Omega^-}\\
		 \lesssim & \|\sum_{i=1}^{\mathcal{N}_h^-}((\nabla u)_I - R_hu_I)(z_i^-)\phi^-_i(z_i^-)\|_{0, \infty, \Omega^-_h} + 
		 \|\sum_{i=1}^{\mathcal{N}_h^+}((\nabla u)_I - R_hu_I)(z_i^+)\phi^+_i(z_i^+)\|_{0, \infty, \Omega^+_h} \\
		  \lesssim &   \|\sum_{i=1}^{\mathcal{N}_h^-}((\nabla u)_I - R_hu_I)(z_i^-)\|_{0, \infty, \Omega^-_h} + 
		 \|\sum_{i=1}^{\mathcal{N}_h^+}((\nabla u)_I - R_hu_I)(z_i^+)\|_{0, \infty, \Omega^+_h} \\
		  \lesssim &  h^2 \|E^-u\|_{3, \infty,  \Omega^-_h} +  h^2 \|E^+u\|_{3, \infty,  \Omega^+_h}\\
		 \lesssim &  h^2 \|u\|_{3, \infty,  \Omega^-\cup\Omega^+}.
	\end{aligned}
\end{equation*}
Then, we estimate the last term:
\begin{equation*}
\begin{aligned}
&\|R_hu_I- R_hu_h\|_{0, \Omega^-\cup\Omega^+}\\
 = & 
\|G_h^-u_I^-- G_h^-u_h^-\|_{0, \Omega^-} + \|G_h^+u_I^+ - G_h^+u_h^+\|_{0, \Omega^+}  \\
\lesssim &\|G_h^-u_I^-- G_h^-u_h^-\|_{0, \Omega^-_h} + \|G_h^+u_I^+ - G_h^+u_h^+\|_{0, \Omega^+_h}  \\
\lesssim &\|\nabla u_I^-- \nabla u_h^-\|_{0, \Omega^-_h} + \|\nabla u_I^+ - \nabla u_h^+\|_{0, \Omega^+_h}  \\
\lesssim &\|\nabla u_I^-- \nabla u_h^-\|_{0, \Omega^-} + \|\nabla u_I^-- \nabla u_h^-\|_{0, \Omega^-_h\backslash\Omega^-} +  \\
 &\|\nabla u_I^+ - \nabla u_h^+\|_{0, \Omega^+} +\|\nabla u_I^+- \nabla u_h^+\|_{0, \Omega^+_h\backslash\Omega^+}  \\
 \lesssim &\|\nabla u_I- \nabla u_h\|_{0, \Omega^-\cup\Omega^+} + \|\nabla u_I^-- \nabla u_h^-\|_{0, \Omega^-_h\backslash\Omega^-}  +\|\nabla u_I^+- \nabla u_h^+\|_{0, \Omega^+_h\backslash\Omega^+} \\
 := & F_1 + F_2 + F_3. 
\end{aligned}
\end{equation*}
For  $F_1$ , Theorem \ref{thm:cut_supercloseness} yields
\begin{equation*}
F_1\le C\left(h^{1+\rho}(\|u\|_{3, \Omega^-\cup\Omega^+} + \|u\|_{2, \infty, \Omega^-\cup\Omega^+}) + h^{3/2} \|u\|_{2, \infty, \Omega^-\cup\Omega^+}\right).
\end{equation*}
Then, we estimate $F_2$ as
\begin{equation*}
\begin{aligned}
F_2^2 & =  \|\nabla u_I^-- \nabla u_h^-\|_{0, \Omega^-_h\backslash\Omega^-}^2
\lesssim \sum_{K\in\mathcal{T}_{\Gamma,h}}\|\nabla((E^-u^-)_I- u_{h}^-)\|_{0,K}^2\\
& \lesssim \sum_{K\in\mathcal{T}_{\Gamma,h}} h^4\|E^-u^-\|_{2,\infty,K}^2 
\lesssim h^4\|u^-\|_{2,\infty,\Omega^-}^2 \sum_{K\in\mathcal{T}_{\Gamma,h}} 1 \\
& \lesssim h^3\|u^-\|_{2,\infty, \Omega^-}^2,
\end{aligned}
\end{equation*}
where we have used the fact $ \sum_{T\in \mathcal{T}_{\Gamma,h}} 1 \approx \mathcal{O}(h^{-1})$.  Similarly, we find
\begin{equation*}
F_3 \le  Ch^{3/2}\|u^+\|_{2,\infty, \Omega^+}.
\end{equation*}
Combining these estimates completes the proof.
\end{proof}

\begin{remark}
	For the cut finite element method, we establish an $\mathcal{O}(h^{3/2})$ order of superconvergence for any function $g$. In contrast, for the body-fitted finite element method, we proved an $\mathcal{O}(h^2)$ superconvergence rate when $g\equiv 0$. 
\end{remark}

\section{Recovery type \textbf{ \textit{a posteriori} }error estimators} \label{sec:estimator}
One of the most important applications of recovery techniques is in the construction of {\it a posteriori} error estimators. In an {\it a posteriori} error estimator, the primary objective is to compute an estimate of the exact error, as defined by
\begin{equation}
e_{K} := \| D^{\alpha} u - D^{\alpha} u_h\|_{0, K},
\end{equation}
where $K \in \mathcal{T}_h$, and $\alpha$ represents a two-index notation corresponding to the differential order of the partial differential equation (PDE). Specifically, for second-order PDEs, $|\alpha| = 1$, while for fourth-order PDEs, $|\alpha|=2$.

As demonstrated by the superconvergence results in Section \ref{sec:super}, the recovered quantities, such as gradients and Hessians, yield higher-order approximations to the exact values of these quantities. The key idea  of recovery type {\it a posteriori} error estimators is to replace unknown quantities with their recovered counterparts. This class of error estimator, originally introduced by Zienkiewicz and Zhu \cite{zhu1992spr1}, is widely known as the ZZ error estimator. The recovery type {\it a posteriori} error estimator is defined as
\begin{equation}
\eta_{K, h} := \| D^{\alpha}_h u_h - D^{\alpha} u_h\|_{0, K},
\end{equation}
where $D^{\alpha}_h$ denotes the recovery operator associated with $D^{\alpha}$.

To  quantify the performance of {\it a posteriori} error estimators,  we introduce the effective index 
\begin{equation}
	\kappa_h := \frac{\left(\sum\limits_{K\in\mathcal{T}_h}\eta_{K, h}^2\right)^{1/2} }{\left(\sum\limits_{K\in\mathcal{T}_h}e_{K}^2 \right)^{1/2}}.
\end{equation}
In the context of {\it a posteriori} error estimation, the ideal behavior of an estimator is referred to as asymptotic exactness.
\begin{definition}
An {\it a posteriori} error estimator $\eta_{K,h}$ is said to be asymptotically exact if
 \begin{equation}
 	\lim_{h\rightarrow 0} \kappa_h = 1. 
 \end{equation}
\end{definition}

\begin{remark}	
	In the context of residual-type {\it a posteriori} error estimates, the vast majority of research has focused on demonstrating the reliability and efficiency of an {\it a posteriori} error estimator, specifically by proving the existence of two positive constants, $C_1$ and $C_2$, such that
	\begin{equation}
		C_1\left(\sum\limits_{K\in\mathcal{T}_h}\eta_{K, h}^2\right)^{1/2} \le \left(\sum\limits_{K\in\mathcal{T}_h}e_{K}^2 \right)^{1/2} \le  C_2\left(\sum\limits_{K\in\mathcal{T}_h}\eta_{K, h}^2\right)^{1/2}.
	\end{equation}
Asymptotic exactness implies both reliability and efficiency, with $C_1 = C_2 = 1$.
	
\end{remark}


\subsection{Second-order elliptic equations}
In this subsubsection, we consider the adaptive computation of the second-order model equation \eqref{equ:elliptic}.   The local {\it a posteriori} error estimator on each element $K\in\mathcal{T}_h$ is defined as
\begin{equation}\label{equ:secondorder_localerrorestimator}
	\eta_{K, h} := \|  G_h u_h - \nabla u_h \|_{0, K},
\end{equation}
and the corresponding global {\it a posteriori}  error estimator as
\begin{equation} \label{equ:secondorder_globalerrorestimator}
	\eta_{h} := \| G_h u_h - \nabla u_h \|_{0, \Omega}.
\end{equation}

With the previous superconvergence result in Subsection \ref{ssec:mildlystructured}, we can show the asymptotic exactness of the recovery type {\it a posteriori} error estimator in \eqref{equ:secondorder_globalerrorestimator}: 
\begin{theorem}\label{thm:ae_secondorder}
	Assume the same conditions in Theorem \ref{thm:grsuperconvergence} or Theorem \ref{equ:grquadsuper}. Let $u$ be the solution of the variation problem \eqref{equ:elliptic} and let $u_h$ be the finite element solution of discrete variational problem \eqref{equ:elliptic_fem}.  Further assume that there is a constant $C(u)>0$  such that 
	\begin{equation} \label{equ:gradlowerbound}
		\|\nabla u - \nabla u_h\|_{0, \Omega} \ge C(u) h^{k}. 
	\end{equation}
	Then, we have 
	\begin{equation}
		\left| \frac{\eta_h}{\|\nabla u - \nabla u_h\|_{0, \Omega}} -1 \right|  \lesssim \mathcal{O}(h^{k-1+ \min\{1, \rho\}}).
	\end{equation}
\end{theorem}  
\begin{proof}
	By the triangle inequality, we can deduce that 
	\begin{equation*}
		\begin{aligned}
			\eta_h \le \|G_h u_h - \nabla u\|_{0, \Omega} + \| \nabla u  - \nabla u_h \|_{0, \Omega}.
		\end{aligned}
	\end{equation*}
	Using \eqref{equ:gradlowerbound} and the superconvergence results of $G_h$, we obtain that 
	\begin{equation*}
		\left| \frac{\eta_h}{\|\nabla u - \nabla u_h\|_{0, \Omega}} -1 \right|  \lesssim \mathcal{O}(h^{k-1+ \min\{1, \rho\}}),
	\end{equation*}
	which concludes the proof. 
\end{proof}

\begin{remark}
	The assumption in \eqref{equ:gradlowerbound} is  the lower bounds approximation results.  Its proof can be found in \cite{lin2014femlowerbound}. 
\end{remark}

\begin{remark}
	The mesh conditions in Theorem \ref{thm:grsuperconvergence} or Theorem \ref{equ:grquadsuper} have been extended to the adaptively refined meshes in \cite{wu2007superconvergenceonadaptivemesh}.  In particular, Wu and Zhang showed in \cite{wu2007superconvergenceonadaptivemesh} that if the meshes satisfy the so called Condition $(\alpha, \sigma, \mu)$ for some $\alpha\ge 0$, $\sigma \ge 0$, and $\mu > 0$, the recovered gradient $G_hu_h$ is superconvergent. 
	\end{remark}

\begin{remark}
  Although we can only prove the asymptotical exactness of the recovery type {\it a posteriori} error on adaptive refined meshes for linear and quadratic elements,  the methodology works for arbitrary high-order elements. It provides a uniform framework to construct an asymptotically exact {\it a posteriori} error estimator for practical purpose. 
\end{remark}

\begin{remark}
    As noted in \cite{wu2007superconvergenceonadaptivemesh}, it is well established that the solution  $u$  may exhibit singularities at the corners of  $\Omega$. Given that the treatment of multiple singular points does not differ from that of a single point, we assume, without loss of generality, that the solution  u has a singularity at the origin $O$.  We decompose the solution  u  into a singular part  $v$  and a smooth part  $w$  as follows:
    \begin{equation*}
        u = v+w
    \end{equation*}
    where 
    \begin{equation*}
        \left|\frac{\partial^m v}{\partial x^i \partial y^{m-i}}\right| \lesssim r^{\delta-m} \text { and }\left|\frac{\partial^m w}{\partial x^i \partial y^{m-i}}\right| \lesssim 1,
    \end{equation*}
for  $ m=1, \ldots, k+2$ and $ i=0, \ldots, m$. In the above equation, $r = \sqrt{x^2+y^2}$.  Under this decomposition, we can establish the superconvergence of recovery techniques on adaptively refined meshes.
\end{remark}

\subsection{Fourth-order elliptic equations}
In this subsection, we consider the model equation given by the following fourth-order elliptic partial differential equation:
\begin{equation}\label{equ:biharmonicmodel}
	\begin{split}
			\Delta^2u = f,  & \quad x\in \Omega, \\
	 u = \frac{\partial  u}{\partial n}  = 0, & \quad  x\in \partial\Omega,
	\end{split}
\end{equation}
where $f \in L^2(\Omega)$ and $n$ denotes the unit outward normal vector to the boundary  $\partial\Omega$.

We first define the local $H^2$ space as follows:
\begin{equation}\label{equ:localh2}
	H^2(\Omega, \mathcal{T}_h) = \left\{v\in L^2(\Omega): v_K = v|_K \in H^2(K), \quad
	\forall K \in \mathcal{T}_h  \right\}.
\end{equation}
Let $\mathcal{E}_h$ denote the set of edges in $\mathcal{T}_h$. For any interior edge $E \in \mathcal{E}_h$, let $K_1$ and $K_2$ be the two elements sharing $E$ as a common edge, and let $n_E$ be the unit outer normal vector of $E$ oriented from $K_1$ to $K_2$. For a function $v \in H^2(\Omega, \mathcal{T}_h)$, we define the averaging and jump of $v$ as
\begin{equation}
	\dgal{\frac{\partial^2 v}{\partial n^2} } = \frac12\left(\left. \frac{\partial^2 v_{K_2}}{\partial n_E^2}\right|_E + \left.\frac{\partial^2 v_{K_1}}{\partial n_E^2}\right|_E\right)
\text{  and  }
\llbracket \frac{\partial v}{\partial n} \rrbracket =\left. \frac{\partial v_{K_2}}{\partial n_E}\right|_E - \left.\frac{\partial v_{K_1}}{\partial n_E}\right|_E.
\end{equation}
These definitions are independent of the orientation of $n_E$.
For any boundary edge $E \in \mathcal{E}_h$, let $n_E$ denote the unit normal pointing outward from $\Omega$. In this case, the averaging and jump of $v$ are defined by
\begin{equation}
	\dgal{\frac{\partial^2 v}{\partial n^2} } =  - \frac{\partial^2 v}{\partial n_E^2}
\text{ and }
\llbracket\frac{\partial v}{\partial n}\rrbracket  = - \frac{\partial v}{\partial n_E}.
\end{equation}
Define the discrete bilinear form $B_h(\cdot, \cdot)$ as
\begin{equation}
\begin{split}
	B_h(v, w) = &\sum_{T\in\mathcal{T}_h} \int_T D^2v:D^2wdx
	+ \sum_{E\in\mathcal{E}_h}\int_E \dgal{\frac{\partial^2 v}{\partial n^2}}\llbracket\frac{\partial w}{\partial n} \rrbracket ds+\\
	&\sum_{E\in\mathcal{E}_h}\int_E \dgal{\frac{\partial v}{\partial n}}
	\llbracket\frac{\partial^2 w}{\partial n^2}\rrbracket ds + \sum_{E\in\mathcal{E}_h}\frac{\gamma}{|E|}\int_E \llbracket\frac{\partial v}{\partial n} \rrbracket
	\llbracket \frac{\partial w}{\partial n} \rrbracket ds,
\end{split}
\end{equation}
where $\gamma$ is the penalty parameter.
The $C^0$ interior penalty (C0IP) method \cite{engel2002c0ip, brenner2005c0ip} for the model problem \eqref{equ:biharmonicmodel} reads as: find $u_h \in S_{h,0}$ such that
\begin{equation}\label{equ:c0ip}
	B_h(u_h, v_h) = (f, v_h) \quad \forall v_h \in S_{h,0}.
\end{equation}

The local {\it a posteriori} error estimator on each element $K\in\mathcal{T}_h$ is defined as \cite{cai2024superconvergentpostprocessingc0interior}
\begin{equation}\label{equ:fourthorder_localerrorestimator}
	\eta_{K, h} := \|  H_h u_h - H u_h \|_{0, K},
\end{equation}
and the corresponding global {\it a posteriori}  error estimator as
\begin{equation} \label{equ:fourthorder_globalerrorestimator}
	\eta_{h} :=  \left(\sum_{K\in\mathcal{T}_h} \| H_h u_h - H u_h \|_{0, K}^2\right)^{1/2}.
\end{equation}

\begin{remark}
For the recovered Hessian matrix $H_h u_h$ of the C0IP method, Cai et al.~\cite{cai2024superconvergentpostprocessingc0interior} proved the following superconvergence result on translation invariant meshes:
\begin{equation}\label{equ:c0ip_error}
		\|Hu-H_hu_h\|_{0, \Omega_0} \lesssim h^{k}\|u\|_{k+2,\Omega}+\|u-u_h\|_{-s,q,\Omega}
\end{equation}
for $k\geq2$ and for some $s \geq 0$, $q \geq 1$. Similar estimates in the $L^\infty$ norm are also presented. Here, $\Omega_0$ refers to a domain similar to the one in Theorem \ref{thm:ultra}. This implies that the error estimator in \eqref{equ:fourthorder_globalerrorestimator} is asymptotically exact, at least on translation invariant meshes. Consequently, it provides a simple and asymptotically exact \textit{a posteriori} error estimator for the C0IP method. 
\end{remark}

\begin{remark}
	For the second-order elliptic problem, the energy norm is based on the gradients. In contrast, for the fourth-order elliptic problem, the Hessian matrix is included in the energy norm. The recovery type \textit{a posteriori} error estimator is defined by replacing the Hessian matrix with the recovered Hessian matrix.
\end{remark}

\subsection{Elliptic interface problems}
In this subsection, we consider the recovery type {\it a posteriori} error estimator for finite element approximations of the interface problem  \eqref{equ:interface_equ}--\eqref{equ:interface_fluxjump}. 
Suppose $u_h = (u^-_h, u^+_h)$ is the numerical solutions using either the body-fitted finite element method \eqref{equ:interface_bffem} or the cut finite element method \eqref{equ:cut_var}.

Using the gradient recovery operator $R_h$ for the interface problems, one can define a local {\it a posteriori} error estimator as 
\begin{equation}\label{equ:localind_interface}
\eta_{K, h} = \|\beta^{1/2}(R_hu_h - \nabla u_h)\|_{0, K},
\end{equation}
and the corresponding global error estimator  as
\begin{equation}\label{equ:globalind_interface}
\eta_h = \|\beta^{1/2}(R_hu_h - \nabla u_h)\|_{0, \Omega^-\cup\Omega^+}.
\end{equation}

Building upon the superconvergence results of $R_h$ in Theorem \ref{thm:bffem_superconvergence} or Theorem \ref{thm:cut_superconvergence}, we can show  the asymptotic exactness of the recovery type {\it a posteriori} error estimator in \eqref{equ:globalind_interface}: 
\begin{theorem}\label{thm:ae_interface}
	Assume the same conditions as in Theorem \ref{thm:bffem_superconvergence} or Theorem \ref{thm:cut_superconvergence}. Let $u$ be the solution to the variational problem \eqref{equ:interface}, and let $u_h$ be the finite element solution to the discrete variational problem \eqref{equ:interface_bffem} or \eqref{equ:cut_var}. Further, assume that there exists a constant $C(u) > 0$ such that
	\begin{equation} \label{equ:gradlowerbound_interface}
		\|\nabla u - \nabla u_h\|_{0, \Omega} \ge C(u) h. 
	\end{equation}
	Then, we have 
	\begin{equation}
		\left| \frac{\eta_h}{\|\nabla u - \nabla u_h\|_{0, \Omega}} -1 \right|  \lesssim \mathcal{O}(h^{k-1+ \min\{1, \rho\}}).
	\end{equation}
\end{theorem}  
\begin{proof}
By the triangle inequality, we can deduce that
	\begin{equation*}
		\begin{aligned}
			\eta_h \le \|R_h u_h - \nabla u\|_{0, \Omega} + \| \nabla u  - \nabla u_h \|.
		\end{aligned}
	\end{equation*}
The superconvergence results of $R_h$ in Theorem \ref{thm:bffem_superconvergence} or Theorem \ref{thm:cut_superconvergence}  imply that
	\begin{equation*}
		\left| \frac{\eta_h}{\|\nabla u - \nabla u_h\|_{0, \Omega}} -1 \right|  \lesssim \mathcal{O}(h^{ \min\{1, \rho\}}),
	\end{equation*}
where we have used the lower bound approximation property of \eqref{equ:gradlowerbound_interface}. This concludes the proof.
\end{proof}

\begin{remark}
For the cut finite element method, the local \textit{a posteriori} error estimator on the interface element $K \in \mathcal{T}_{\Gamma, h}$ should be understood as
	\begin{equation}
		\eta_{K, h}^2 =\|\beta^{1/2}(R_hu_h^- - \nabla u_h^-)\|_{0, K^-}^2 +
		  \|\beta^{1/2}(R_hu_h^+ - \nabla u_h^+)\|_{0, K^+}^2.
	\end{equation}
\end{remark}

\section{Numerical illustrations}\label{sec:ne}
In this section, we present a series of benchmark numerical examples to illustrate the performance of the recovery type {\it a posteriori} error estimators introduced in Section \ref{sec:estimator}. As outlined in the introduction, the adaptive finite element method comprises the following iterative loop: Solve, Estimate, Mark, and Refine.

To complete the loop of adaptive computation, it is essential to employ a strategy for selecting a subset of elements for refinement. In this study, we adopt the bulk marking strategy proposed by D{\"o}rfler \cite{dorfler1996adaptiveconvergence}.  For a given constant $\zeta \in (0,1]$, the bulk marking strategy identifies a subset $\mathcal{M}_h \subset \mathcal{T}_h$ satisfying the condition:
$$
\left(\sum_{K \in \mathcal{M}_h} \eta_{K,h}^2\right)^{\frac{1}{2}} \geq \zeta\left(\sum_{K \in \mathcal{T}_h} \eta_{K,h}^2\right)^{\frac{1}{2}},
$$
where $\mathcal{M}_h$ is chosen to minimize its cardinality. For all numerical examples presented in this section, the parameter $\zeta$ is set to $0.2$ and the rates of convergence are measured in terms of  $N$,  where  $N$ represents the number of vertices in the mesh.  The rates of convergence in terms of  $h$  are double of those measured in terms of  $N$.

Once the subset of marked elements is determined, we perform local refinement. In this work, we employ the newest vertex bisection method \cite{mitchell2017newestvertexbisection, binev2003adaptive}. The main idea of the newest bisection method is to connect one vertex to the midpoint of its opposite edge. For implementation, the longest edge of each triangle is labeled as the refinement edge. To ensure a conforming mesh, neighboring elements may also be bisected. We want to remark that the recovery techniques are independent of local refinement methods. They work well with other local refinement strategies, such as the red-green algorithm \cite{bank1981redgreen}, or even meshes with hanging nodes, as it is a unified meshfree framework.

\subsection{Numerical examples of second-order elliptic problems}\label{ssec:second}
In this subsection, we consider the performance of the recovery type {\it a posteriori} error estimator \eqref{equ:secondorder_localerrorestimator} for the second-order elliptic equations.  We use linear element for the first example while the quadratic element is used for the other two examples.

\textbf{Numerical example 1} In this example, the computational domain is $$\Omega = (-1, 1)^2\backslash \{(x, 0), 0\le x \le 1 \}$$. We consider the Poisson equation 
\begin{equation}
	 - \Delta u = 1.
\end{equation}
The Dirichlet boundary condition is specified to match the exact solution:
\begin{equation*}
	u =  \frac{\sqrt{2}}{2}\sqrt{r-x} - \frac{1}{4}r^2, 
	\end{equation*}
where $r = \sqrt{x^2+y^2}$.

\begin{figure}[!h]
   \centering
   \subcaptionbox{\label{fig:adaptvemesh_crack}}
  {\includegraphics[width=0.39\textwidth]{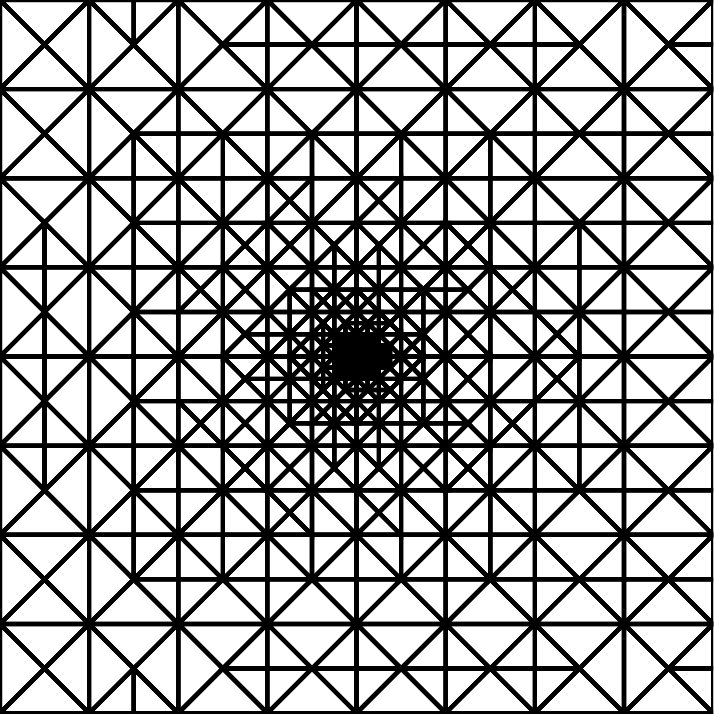}}
 \hspace{0.2in}
  \subcaptionbox{\label{fig:sol_crack}}
   {\includegraphics[width=0.54\textwidth]{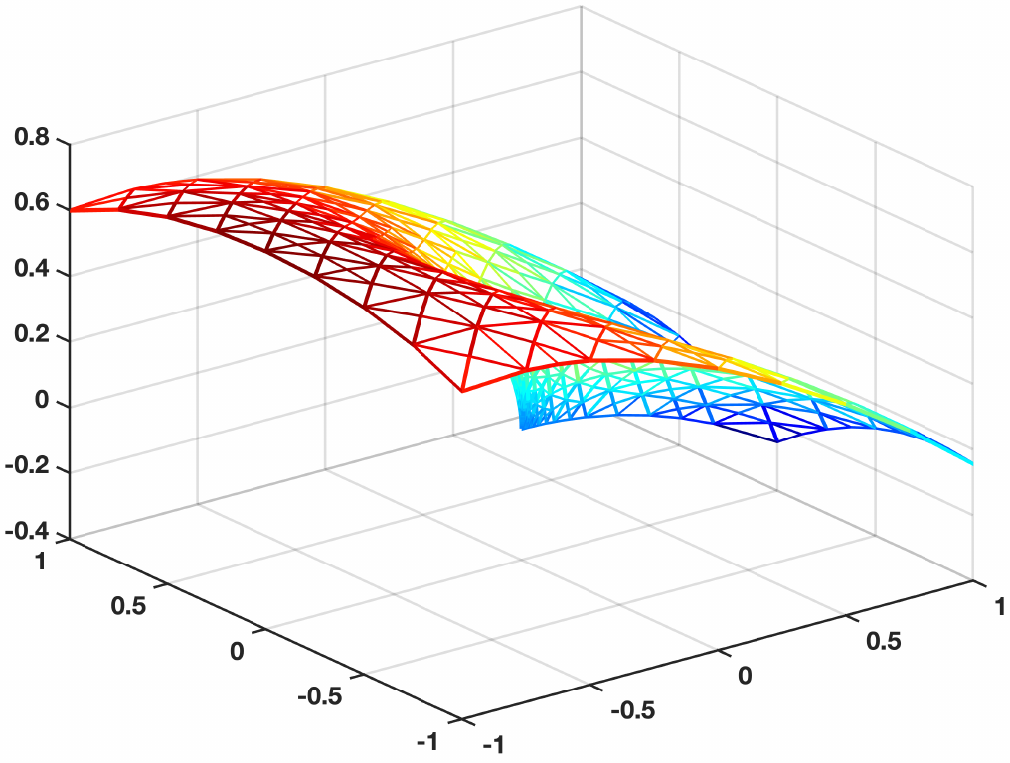}}
   \caption{Adaptive mesh and numerical solution for the crack problem using linear elements: (a) adaptively refined mesh; (b) finite element solution.}\label{fig:mesh_crack}
\end{figure}

\begin{figure}[!h]
   \centering
   \subcaptionbox{\label{fig:error_crack}}
  {\includegraphics[width=0.47\textwidth]{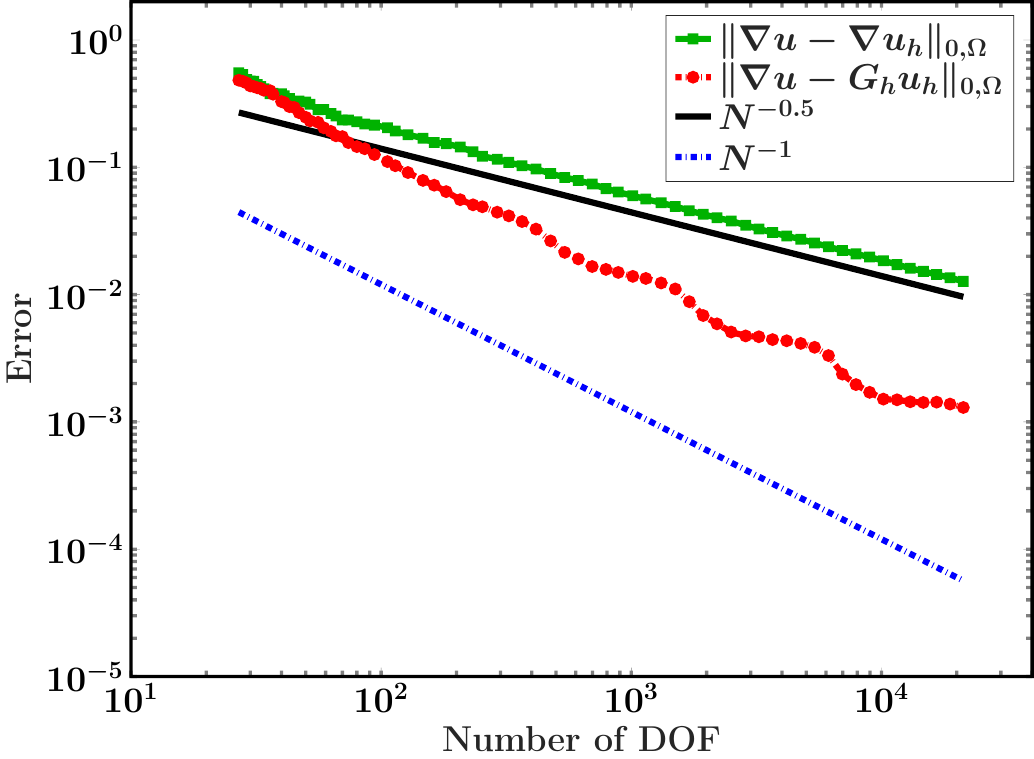}}
   \hspace{0.2in}
    \subcaptionbox{\label{fig:effectiveindex_crack}}
   {\includegraphics[width=0.45\textwidth]{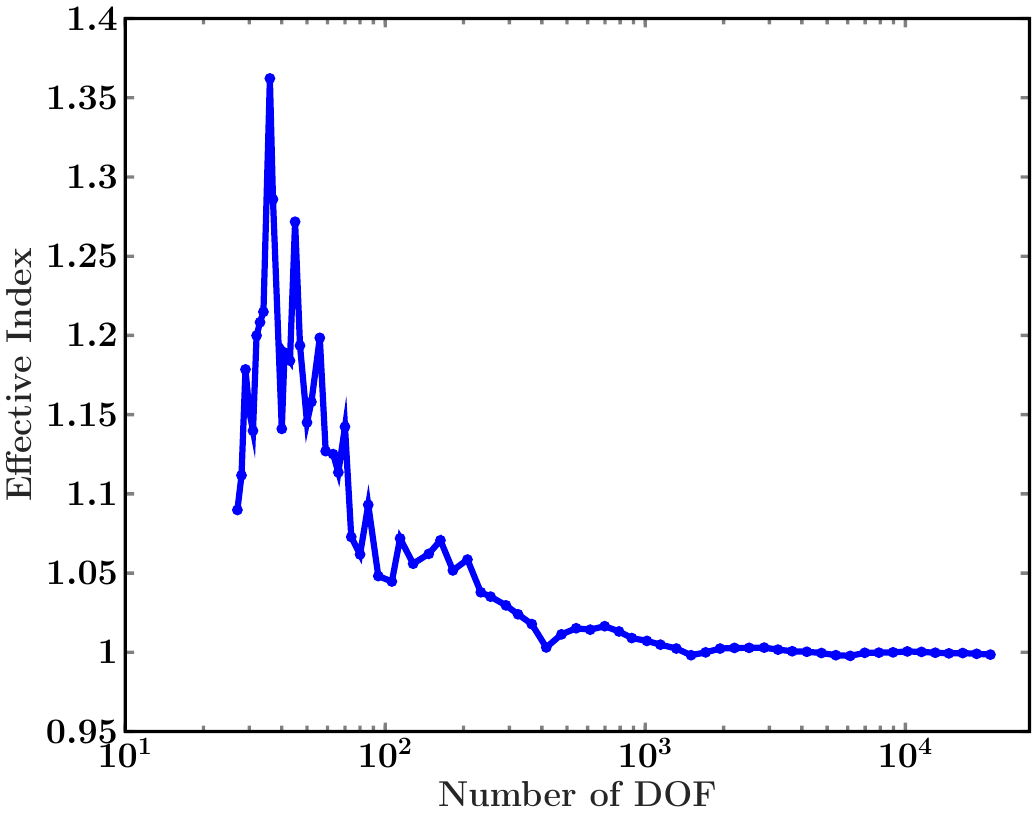}}
\caption{Numerical results for the crack problem using linear elements: (a) numerical error; (b) plot of the effective index. }
\label{fig:result_crack}
\end{figure}

The solution exhibits singular behavior near the origin. To recover the optimal convergence rate, we employ the adaptive finite element method (AFEM) incorporating the {\it a posteriori} error estimator \eqref{equ:secondorder_localerrorestimator}. Figure \ref{fig:adaptvemesh_crack} displays the adaptively refined mesh. The figure illustrates the effectiveness of the error estimator \eqref{equ:secondorder_localerrorestimator} in capturing the singularity.

The convergence history is summarized in Figure \ref{fig:error_crack}. From the results, it is evident that $\|\nabla u - \nabla u_h\|_{0, \Omega}$ decreases at the optimal rate of $\mathcal{O}(N^{-0.5})$, while $\|\nabla u - G_h u_h\|_{0, \Omega}$ achieves a superconvergent rate of $\mathcal{O}(N^{-1})$. In Figure \ref{fig:effectiveindex_crack}, the effective index of the error estimator is presented. This confirms that the error estimator \eqref{equ:secondorder_localerrorestimator} is asymptotically exact, as predicted by Theorem~\ref{thm:ae_secondorder}.

\vspace{0.1in}

\textbf{Numerical example 2} In this example, we consider the Poisson equation on the unit square $\Omega = (0,1)^2$. The exact solution is given by
\begin{equation*}
u(x, y) = \tan^{-1}(a(r - r_0)),
\end{equation*}
where $r = \sqrt{(x - x_0)^2 + (y - y_0)^2}$. For this test, we set $r_0 = 0.7$, $a = 50$, and $(x_0, y_0) = (-0.05, -0.05)$. The solution is smooth but exhibits a steep interior layer, as illustrated in Figure \ref{fig:sol_wavefront}.

The initial mesh is the uniform mesh of regular pattern with 32 triangles. The interior sharp layer is totally unresolved by the initial mesh which causes the major difficulties.  Figure \ref{fig:adaptvemesh_wavefront}  is the mesh generated by the adaptive finite  element method. It is obvious that the mesh is refined to resolve the interior layer as expected.

\begin{figure}[!ht]
   \centering
   \subcaptionbox{\label{fig:adaptvemesh_wavefront}}
  {\includegraphics[width=0.39\textwidth]{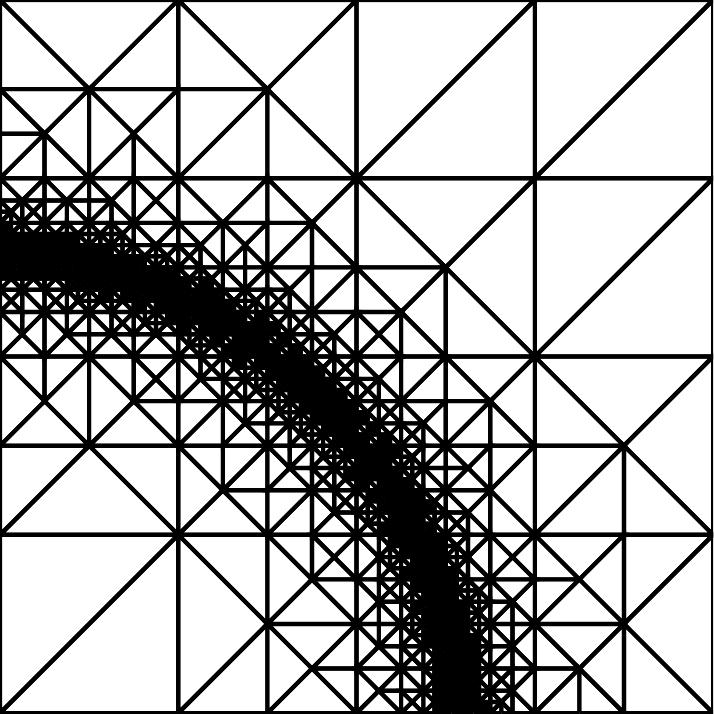}}
 \hspace{0.2in}
  \subcaptionbox{\label{fig:sol_wavefront}}
   {\includegraphics[width=0.54\textwidth]{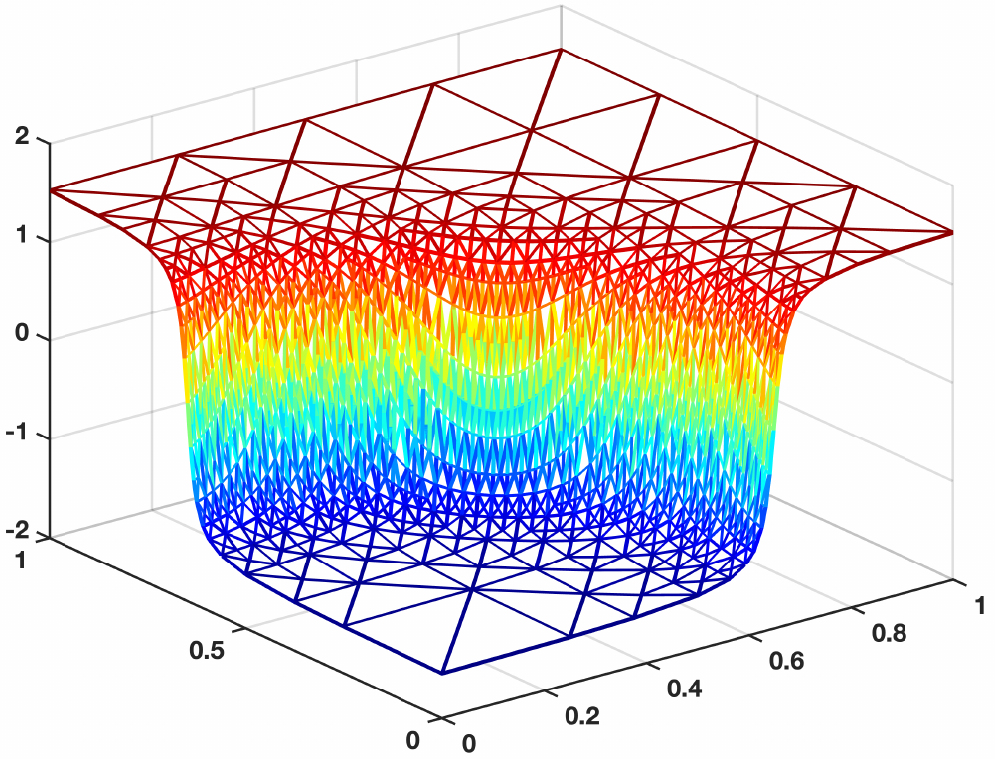}}
   \caption{Adaptive mesh and numerical solution for the problem with an interior layer using linear elements: (a) adaptively refined mesh; (b) finite element solution. }\label{fig:mesh_wavefront}
\end{figure}

\begin{figure}[!ht]
   \centering
   \subcaptionbox{\label{fig:error_wavefront}}
  {\includegraphics[width=0.47\textwidth]{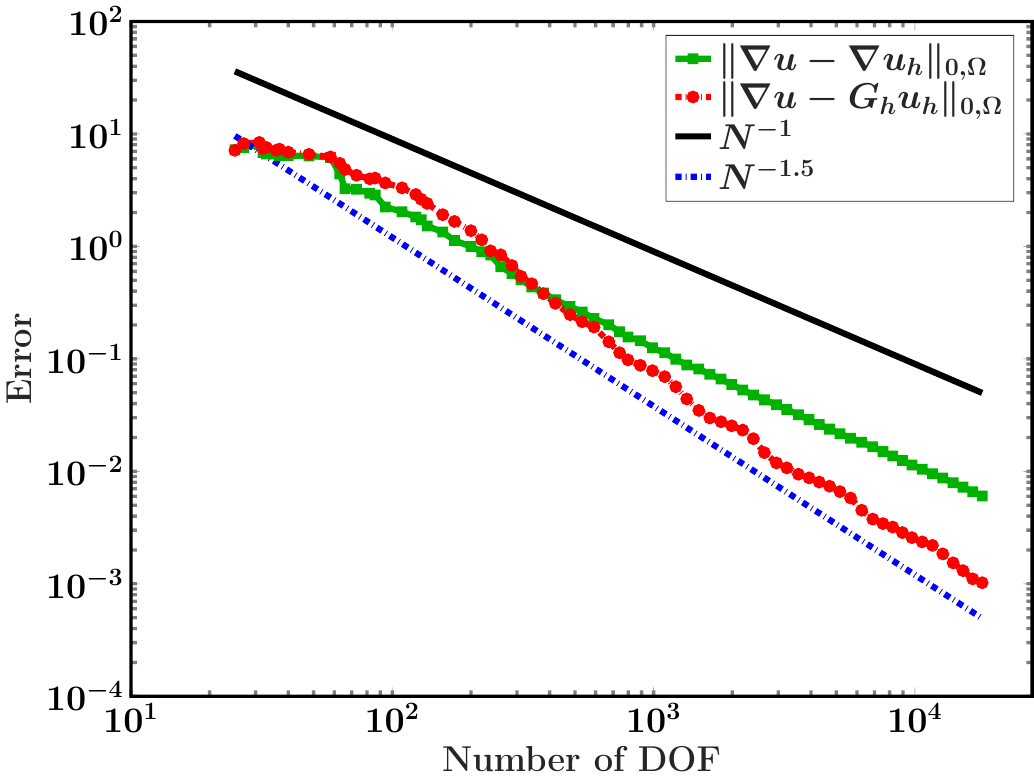}}
   \hspace{0.2in}
    \subcaptionbox{\label{fig:effectiveindex_wavefront}}
   {\includegraphics[width=0.45\textwidth]{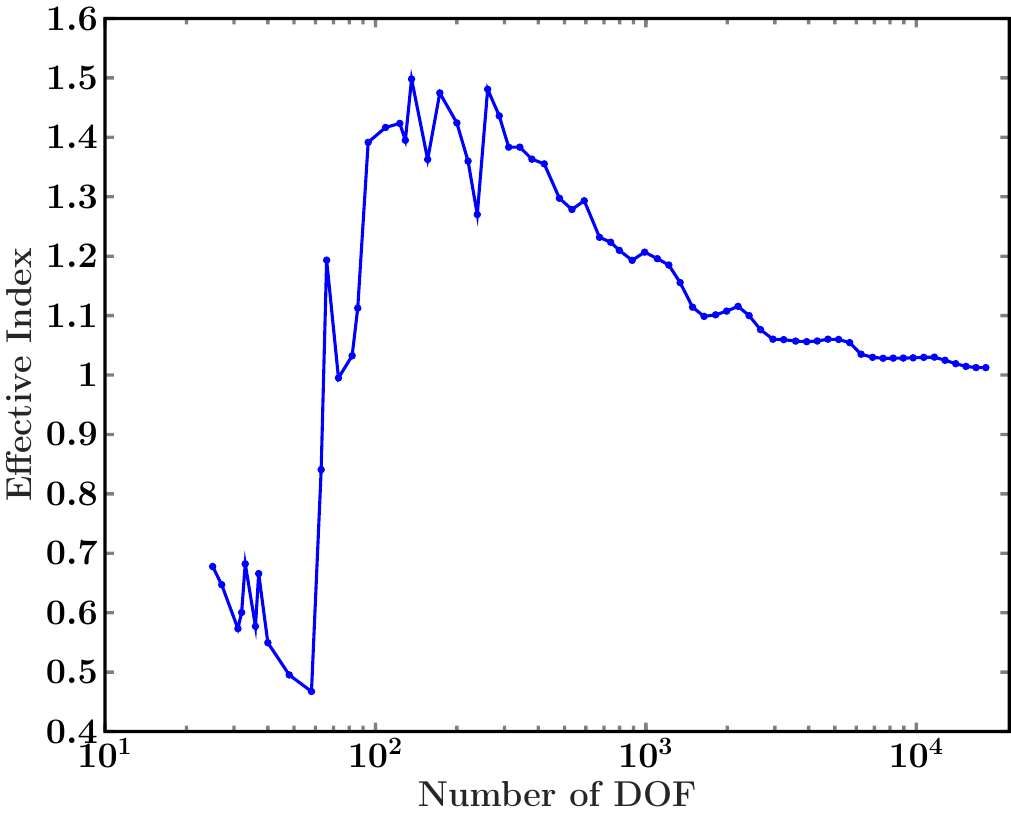}}
\caption{Numerical results for the problem with an interior layer using quadratic elements: (a) numerical error; (b) plot of the effective index.}
\label{fig:result_wavefront}
\end{figure}

In Figure \ref{fig:result_wavefront}, we present the qualitative results. As anticipated, the desired $\mathcal{O}(N^{-1})$ optimal convergence rate for the finite element gradient and $\mathcal{O}(N^{-1.5})$ superconvergence rate for the recovered gradient is observed. Additionally, the limit of the effective index is numerically proven to be one, which validates the asymptotic exactness of the error estimator \eqref{equ:secondorder_localerrorestimator}.

\vspace{0.1in}

\textbf{Numerical example 3} In this example, we consider the Poisson equation on the unit square, where the exact solution is given in \cite{guo2019grvem}
$$
u(x,y) = \frac{1}{2\pi \sigma }\left[ e^{-\frac{1}{2}\left( \frac{x-\mu_1}{\sigma}\right)^2}
e^{-\frac{1}{2}\left( \frac{y-\mu_1}{\sigma}\right)^2} +
e^{-\frac{1}{2}\left( \frac{x-\mu_2}{\sigma}\right)^2}
e^{-\frac{1}{2}\left( \frac{y-\mu_2}{\sigma}\right)^2}
\right].
$$
The standard deviation is chosen to be $ \sigma = \sqrt{10^{-3}} $, and the two mean values are $ \mu_1 = 0.25 $ and $ \mu_2 = 0.75 $.

The difficulty of this problem arises from the presence of two Gaussian surfaces, as shown in Figure \ref{fig:sol_gaussian}, where the solution decays rapidly. The initial mesh consists of 32 uniform triangles, which do not sufficiently capture the Gaussian surfaces. In Figure \ref{fig:adaptvemesh_gaussian}, we present the adaptively refined mesh, which is evidently refined around the locations of the Gaussian surfaces. Figure \ref{fig:error_gaussian} displays the numerical errors. Consistent with Numerical Example 2, the expected optimal and superconvergent results are observed. Additionally, the asymptotic accuracy of the error estimator \eqref{equ:secondorder_localerrorestimator} is numerically confirmed in Figure \ref{fig:effectiveindex_gaussian}, where the effective index converges to one.

\begin{figure}[!ht]
   \centering
   \subcaptionbox{\label{fig:adaptvemesh_gaussian}}
  {\includegraphics[width=0.39\textwidth]{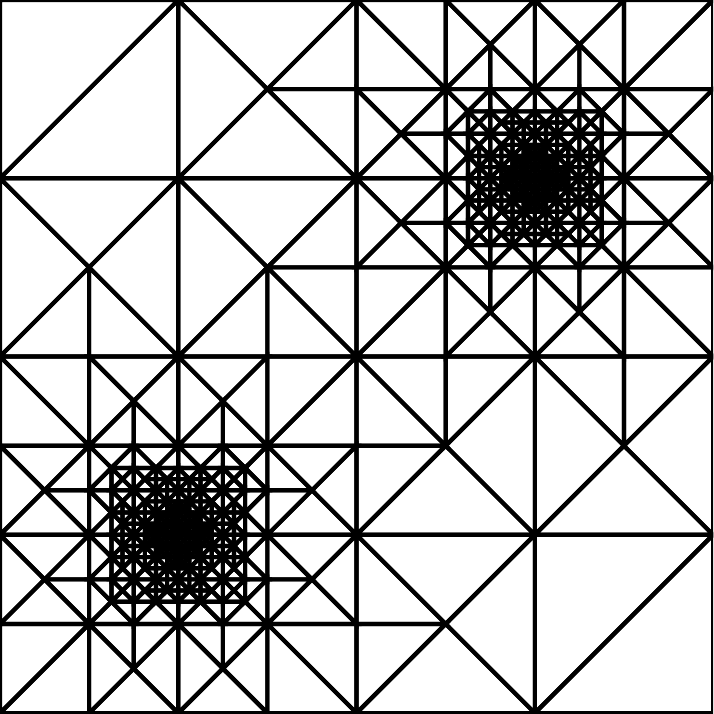}}
 \hspace{0.2in}
  \subcaptionbox{\label{fig:sol_gaussian}}
   {\includegraphics[width=0.54\textwidth]{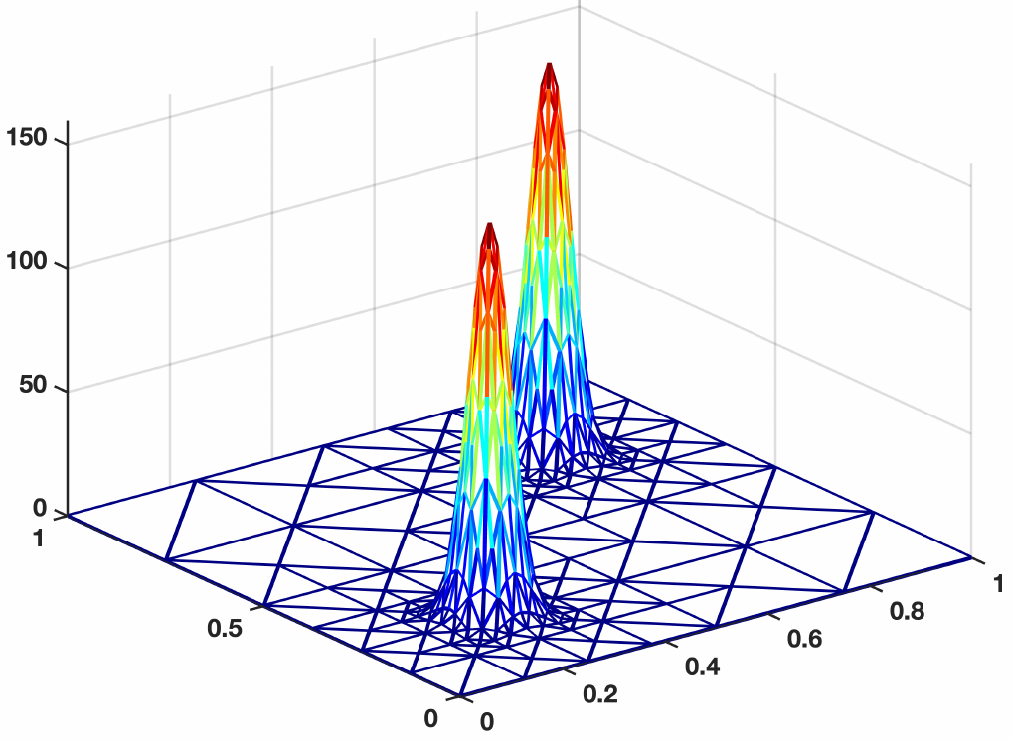}}
   \caption{Adaptive mesh and numerical solution for the Gaussian problem using quadratic elements: (a) adaptively refined mesh; (b) finite element solution.}\label{fig:mesh_gaussian}
\end{figure}

\begin{figure}[!ht]
   \centering
   \subcaptionbox{\label{fig:error_gaussian}}
  {\includegraphics[width=0.47\textwidth]{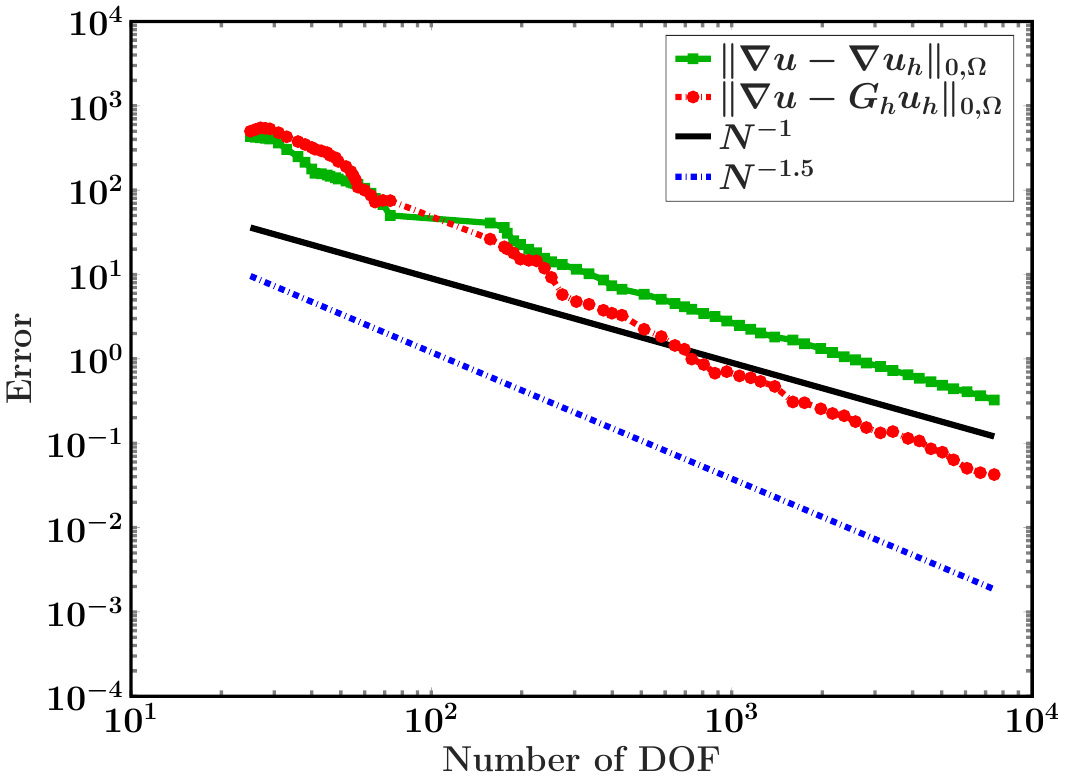}}
   \hspace{0.1in}
    \subcaptionbox{\label{fig:effectiveindex_gaussian}}
   {\includegraphics[width=0.45\textwidth]{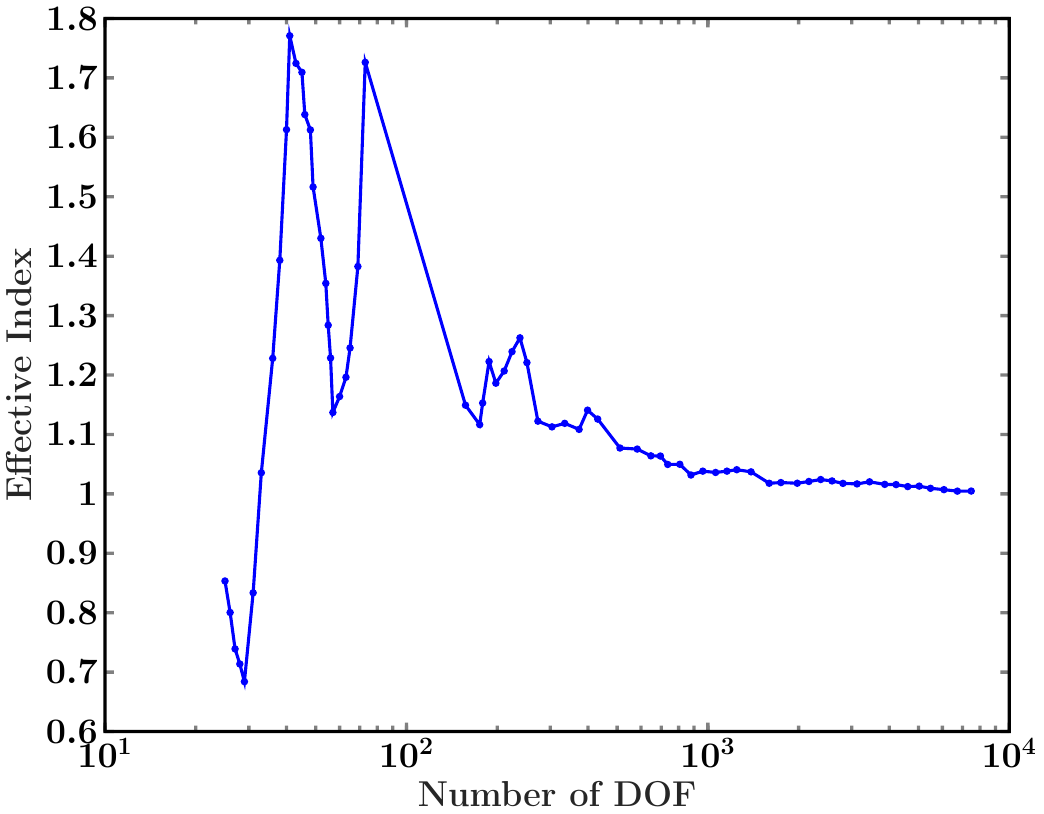}}
\caption{Numerical results for the Gaussian problem using quadratic elements: (a) numerical error; (b) plot of the effective index.}
\label{fig:result_gaussian}
\end{figure}

\subsection{Numerical examples  of fourth-order problems}\label{ssec:fourth}
In this subsection, we consider the performance the recovery type {\it a posteriori} error estimator \eqref{equ:fourthorder_localerrorestimator} for the C0IP method.  

\textbf{Numerical example 4} In this example, we consider the biharmonic equation \eqref{equ:biharmonicmodel} defined over the Lshaped domain $\Omega = (-1, 1)^2 \backslash ([0, 1)\times(-1, 0])$, as investigated in \cite{brenner2010adaptivec0ip}.  The exact singular solution in polar coordinates is given by 
\begin{equation}\label{equ:biharmonic_exactsolution}
	u(r, \theta) = (r^2\cos^2(\theta)-1)^2 (r^2\sin^2(\theta)-1)^2 g(\theta), 
\end{equation}
where the function $g(\theta)$ is defined to be
\begin{equation*}
	\begin{aligned}
		g(\theta) =  &  [\frac{1}{z-1}\sin((z-1)\omega) - \frac{1}{z+1}\sin((z+1)\omega)  ] 
		\times \left[\cos((z-1)\theta) - \cos((z+1)\theta) \right] \\
		&-[\frac{1}{z-1}\sin((z-1)\theta) - \frac{1}{z+1}\sin((z+1)\theta)  ] 
		 \times \left[\cos((z-1)\omega) - \cos((z+1)\omega) \right].
	\end{aligned}
\end{equation*}
Here,  $z= 0.544483736782464$ is a noncharacteristic root of $\sin^2(z\omega)=z^2\sin^2(\omega)$,  and $\omega = \frac{3\pi}{2}$.

\begin{figure}[!ht]
   \centering
   \subcaptionbox{\label{fig:adaptvemesh_brenner}}
  {\includegraphics[width=0.39\textwidth]{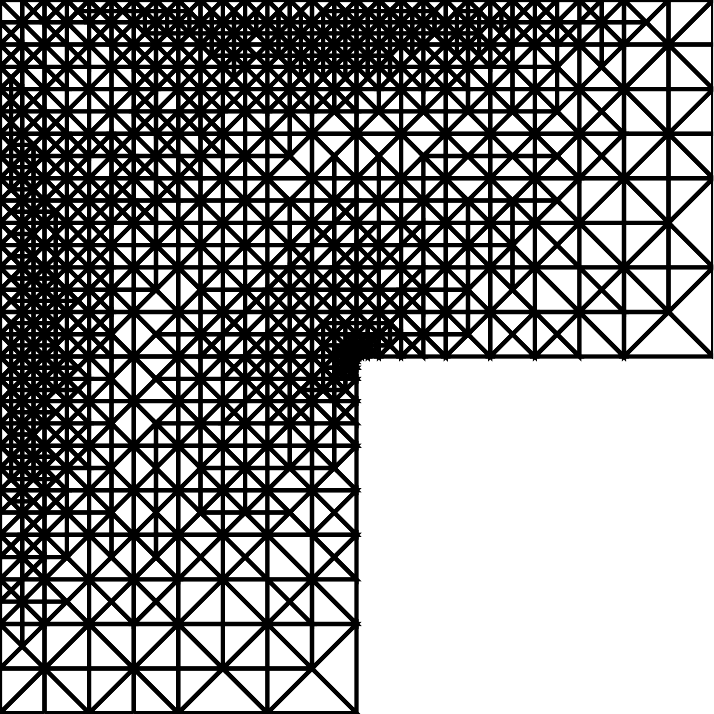}}
 \hspace{0.2in}
  \subcaptionbox{\label{fig:sol_brenner}}
   {\includegraphics[width=0.54\textwidth]{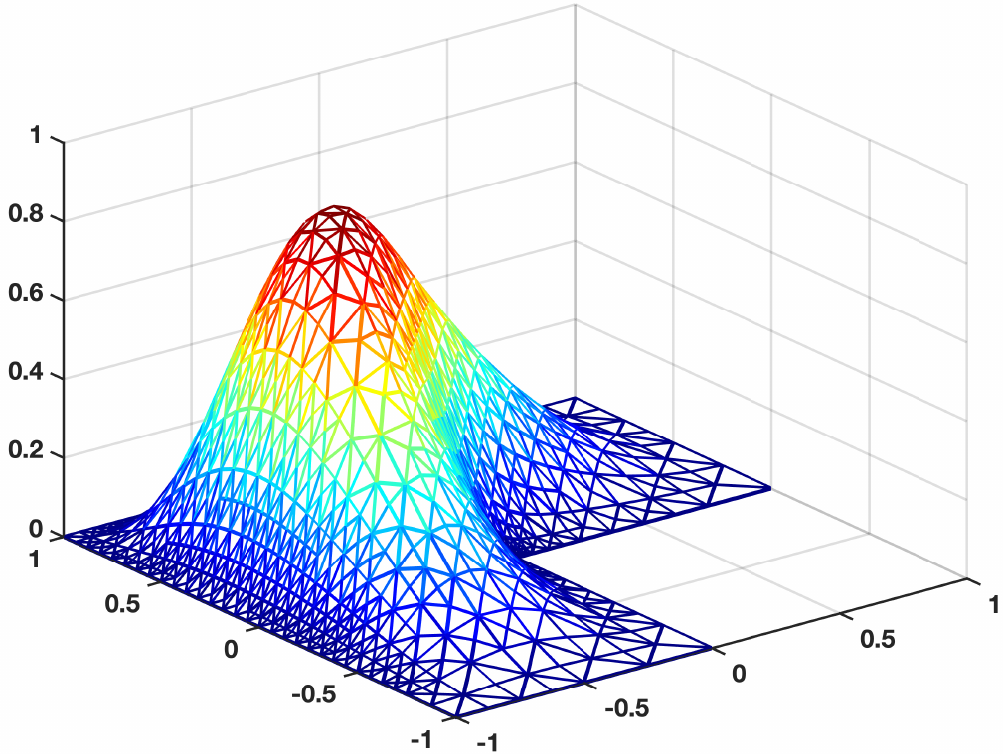}}
   \caption{Adaptive mesh and numerical solution for the biharmonic equation on the Lshape domain using the quadratic C0IP method: (a) adaptively refined mesh; (b) finite element solution. }\label{fig:mesh_gaussian1}
\end{figure}

\begin{figure}[!ht]
   \centering
   \subcaptionbox{\label{fig:error_brenner}}
  {\includegraphics[width=0.47\textwidth]{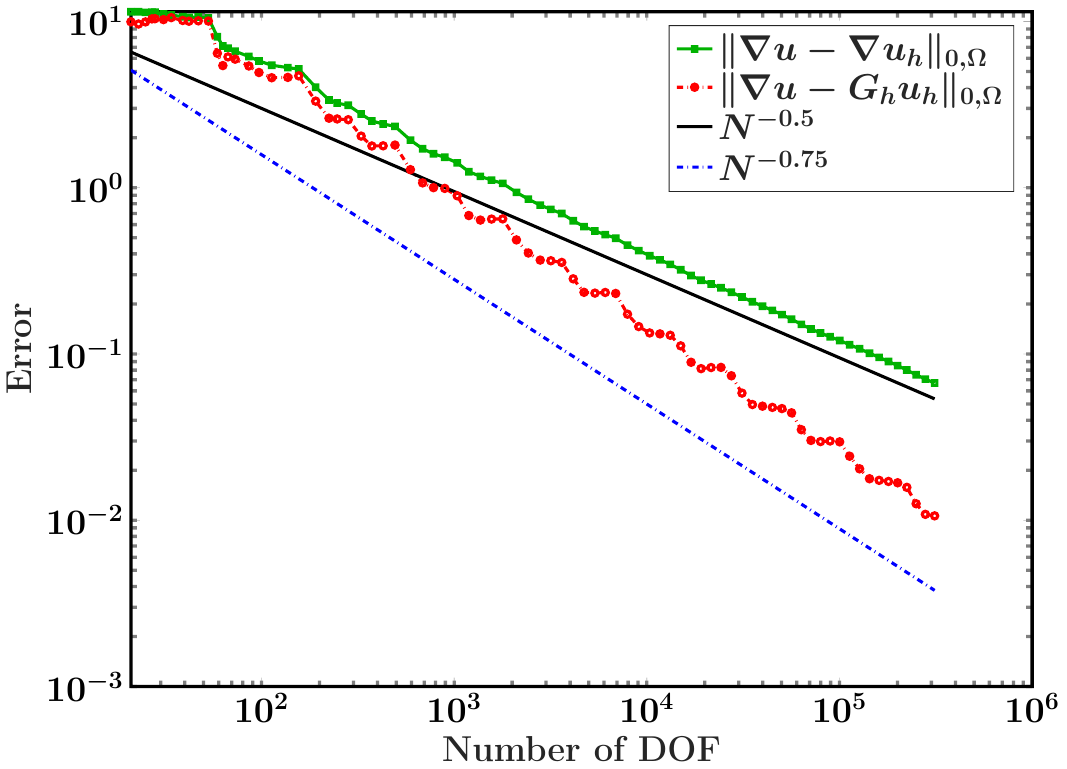}}
   \hspace{0.1in}
    \subcaptionbox{\label{fig:effectiveindex_brenner}}
   {\includegraphics[width=0.45\textwidth]{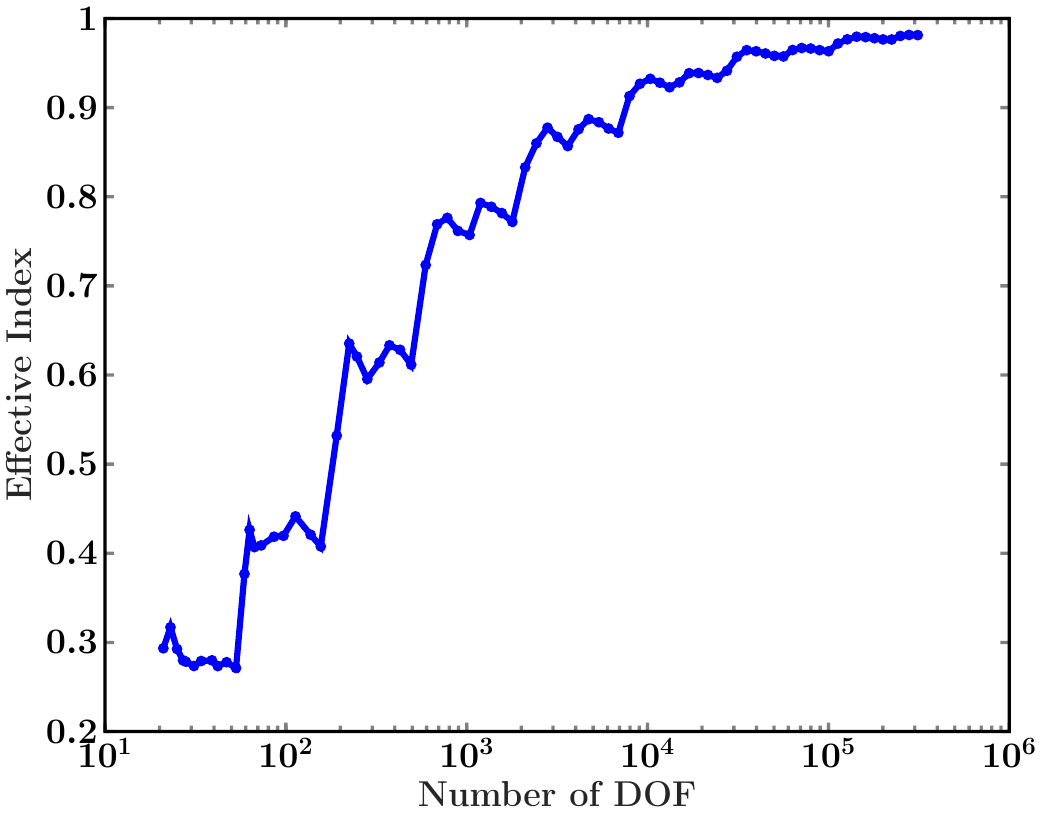}}
\caption{Numerical results for the biharmonic equation on the Lshape domain using the quadratic C0IP method: (a) numerical error; (b) plot of the effective index. }
\label{fig:result_gaussian1}
\end{figure}

As illustrated in Figure \ref{fig:sol_brenner}, the solution $u$ exhibits sharp transitions along two long edges and a singularity near the origin. The initial computational domain is discretized using a uniform mesh of regular pattern consisting of 24 triangles. In Figure \ref{fig:adaptvemesh_brenner}, the adaptively refined mesh is depicted. It is evident that the refinement effectively resolves both the singularity and the regions of sharp transitions.

The numerical errors are presented in Figure \ref{fig:error_brenner}. Examination of the data reveals that the discrete $H^2$ error converges at the optimal rate of $\mathcal{O}(N^{-0.5})$, while the recovered $H^2$ error achieves a superconvergent rate of $\mathcal{O}(N^{-0.75})$. In Figure \ref{fig:effectiveindex_brenner}, the effective index of the recovery-based {\it a posteriori} error estimator \eqref{equ:fourthorder_localerrorestimator} is plotted.  The effective index $\kappa$ asymptotically approaches one, thereby affirming the asymptotic exactness of the error estimator \eqref{equ:fourthorder_localerrorestimator}.

\vspace{0.1in}

\textbf{Numerical example 5} In this example, we consider the computational domain $\Omega$ with a reflex angle, as shown in Figure \ref{fig:mesh_obtuseangle}. The biharmonic equation \eqref{equ:biharmonicmodel} is solved on $\Omega$, with the exact solution given by \eqref{equ:biharmonic_exactsolution}, where $z = 0.505009698896589$ and $\omega = \frac{7\pi}{4}$, as in \cite{brenner2010adaptivec0ip}.

Similar to Example 4, the singularity and regions of sharp transitions are effectively captured by the adaptively refined mesh. Figure \ref{fig:error_obtuseangle} displays the numerical errors in terms of degrees of freedom. The figure demonstrates the optimal decay of the discrete $H^2$ error. Additionally, it is observed that the recovered gradient error superconverges at a rate of $\mathcal{O}(N^{-0.75})$. In Figure \ref{fig:effectiveindex_obtuseangle}, the effectivity index is plotted. This figure reveals that the effectivity index is close to 1, confirming that the recovery type {\it a posteriori} error estimator \eqref{equ:fourthorder_localerrorestimator} is asymptotically exact.

Compared to the residual type {\it a posteriori} error estimator for the C0IP method in \cite{brenner2010adaptivec0ip}, the error estimator \eqref{equ:fourthorder_localerrorestimator} is much simpler in formulation and asymptotically exact. As a byproduct, it provides a superconvergent discrete Hessian matrix.

\begin{figure}[!ht]
   \centering
   \subcaptionbox{\label{fig:adaptvemesh_obtuseangle}}
  {\includegraphics[width=0.39\textwidth]{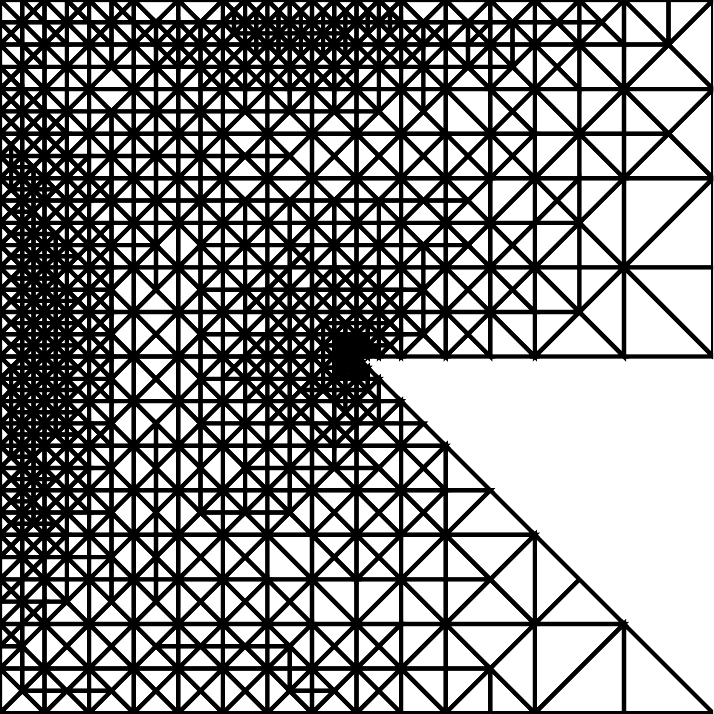}}
 \hspace{0.2in}
  \subcaptionbox{\label{fig:sol_obtuseangle}}
   {\includegraphics[width=0.54\textwidth]{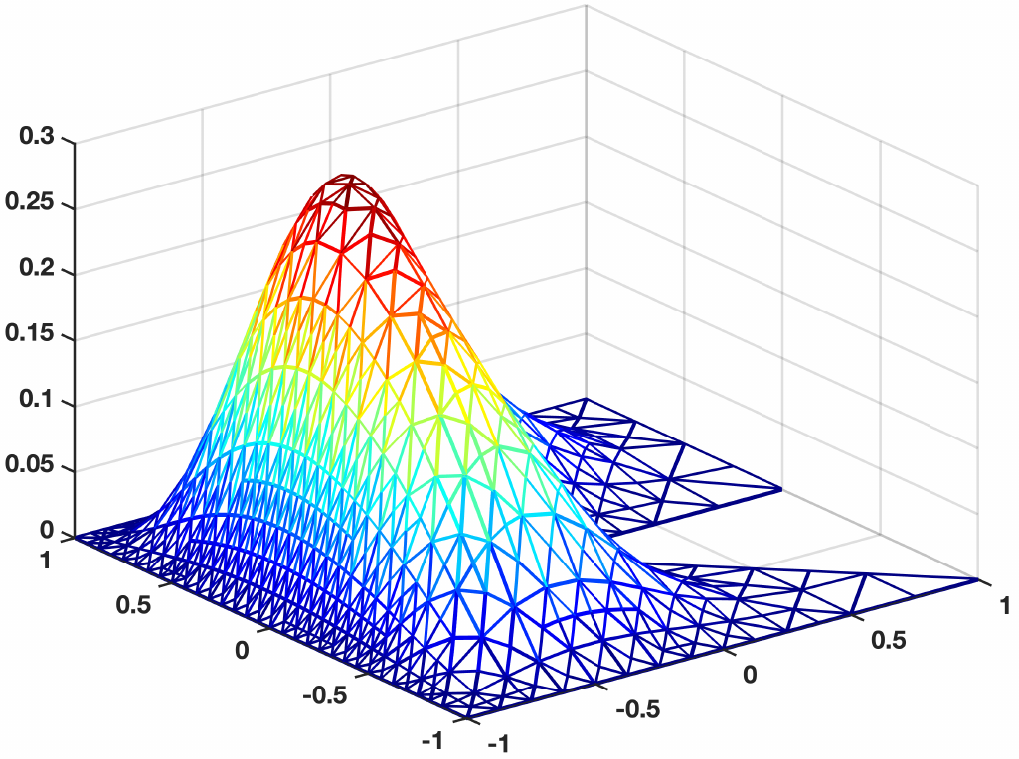}}
   \caption{Adaptive mesh and numerical solution for the biharmonic equation on a domain with a reflex angle using the quadratic C0IP method: (a) adaptively refined mesh; (b) finite element solution. }\label{fig:mesh_obtuseangle}
\end{figure}

\begin{figure}[!ht]
   \centering
   \subcaptionbox{\label{fig:error_obtuseangle}}
  {\includegraphics[width=0.47\textwidth]{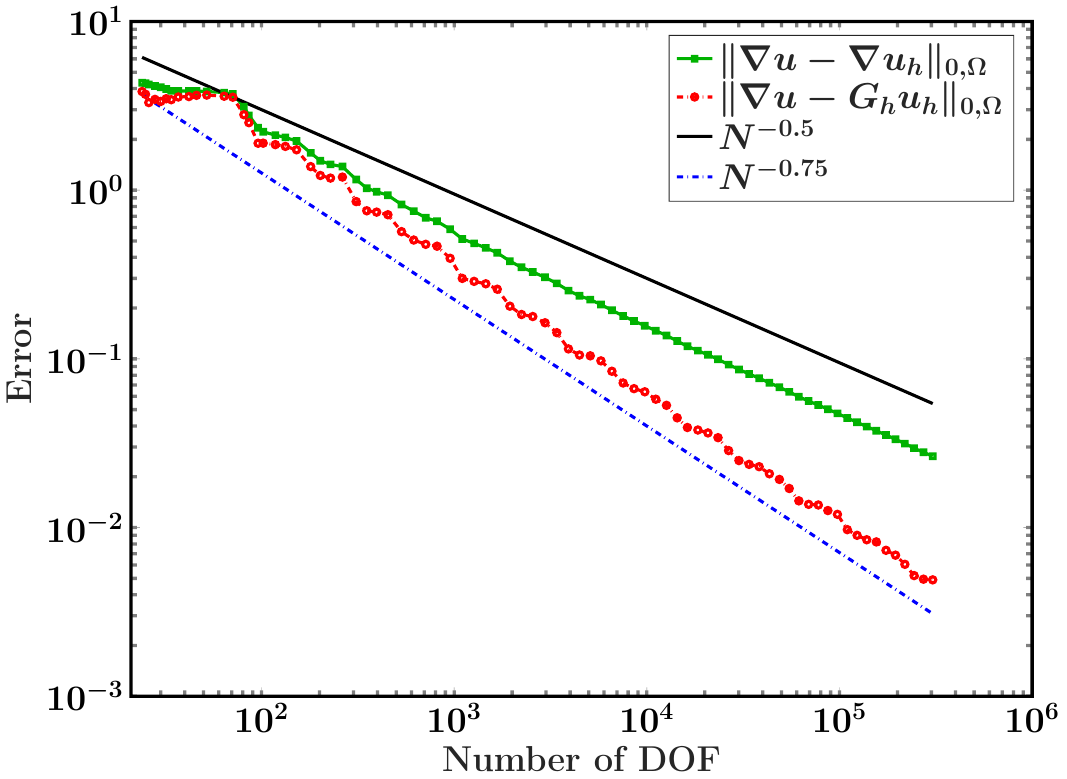}}
   \hspace{0.1in}
    \subcaptionbox{\label{fig:effectiveindex_obtuseangle}}
   {\includegraphics[width=0.45\textwidth]{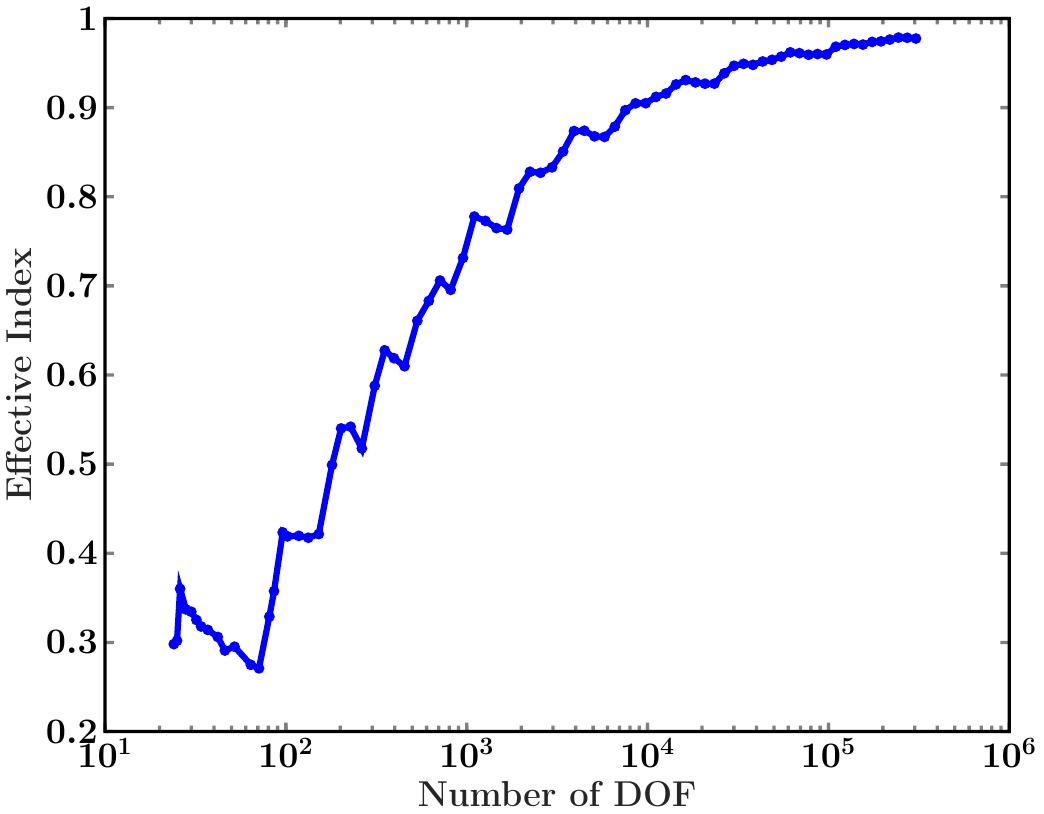}}
\caption{Numerical results for the biharmonic equation on a domain with a reflex angle using the quadratic C0IP method: (a) numerical error; (b) plot of the effective index. }
\label{fig:result_obtuseangle}
\end{figure}

\subsection{Numerical examples of interface problems} \label{ssec:interface_numerical}
In this subsection, we  present several numerical examples to demonstrate the asymptotic exactness of the recovery type {\it a posteriori} error estimator \eqref{equ:localind_interface} for interface problems. 
\subsubsection{Numerical example for body-fitted finite element methods}

\textbf{Numerical example 6} In this example, we decompose the computational domain
 $\Omega = (-1, 1)\times (-1, 1)$ into two parts:
 $\Omega^- = \{ z=(x,y)\in \Omega: x> 0, y>0\}$ and $\Omega^+ = \Omega\setminus \Omega^-$.  The diffusion coefficient $\beta$ in \eqref{equ:interface_equ} is chosen as
 \begin{equation*}
\beta(z) =
\left\{
\begin{array}{ccc}
   \beta^- &   \text{if } z\in \Omega^-,\\
       1 &  \text{if } z\in \Omega^+,
\end{array}
\right.
\end{equation*}
with $\beta^-$ being a constant.  When $f=0$ in \eqref{equ:interface_model},  the exact solution $u$ in polar coordinates is given by
 \begin{equation*}
u(r, \theta) =
\left\{
\begin{array}{llc}
    r^{\mu}\cos(\mu(\theta-\pi/4))&  \text{if } 0 \le \theta \le \pi/2, \\
       r^{\mu}\nu\cos(\mu(\theta-5\pi/4))  &   \text{if } \pi/2 \le \theta \le 2\pi,\\
   \end{array}
\right.
\end{equation*}
where
 \begin{equation*}
\mu = \frac{4}{\pi}\left( \sqrt{\frac{3+\beta^-}{1+3\beta^-}} \right) \quad \text{        and        } \quad
\nu = - \beta^-\frac{\sin(\mu\pi/4)}{\sin(3\mu\pi/4)}.
\end{equation*}
Note that $u \in H^{1+s}(\Omega^\pm)$ for any $0 < s < \mu$.

We begin with a uniform initial mesh of  regular pattern, comprising  right triangles. Figure \ref{fig:adaptvemesh_bffem} illustrates an adaptively refined mesh, while Figure \ref{fig:sol_bffem} presents its corresponding finite element solution. These results demonstrate that the recovery-type \textit{a posteriori} error estimator \eqref{equ:localind_interface} effectively captures the singularity without causing over-refinement.

Figures \ref{fig:result_bffem_10e3} and \ref{fig:result_bffem_10e6} illustrate the numerical convergence rates  $\beta^-=1000$ and $\beta^-=10000$, respectively. In both cases, the results exhibit an optimal convergence rate of $\mathcal{O}(N^{-0.5})$ for the energy error and a superconvergence rate of $\mathcal{O}(N^{-1})$,  aligning with Theorem~\ref{thm:bffem_superconvergence}. The effective index for $ \beta^- = 10^k $ with $ k = 1, \dots, 4 $ is plotted in Figure \ref{fig:effectiveindex_bffem}. For all choices of $ \beta^- $, it demonstrates that the error indicator $ \eqref{equ:localind_interface} $ is an asymptotically exact \textit{a posteriori} error estimator for the interface problem, as stated in Theorem~\ref{thm:ae_interface}.

\begin{figure}[!ht]
   \centering
   \subcaptionbox{\label{fig:adaptvemesh_bffem}}
  {\includegraphics[width=0.39\textwidth]{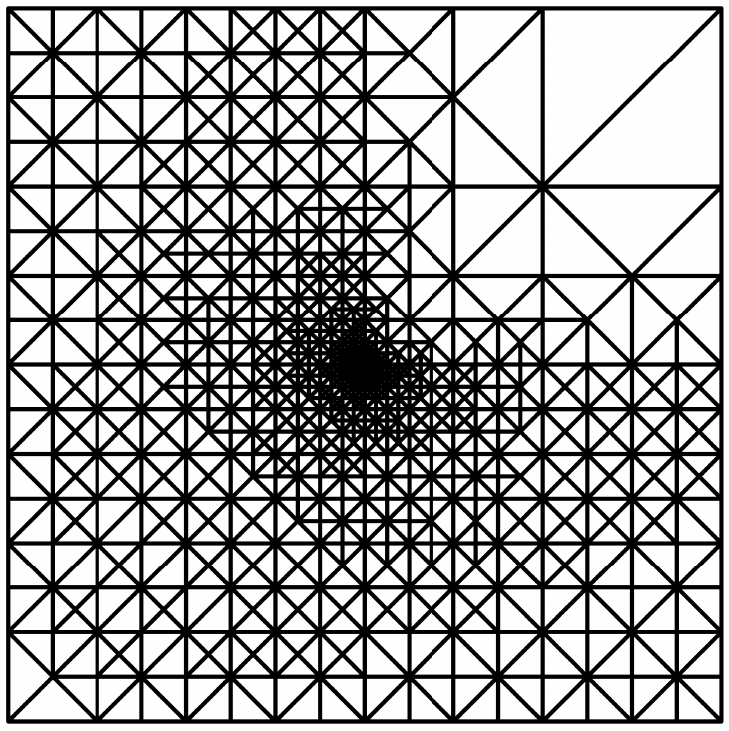}}
  \hspace{0.2in}
  \subcaptionbox{\label{fig:sol_bffem}}
   {\includegraphics[width=0.54\textwidth]{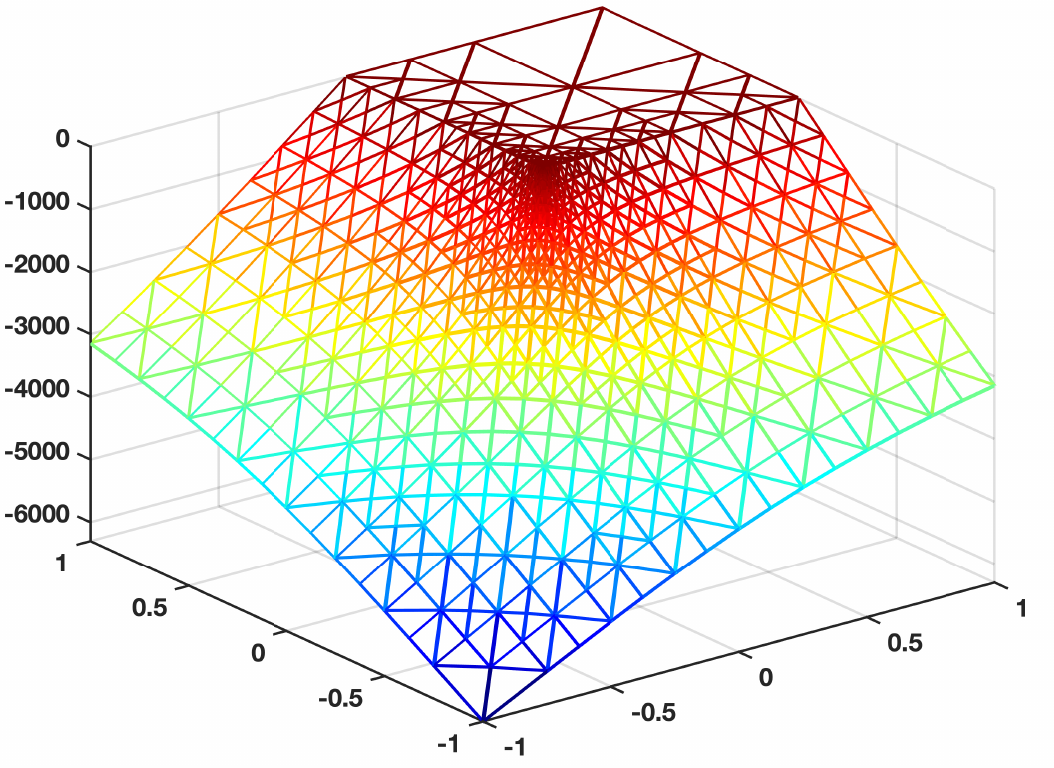}}
   \caption{Adaptive mesh and numerical solution using the body-fitted finite element method with  $\beta^-=1$ and $\beta^+=1000$:  (a) adaptively refined mesh; (b) body-fitted finite element solution. }\label{fig:mesh_bffem}
\end{figure}

\begin{figure}[!ht]
   \centering
   \subcaptionbox{\label{fig:result_bffem_10e3}}
  {\includegraphics[width=0.32\textwidth]{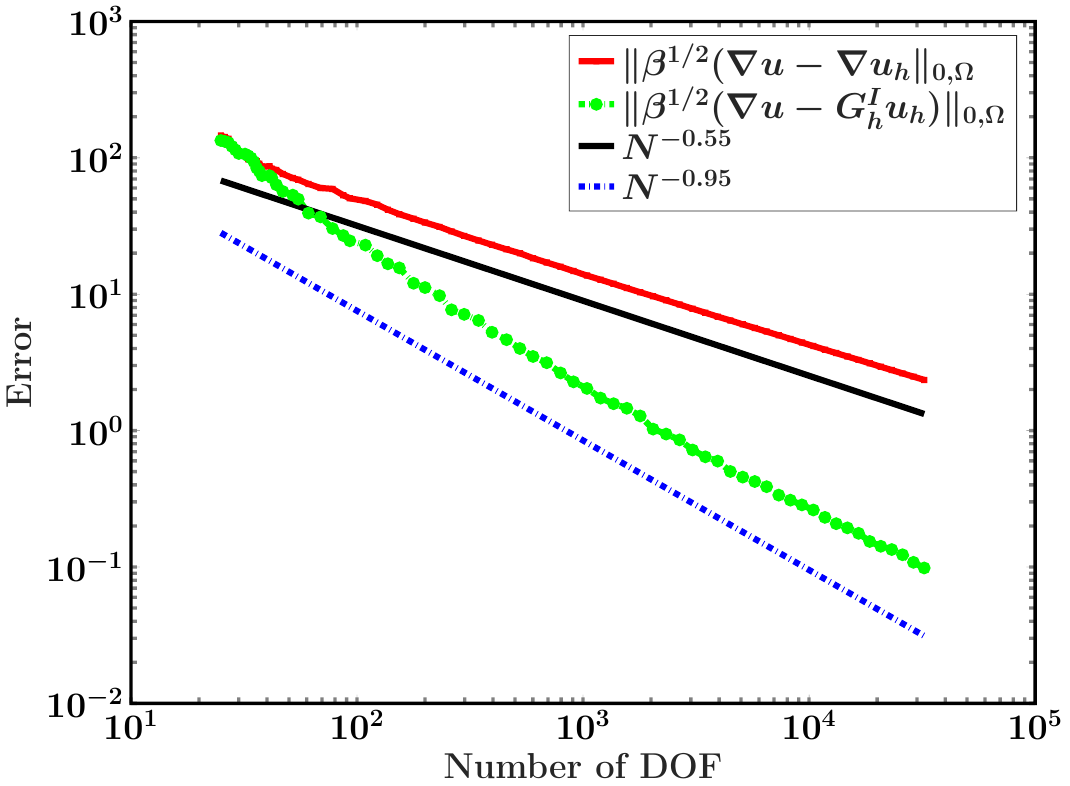}}
  \subcaptionbox{\label{fig:result_bffem_10e6}}
   {\includegraphics[width=0.32\textwidth]{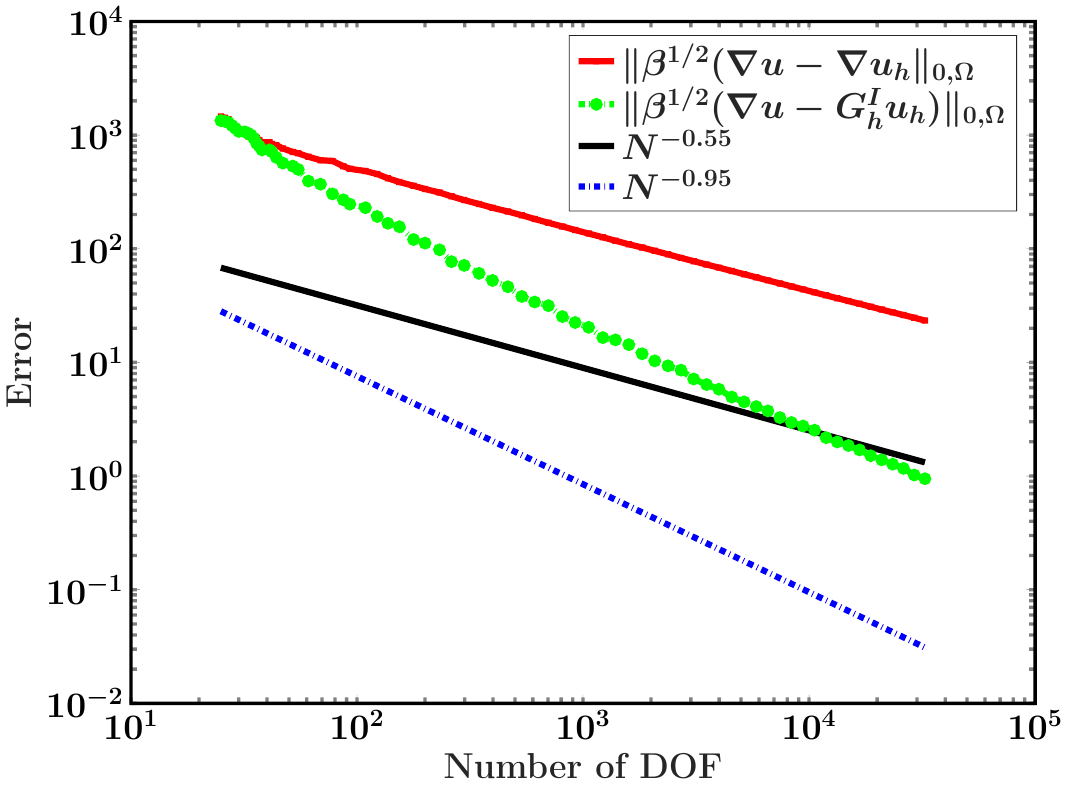}}
    \subcaptionbox{\label{fig:effectiveindex_bffem}}
   {\includegraphics[width=0.31\textwidth]{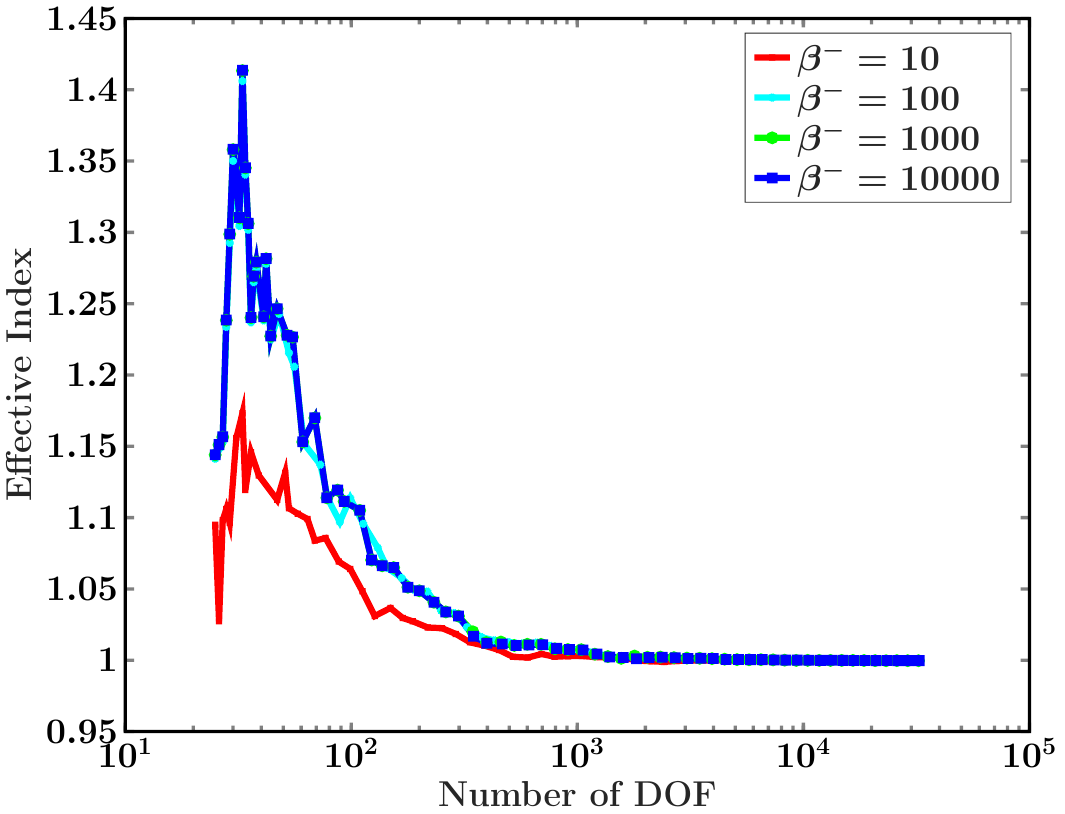}}
\caption{Numerical results for the adaptive body-fitted finite element method:  (a) numerical error for $\beta^+=1000$; (b) numerical error for $\beta^+=1000000$; (c) plot of the effective index. }
\label{fig:result_cutfem}
\end{figure}

\subsubsection{Numerical example for cut finite element methods}

\textbf{Numerical Example 7} In this example, we consider  the interface problem \eqref{equ:interface_equ}-\eqref{equ:interface_fluxjump} with homogeneous jumps. The interface $\Gamma$ a circular interface of radius $r_0 = 0.5$.  The exact solution is
\begin{equation*}
u(x,y) =
\left\{
\begin{array}{ll}
    \frac{r^p}{\beta^-}   &  \text{if }   (x,y)\in \Omega^-, \\
      \frac{r^p}{\beta^+} + \left( \frac{1}{\beta^-}-\frac{1}{\beta^+} \right)r_0^p&  \text{if } (x,y)\in \Omega^+,\\
   \end{array}
\right.
\end{equation*}
where $r = \sqrt{x^2+y^2}$.

When $p = 0.5$, there is a singularity at the origin, as plotted in Figure \ref{fig:sol_cutfem}. In fact, we can show that $u \in H^{1.5 - \epsilon}$ for small $\epsilon > 0$. The cut finite element method with uniform refinement can only achieve suboptimal convergence of $\mathcal{O}(N^{-0.25})$. We use the adaptive cut finite element method with the error estimator \eqref{equ:localind_interface}. The mesh is adaptively refined using the newest vertex bisection method without fitting the interface.
We consider two typical jump ratios: $\beta^- / \beta^+ = 1/1000$ (large jump) and $\beta^- / \beta^+ = 1/1000000$ (huge jump). In Figure \ref{fig:mesh_cutfem}, we plot the adaptively refined mesh and its corresponding cut finite element solution when $\beta^- = 1$ and $\beta^+ = 1000$. It can be seen from the figure that the mesh is refined in the vicinity of the singularity.

Figures \ref{fig:result_cutfem_10e6} show the numerical errors for both the large and huge jump cases. One can observe the optimal convergence rate of $\mathcal{O}(N^{-0.5})$ for the energy error and the superconvergence rate of $\mathcal{O}(N^{-1})$ for the recovered energy error. Figure \ref{fig:effectiveindex_cutfem} presents the history of the effective index. The recovery type a posteriori error estimator \eqref{equ:localind_interface} is asymptotically exact, which is consistent with Theorem \ref{thm:ae_interface}.

\begin{figure}[!ht]
   \centering
   \subcaptionbox{\label{fig:adaptvemesh_cutfem}}
  {\includegraphics[width=0.395\textwidth]{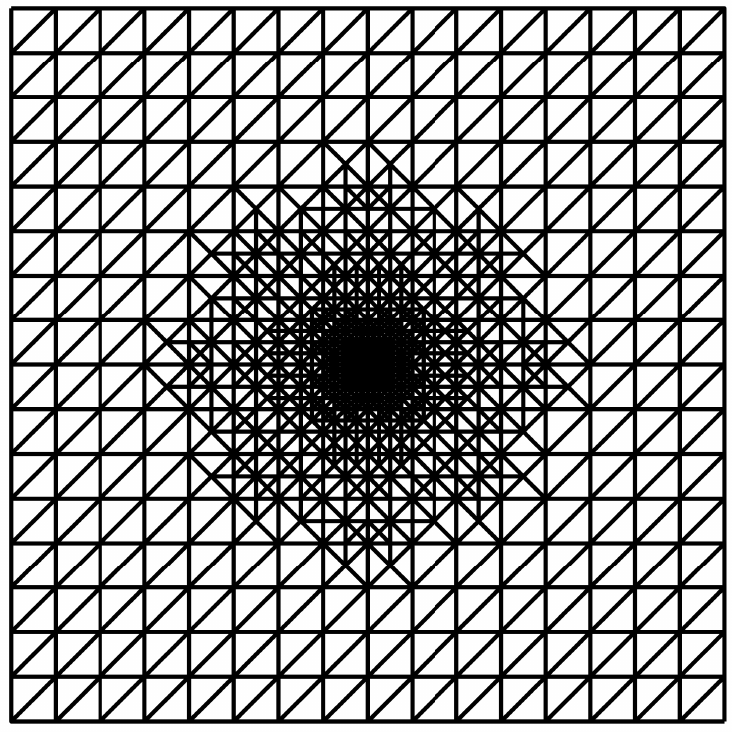}}
  \hspace{0.2in}
  \subcaptionbox{\label{fig:sol_cutfem}}
   {\includegraphics[width=0.555\textwidth]{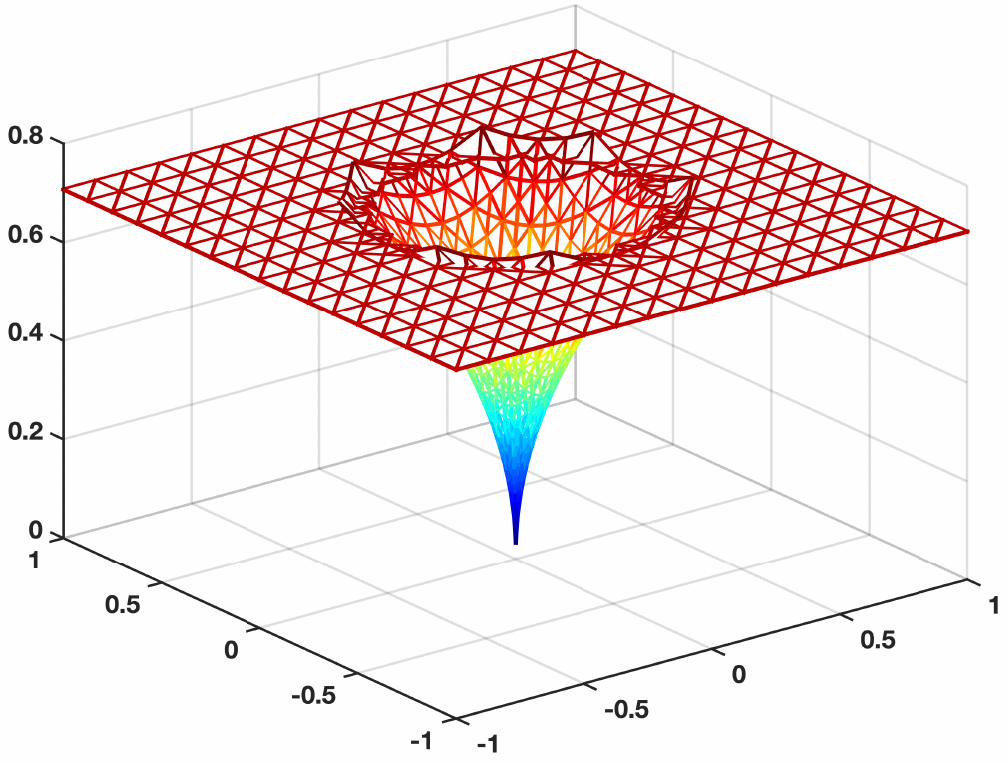}}
   \caption{Adaptive mesh and numerical solution using the cut finite element method with $\beta^-=1$ and $\beta^+=1000$: (a) adaptively refined mesh; (b) cut finite element solution. }\label{fig:mesh_cutfem}
\end{figure}

\begin{figure}[!ht]
   \centering
    \subcaptionbox{\label{fig:result_cutfem_10e3}}
   {\includegraphics[width=0.32\textwidth]{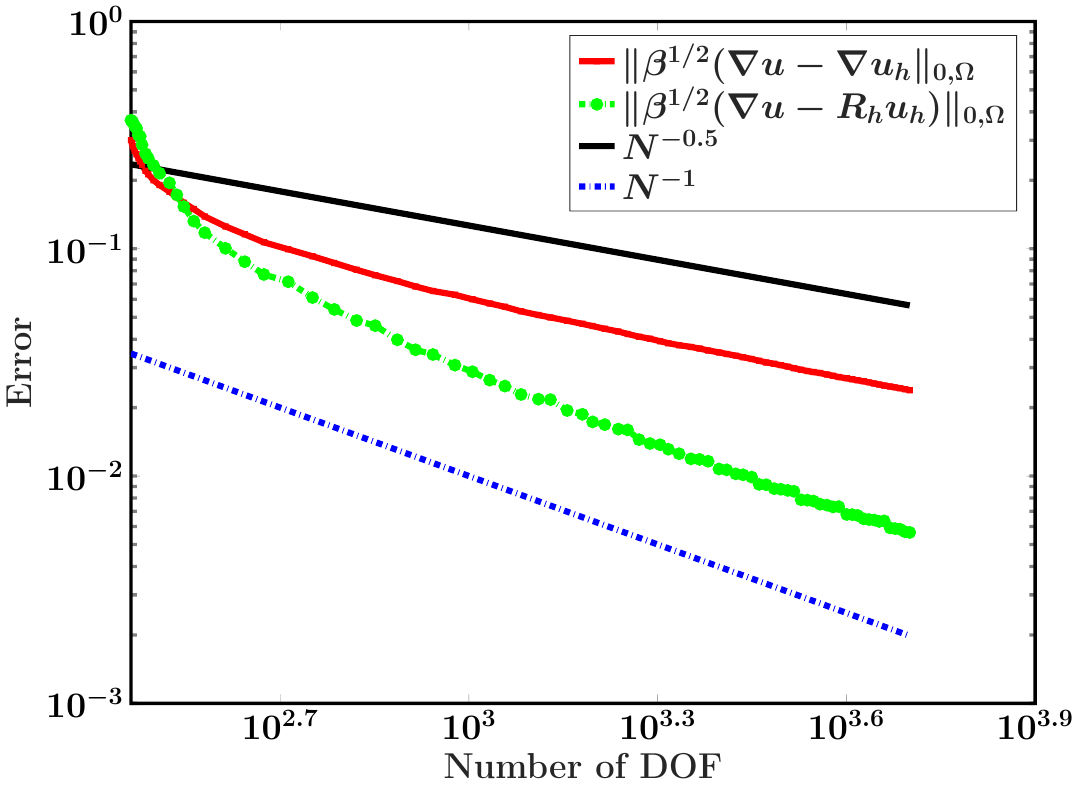}}
  \subcaptionbox{\label{fig:result_cutfem_10e6}}
   {\includegraphics[width=0.32\textwidth]{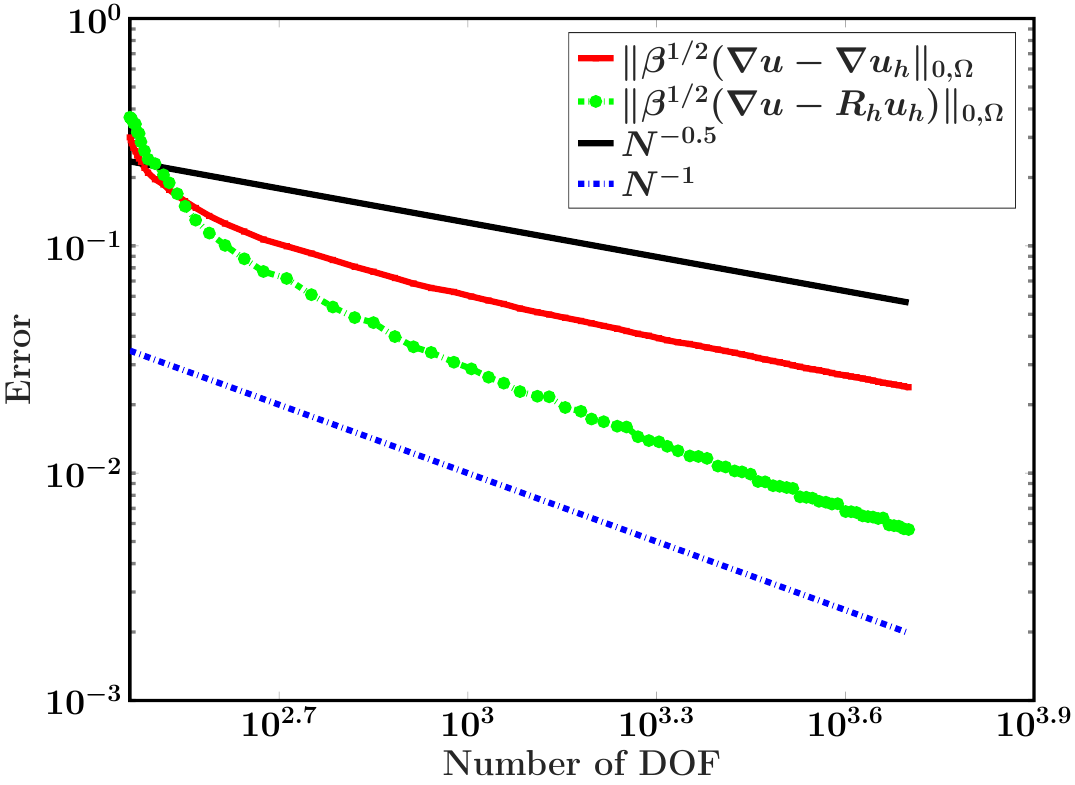}}
    \subcaptionbox{\label{fig:effectiveindex_cutfem}}
   {\includegraphics[width=0.32\textwidth]{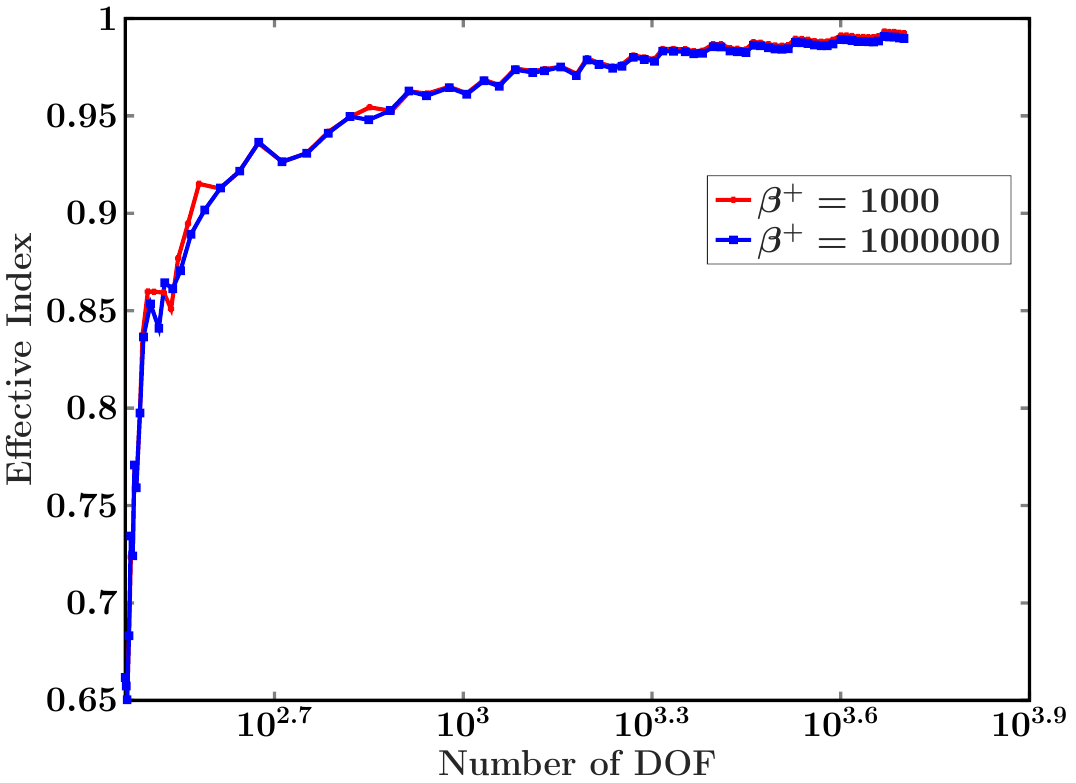}}
\caption{Numerical results for the adaptive cut finite element method: (a) numerical error for $\beta^+=1000$; (b) numerical error for $\beta^+=1000000$; (c) plot of the effective index. }
\label{fig:result_cutfem1}
\end{figure}

\section{Conclusion}\label{sec:con}
In this paper, we present a review of polynomial preserving recovery and its application in constructing a simple \textit{a posteriori} error estimator for adaptive methods. Assuming certain conditions on the underlying meshes, we establish the superconvergence and, in some cases, the ultraconvergence properties of the recovery operators. Using these superconvergence results, we show that the ZZ-type recovery based \textit{a posteriori} error estimator is asymptotically exact. Benchmark numerical examples are provided to demonstrate the asymptotic exactness of recovery-type \textit{a posteriori} error estimators. In this paper, we focus on the two-dimensional case with triangular meshes. We also draw the reader's attention to the fact that superconvergence results have been studied in parallel using quadrilateral elements in 2D \cite{zhang2004pprquadrilaterials} and tetrahedral and hexahedral elements in 3D \cite{naga2005ppr2d3d, chen2014super3d}.

We wish to emphasize that the application of recovery techniques has been significantly extended beyond their original use in \textit{a posteriori} error estimation. These techniques now play an important  role across various  scenarios, from post-processing to pre-processing. In the following, we provide an incomplete list of their applications:
\begin{itemize}
	\item Eigenvalue enhancement \cite{naga2006eigenenhance, wu2009eigenvalueadaptive, naga2012functionrecovery, guo2017superconvergenttwogrid, guo2019blochrecovery, cai2024surfacebiharmonic}.
	\item Anisotropic mesh adaption \cite{huang2011movingmesh, picasso2003anisotropic, fortin2024recovery}.
	\item  Development of new finite element methods for high-order partial differential equations  \cite{chen2017rbfemeigenvalue, guo2018linearfem, xu2019hbfem, adetola2022recovery, guo2018sixthorder}.
	\item Design of novel finite element methods for partial differential equations in non-divergence form \cite{chu2024rbfemhjb, lakkis2011nondiv, xu2023nondivergence}.
	\item Discretization of level set functions \cite{reusken2013levelsetrecovery}.
	\item High-order geometric discretization for surfaces  \cite{dong2024super, grande2016highsurface}.
	\item Gradient enhancement  in discretization nonlinear partial differential equations \cite{dong2024ot, wu2024pprpic}.
\end{itemize}

\section*{Acknowledgments}
We would like to thank the editor for their valuable comments and suggestions, which have significantly improved the paper. We also thank Dr. Ying Cai and Dr. Ren Zhao for reading the draft and for their helpful discussions.

H. Guo was supported in part by the Andrew Sisson Fund, Dyason Fellowship, and the Faculty of Science Researcher Development Grant of the University of Melbourne. Z. Zhang was supported in part by the NSFC grant 12131005.

\bibliographystyle{plainurl} 
\bibliography{references}

\begin{thebibliography}{100}

\bibitem{adetola2022recovery}
J.~Adetola, B.~Ahounou, Gerard Awanou, and H.~Guo.
\newblock Low order mixed finite element approximations of the {M}onge-{A}mp\`ere equation.
\newblock {\em Int. J. Numer. Anal. Model.}, 19(5):669--684, 2022.

\bibitem{agouzal2022hessianrecovery}
A.~Agouzal and Y.~Vassilevski.
\newblock On a discrete {H}essian recovery for {$P_1$} finite elements.
\newblock {\em J. Numer. Math.}, 10(1):1--12, 2002.

\bibitem{ainsworth1993flux}
M.~Ainsworth and J.~T. Oden.
\newblock A posteriori error estimators for second order elliptic systems. {II}. {A}n optimal order process for calculating self-equilibrating fluxes.
\newblock {\em Comput. Math. Appl.}, 26(9):75--87, 1993.

\bibitem{oden2000posterioribook}
M.~Ainsworth and J.~T. Oden.
\newblock {\em A posteriori error estimation in finite element analysis}.
\newblock Pure and Applied Mathematics (New York). Wiley-Interscience [John Wiley \& Sons], New York, 2000.

\bibitem{ann2012robustnitsche}
C.~Annavarapu, M.~Hautefeuille, and J.~E. Dolbow.
\newblock A robust {N}itsche's formulation for interface problems.
\newblock {\em Comput. Methods Appl. Mech. Engrg.}, 225/228:44--54, 2012.

\bibitem{axler2024labook}
Sheldon Axler.
\newblock {\em Linear algebra done right}.
\newblock Undergraduate Texts in Mathematics. Springer, Cham, fourth edition, 2024.

\bibitem{babuska1970bodyfitted}
I.~Babu{\v{s}}ka.
\newblock The finite element method for elliptic equations with discontinuous coefficients.
\newblock {\em Computing (Arch. Elektron. Rechnen)}, 5:207--213, 1970.

\bibitem{babuska1978badaptive}
I.~Babu{\v{s}}ka and W.~C. Rheinboldt.
\newblock A-posteriori error estimates for the finite element method.
\newblock {\em Internat. J. Numer. Methods Engrg.}, 12(10):1597--1615, 1978.

\bibitem{babuska1978aadaptive}
I.~Babu{\v{s}}ka and W.~C. Rheinboldt.
\newblock Error estimates for adaptive finite element computations.
\newblock {\em SIAM J. Numer. Anal.}, 15(4):736--754, 1978.

\bibitem{babushka2001fembook}
I.~Babuska and T.~Strouboulis.
\newblock {\em The finite element method and its reliability}.
\newblock Numerical Mathematics and Scientific Computation. The Clarendon Press, Oxford University Press, New York, 2001.

\bibitem{babushka1996computerassistant}
I.~Babuska, T.~Strouboulis, C.~S. Upadhyay, and S.~K. Gangaraj.
\newblock Computer-based proof of the existence of superconvergence points in the finite element method; superconvergence of the derivatives in finite element solutions of {L}aplace's, {P}oisson's, and the elasticity equations.
\newblock {\em Numer. Methods Partial Differential Equations}, 12(3):347--392, 1996.

\bibitem{bangerth2003adaptivebook}
W.~Bangerth and R.~Rannacher.
\newblock {\em Adaptive finite element methods for differential equations}.
\newblock Lectures in Mathematics ETH Z\"urich. Birkh\"auser Verlag, Basel, 2003.

\bibitem{bank1996hierarchical}
R.~E. Bank.
\newblock Hierarchical bases and the finite element method.
\newblock In {\em Acta numerica, 1996}, volume~5 of {\em Acta Numer.}, pages 1--43. Cambridge Univ. Press, Cambridge, 1996.

\bibitem{bank1981redgreen}
R.~E. Bank and A.~H. Sherman.
\newblock An adaptive, multilevel method for elliptic boundary value problems.
\newblock {\em Computing}, 26(2):91--105, 1981.

\bibitem{bank2003asymptotic1}
R.~E. Bank and J.~Xu.
\newblock Asymptotically exact a posteriori error estimators. {I}. {G}rids with superconvergence.
\newblock {\em SIAM J. Numer. Anal.}, 41(6):2294--2312, 2003.

\bibitem{barrau2012robust}
N.~Barrau, R.~Becker, E.~Dubach, and R.~Luce.
\newblock A robust variant of {NXFEM} for the interface problem.
\newblock {\em C. R. Math. Acad. Sci. Paris}, 350(15-16):789--792, 2012.

\bibitem{bartels2024flux}
Sören Bartels and Alex Kaltenbach.
\newblock Chapter seven - exact a posteriori error control for variational problems via convex duality and explicit flux reconstruction.
\newblock In F.~Chouly, S.~P.~A. Bordas, R.~Becker, and P.~Omnes, editors, {\em Error Control, Adaptive Discretizations, and Applications, Part 1}, volume~58 of {\em Advances in Applied Mechanics}, pages 295--375. Elsevier, 2024.

\bibitem{becker2001review}
R.~Becker and R.~Rannacher.
\newblock An optimal control approach to a posteriori error estimation in finite element methods.
\newblock {\em Acta Numer.}, 10:1--102, 2001.

\bibitem{beirao2013vem}
L.~Beir\~ao~da Veiga, F.~Brezzi, A.~Cangiani, G.~Manzini, L.~D. Marini, and A.~Russo.
\newblock Basic principles of virtual element methods.
\newblock {\em Math. Models Methods Appl. Sci.}, 23(1):199--214, 2013.

\bibitem{beirao2014mfdmbook}
L.~Beir\~ao~da Veiga, K.~Lipnikov, and G.~Manzini.
\newblock {\em The mimetic finite difference method for elliptic problems}, volume~11 of {\em MS\&A. Modeling, Simulation and Applications}.
\newblock Springer, Cham, 2014.

\bibitem{binev2003adaptive}
P.~Binev, W.~Dahmen, and R.~DeVore.
\newblock Adaptive finite element methods with convergence rates.
\newblock {\em Numer. Math.}, 97(2):219--268, 2004.

\bibitem{bonito2024adaptivereview}
A.~Bonito, C.~Canuto, R.~H. Nochetto, and Andreas Veeser.
\newblock Adaptive finite element methods.
\newblock {\em Acta Numer.}, 33:163--485, 2024.

\bibitem{bordas2007recovery}
S.~Bordas and M.~Duflot.
\newblock Derivative recovery and a posteriori error estimate for extended finite elements.
\newblock {\em Comput. Methods Appl. Mech. Engrg.}, 196(35-36):3381--3399, 2007.

\bibitem{boroomand1997irep}
B.~Boroomand and O.~C. Zienkiewicz.
\newblock An improved {REP} recovery and the effectivity robustness test.
\newblock {\em Internat. J. Numer. Methods Engrg.}, 40(17):3247--3277, 1997.

\bibitem{boroomand1997rep}
B.~Boroomand and O.~C. Zienkiewicz.
\newblock Recovery by equilibrium in patches ({REP}).
\newblock {\em Internat. J. Numer. Methods Engrg.}, 40(1):137--164, 1997.

\bibitem{bramble1977averaging}
J.~H. Bramble and A.~H. Schatz.
\newblock Higher order local accuracy by averaging in the finite element method.
\newblock {\em Math. Comp.}, 31(137):94--111, 1977.

\bibitem{brenner2010adaptivec0ip}
S.~C. Brenner, T.~Gudi, and L.-y. Sung.
\newblock An a posteriori error estimator for a quadratic {$C^0$}-interior penalty method for the biharmonic problem.
\newblock {\em IMA J. Numer. Anal.}, 30(3):777--798, 2010.

\bibitem{brenner2008fembook}
S.~C. Brenner and L.~R. Scott.
\newblock {\em The mathematical theory of finite element methods}, volume~15 of {\em Texts in Applied Mathematics}.
\newblock Springer, New York, third edition, 2008.

\bibitem{brenner2005c0ip}
S.~C. Brenner and L.-Y. Sung.
\newblock {$C^0$} interior penalty methods for fourth order elliptic boundary value problems on polygonal domains.
\newblock {\em J. Sci. Comput.}, 22/23:83--118, 2005.

\bibitem{burman2015cutfem}
E.~Burman, S.~Claus, P.~Hansbo, M.~G. Larson, and A.~Massing.
\newblock Cut{FEM}: discretizing geometry and partial differential equations.
\newblock {\em Internat. J. Numer. Methods Engrg.}, 104(7):472--501, 2015.

\bibitem{cai2024surfacebiharmonic}
Y.~Cai, H.~Guo, and Z.~Zhang.
\newblock Continuous linear finite element method for biharmonic problems on surfaces, 2024.
\newblock \href {https://arxiv.org/abs/2404.17958 [math.NA]} {\path{arXiv:2404.17958 [math.NA]}}.

\bibitem{cai2024superconvergentpostprocessingc0interior}
Y.~Cai, H.~Guo, and Z.~Zhang.
\newblock Superconvergent postprocessing of {$C^0$} interior penalty method, 2024.
\newblock \href {https://arxiv.org/abs/2401.12589 [math.NA]} {\path{arXiv:2401.12589 [math.NA]}}.

\bibitem{caoweiming2015granisotropic}
W.~Cao.
\newblock Superconvergence analysis of the linear finite element method and a gradient recovery postprocessing on anisotropic meshes.
\newblock {\em Math. Comp.}, 84(291):89--117, 2015.

\bibitem{cao2014superdg}
W.~Cao, Z.~Zhang, and Q.~Zou.
\newblock Superconvergence of discontinuous {G}alerkin methods for linear hyperbolic equations.
\newblock {\em SIAM J. Numer. Anal.}, 52(5):2555--2573, 2014.

\bibitem{cao2015fvm2k}
W.~Cao, Z.~Zhang, and Q.~Zou.
\newblock Is {$2k$}-conjecture valid for finite volume methods?
\newblock {\em SIAM J. Numer. Anal.}, 53(2):942--962, 2015.

\bibitem{cc2002average}
C.~Carstensen and S.~Bartels.
\newblock Each averaging technique yields reliable a posteriori error control in {FEM} on unstructured grids. {I}. {L}ow order conforming, nonconforming, and mixed {FEM}.
\newblock {\em Math. Comp.}, 71(239):945--969, 2002.

\bibitem{chamoin2023postererrorreview}
L.~Chamoin and F.~Legoll.
\newblock An introductory review on a posteriori error estimation in finite element computations.
\newblock {\em SIAM Rev.}, 65(4):963--1028, 2023.

\bibitem{chen1998elementanalysis}
C.~Chen.
\newblock Element analysis method and superconvergence.
\newblock In {\em Finite element methods ({J}yv\"askyl\"a, 1997)}, volume 196 of {\em Lecture Notes in Pure and Appl. Math.}, pages 71--84. Dekker, New York, 1998.

\bibitem{chen2001structure}
C.~Chen.
\newblock Structure theory of superconvergence of finite elements.
\newblock {\em Hunan Science and Technology Press, Changsha}, 2001.

\bibitem{chen20122kconjecture}
C.~Chen and S.~Hu.
\newblock The highest order superconvergence for bi-{$k$} degree rectangular elements at nodes: a proof of {$2k$}-conjecture.
\newblock {\em Math. Comp.}, 82(283):1337--1355, 2013.

\bibitem{chen1995high}
C.~Chen and Y.~Huang.
\newblock High accuracy theory of finite element methods.
\newblock {\em Hunan science and technology press, Changsha}, 1995, 1995.

\bibitem{chen2017rbfemeigenvalue}
H.~Chen, H.~Guo, Z.~Zhang, and Q.~Zou.
\newblock A {$C^0$} linear finite element method for two fourth-order eigenvalue problems.
\newblock {\em IMA J. Numer. Anal.}, 37(4):2120--2138, 2017.

\bibitem{chen2014super3d}
J.~Chen, D.~Wang, and Q.~Du.
\newblock Linear finite element superconvergence on simplicial meshes.
\newblock {\em Math. Comp.}, 83(289):2161--2185, 2014.

\bibitem{chen1998bffem}
Z.~Chen and J.~Zou.
\newblock Finite element methods and their convergence for elliptic and parabolic interface problems.
\newblock {\em Numer. Math.}, 79(2):175--202, 1998.

\bibitem{chi2019vemppr}
H.~Chi, L.~Beir\~ao~da Veiga, and G.~H. Paulino.
\newblock A simple and effective gradient recovery scheme and {\it a posteriori} error estimator for the virtual element method ({VEM}).
\newblock {\em Comput. Methods Appl. Mech. Engrg.}, 347:21--58, 2019.

\bibitem{chu2024rbfemhjb}
T.~Chu, H.~Guo, and Z.~Zhang.
\newblock Recovery based finite element methods for hamilton-jacobi-bellman equation with cordes coefficients.
\newblock {\em SIAM J. Numer. Anal.}, 2024.
\newblock , To appear.

\bibitem{ciarlet2002fembook}
P.~G. Ciarlet.
\newblock {\em The finite element method for elliptic problems}, volume~40 of {\em Classics in Applied Mathematics}.
\newblock Society for Industrial and Applied Mathematics (SIAM), Philadelphia, PA, 2002.
\newblock Reprint of the 1978 original [North-Holland, Amsterdam; MR0520174 (58 \#25001)].

\bibitem{deboor1973super}
C.~de~Boor and B.~Swartz.
\newblock Collocation at {G}aussian points.
\newblock {\em SIAM J. Numer. Anal.}, 10:582--606, 1973.

\bibitem{dong2020pppr}
G.~Dong and H.~Guo.
\newblock Parametric polynomial preserving recovery on manifolds.
\newblock {\em SIAM J. Sci. Comput.}, 42(3):A1885--A1912, 2020.

\bibitem{dong2024super}
G.~Dong, H.~Guo, and T.~Guo.
\newblock Superconvergence of differential structure for finite element methods on perturbed surface meshes.
\newblock {\em J. Comp. Math.}, 2024.

\bibitem{dong2024ot}
G.~Dong, H.~Guo, C.~Jiang, and Z.~Shi.
\newblock Gradient enhanced admm algorithm for dynamic optimal transport on surfaces, 2024.
\newblock \href {https://arxiv.org/abs/2406.16285 [math.NA]} {\path{arXiv:2406.16285 [math.NA]}}.

\bibitem{dorfler1996adaptiveconvergence}
W.~D{\"o}rfler.
\newblock A convergent adaptive algorithm for {P}oisson's equation.
\newblock {\em SIAM J. Numer. Anal.}, 33(3):1106--1124, 1996.

\bibitem{douglas1973super}
J.~Douglas, Jr. and T.~Dupont.
\newblock Some superconvergence results for {G}alerkin methods for the approximate solution of two-point boundary problems.
\newblock In {\em Topics in numerical analysis ({P}roc. {R}oy. {I}rish {A}cad. {C}onf., {U}niversity {C}oll., {D}ublin, 1972)}, pages 89--92. Roy. Irish Acad., Dublin, 1973.

\bibitem{douglas1973superconvergenceproof}
J.~Douglas, Jr. and T.~Dupont.
\newblock Superconvergence for {G}alerkin methods for the two point boundary problem via local projections.
\newblock {\em Numer. Math.}, 21:270--278, 1973/74.

\bibitem{douglas1974super}
J.~Douglas, Jr., T.~Dupont, and M.~F. Wheeler.
\newblock An {$L\sp{\infty }$} estimate and a superconvergence result for a {G}alerkin method for elliptic equations based on tensor products of piecewise polynomials.
\newblock {\em Rev. Fran\c caise Automat. Informat. Recherche Op\'erationnelle S\'er. Rouge}, 8(no. R-2):61--66, 1974.

\bibitem{du2019dgppr}
Y.~Du and Z.~Zhang.
\newblock Supercloseness of linear {DG}-{FEM} and its superconvergence based on the polynomial preserving recovery for {H}elmholtz equation.
\newblock {\em J. Sci. Comput.}, 79(3):1713--1736, 2019.

\bibitem{edtmayer2024adaptive}
B~Endtmayer, U~Langer, T~Richter, A.~Schafelner, and T.~Wick.
\newblock Chapter two - a posteriori single- and multi-goal error control and adaptivity for partial differential equations.
\newblock In F.~Chouly, S.~P.~A. Bordas, R.~Becker, and P.~Omnes, editors, {\em Error Control, Adaptive Discretizations, and Applications, Part 2}, volume~59 of {\em Advances in Applied Mechanics}, pages 19--108. Elsevier, 2024.

\bibitem{engel2002c0ip}
G.~Engel, K.~Garikipati, T.~J.~R. Hughes, M.~G. Larson, L.~Mazzei, and R.~L. Taylor.
\newblock Continuous/discontinuous finite element approximations of fourth-order elliptic problems in structural and continuum mechanics with applications to thin beams and plates, and strain gradient elasticity.
\newblock {\em Comput. Methods Appl. Mech. Engrg.}, 191(34):3669--3750, 2002.

\bibitem{evans2010pdebook}
L.~C. Evans.
\newblock {\em Partial differential equations}, volume~19 of {\em Graduate Studies in Mathematics}.
\newblock American Mathematical Society, Providence, RI, second edition, 2010.

\bibitem{feischl2014bemzz}
M.~Feischl, T.~F\"uhrer, M.~Karkulik, and D.~Praetorius.
\newblock Z{Z}-type a posteriori error estimators for adaptive boundary element methods on a curve.
\newblock {\em Eng. Anal. Bound. Elem.}, 38:49--60, 2014.

\bibitem{fortin2024recovery}
A.~Fortin, T.~Thomas~Briffard, L.~Plasman, and S.~L\'eger.
\newblock Chapter three - an anisotropic mesh adaptation method based on gradient recovery and optimal shape elements.
\newblock In F.~Chouly, S.~P.~A. Bordas, R.~Becker, and P.~Omnes, editors, {\em Error Control, Adaptive Discretizations, and Applications, Part 1}, volume~58 of {\em Advances in Applied Mechanics}, pages 101--148. Elsevier, 2024.

\bibitem{gonzalez2013recovery}
O.~A. Gonz\'alez-Estrada, S.~Natarajan, J.~J. R\'odenas, H.~Nguyen-Xuan, and S.~P.~A. Bordas.
\newblock Efficient recovery-based error estimation for the smoothed finite element method for smooth and singular linear elasticity.
\newblock {\em Comput. Mech.}, 52(1):37--52, 2013.

\bibitem{grande2016highsurface}
J.~Grande and A.~Reusken.
\newblock A higher order finite element method for partial differential equations on surfaces.
\newblock {\em SIAM J. Numer. Anal.}, 54(1):388--414, 2016.

\bibitem{guo2020scr}
H.~Guo.
\newblock Surface {C}rouzeix-{R}aviart element for the {L}aplace-{B}eltrami equation.
\newblock {\em Numer. Math.}, 144(3):527--551, 2020.

\bibitem{guo2019grvem}
H.~Guo, C.~Xie, and R.~Zhao.
\newblock Superconvergent gradient recovery for virtual element methods.
\newblock {\em Math. Models Methods Appl. Sci.}, 29(11):2007--2031, 2019.

\bibitem{guo2017grifem}
H.~Guo and X.~Yang.
\newblock Gradient recovery for elliptic interface problem: {II}. {I}mmersed finite element methods.
\newblock {\em J. Comput. Phys.}, 338:606--619, 2017.

\bibitem{guo2018grbfem}
H.~Guo and X.~Yang.
\newblock Gradient recovery for elliptic interface problem: {I}. {B}ody-fitted mesh.
\newblock {\em Commun. Comput. Phys.}, 23(5):1488--1511, 2018.

\bibitem{guo2018grcutfem}
H.~Guo and X.~Yang.
\newblock Gradient recovery for elliptic interface problem: {III}. {N}itsche's method.
\newblock {\em J. Comput. Phys.}, 356:46--63, 2018.

\bibitem{guo2017superconvergentppifem}
H.~Guo, X.~Yang, and Z.~Zhang.
\newblock Superconvergence of partially penalized immersed finite element methods.
\newblock {\em IMA J. Numer. Anal.}, 38(4):2123--2144, 2018.

\bibitem{guo2019blochrecovery}
H.~Guo, X.~Yang, and Y.~Zhu.
\newblock Bloch theory-based gradient recovery method for computing topological edge modes in photonic graphene.
\newblock {\em J. Comput. Phys.}, 379:403--420, 2019.

\bibitem{guo2015crelement}
H.~Guo and Z.~Zhang.
\newblock Gradient recovery for the {C}rouzeix-{R}aviart element.
\newblock {\em J. Sci. Comput.}, 64(2):456--476, 2015.

\bibitem{guo2017hessianrecovery}
H.~Guo, Z.~Zhang, and R.~Zhao.
\newblock Hessian recovery for finite element methods.
\newblock {\em Math. Comp.}, 86(306):1671--1692, 2017.

\bibitem{guo2017superconvergenttwogrid}
H.~Guo, Z.~Zhang, and R.~Zhao.
\newblock Superconvergent two-grid methods for elliptic eigenvalue problems.
\newblock {\em J. Sci. Comput.}, 70(1):125--148, 2017.

\bibitem{guo2016pprboundary}
H.~Guo, Z.~Zhang, R.~Zhao, and Q.~Zou.
\newblock Polynomial preserving recovery on boundary.
\newblock {\em J. Comput. Appl. Math.}, 307:119--133, 2016.

\bibitem{guo2018linearfem}
H.~Guo, Z.~Zhang, and Q.~Zou.
\newblock A {$C^0$} linear finite element method for biharmonic problems.
\newblock {\em J. Sci. Comput.}, 74(3):1397--1422, 2018.

\bibitem{guo2018sixthorder}
H.~Guo, Z.~Zhang, and Q.~Zou.
\newblock A $c^0$ linear finite element method for sixth order elliptic equations, 2018.
\newblock \href {https://arxiv.org/abs/1804.03793 [math.NA]} {\path{arXiv:1804.03793 [math.NA]}}.

\bibitem{hansbo2002nitsche}
A.~Hansbo and P.~Hansbo.
\newblock An unfitted finite element method, based on {N}itsche's method, for elliptic interface problems.
\newblock {\em Comput. Methods Appl. Mech. Engrg.}, 191(47-48):5537--5552, 2002.

\bibitem{hinton1974globall2}
E.~Hinton and J.~S. Campbell.
\newblock Local and global smoothing of discontinuous finite element functions using a least squares method.
\newblock {\em Internat. J. Numer. Methods Engrg.}, 8:461--480, 1974.

\bibitem{hou2005petrov}
S.~Hou and X.-D. Liu.
\newblock A numerical method for solving variable coefficient elliptic equation with interfaces.
\newblock {\em J. Comput. Phys.}, 202(2):411--445, 2005.

\bibitem{huang2005meshquality}
W.~Huang.
\newblock Measuring mesh qualities and application to variational mesh adaptation.
\newblock {\em SIAM J. Sci. Comput.}, 26(5):1643--1666, 2005.

\bibitem{huang2011movingmesh}
W.~Huang and D.~Russell, R.
\newblock {\em Adaptive moving mesh methods}, volume 174 of {\em Applied Mathematical Sciences}.
\newblock Springer, New York, 2011.

\bibitem{huang2008quadraticsupercloseness}
Y.~Huang and J.~Xu.
\newblock Superconvergence of quadratic finite elements on mildly structured grids.
\newblock {\em Math. Comp.}, 77(263):1253--1268, 2008.

\bibitem{ji2014scifem}
Haifeng Ji, Jinru Chen, and Zhilin Li.
\newblock A symmetric and consistent immersed finite element method for interface problems.
\newblock {\em J. Sci. Comput.}, 61(3):533--557, 2014.

\bibitem{huang2014nonconvergenthessian}
L.~Kamenski and W.~Huang.
\newblock How a nonconvergent recovered {H}essian works in mesh adaptation.
\newblock {\em SIAM J. Numer. Anal.}, 52(4):1692--1708, 2014.

\bibitem{skumar2017spr4iga}
M.~Kumar, T.~Kvamsdal, and K.~A. Johannessen.
\newblock Superconvergent patch recovery and a posteriori error estimation technique in adaptive isogeometric analysis.
\newblock {\em Comput. Methods Appl. Mech. Engrg.}, 316:1086--1156, 2017.

\bibitem{lakhnay2000superconvergence}
A.~M. Lakhany, I.~Marek, and J.~R. Whiteman.
\newblock Superconvergence results on mildly structured triangulations.
\newblock {\em Comput. Methods Appl. Mech. Engrg.}, 189(1):1--75, 2000.

\bibitem{lakkis2011nondiv}
O.~Lakkis and T.~Pryer.
\newblock A finite element method for second order nonvariational elliptic problems.
\newblock {\em SIAM J. Sci. Comput.}, 33(2):786--801, 2011.

\bibitem{li2004counterexample}
B.~Li.
\newblock Lagrange interpolation and finite element superconvergence.
\newblock {\em Numer. Methods Partial Differential Equations}, 20(1):33--59, 2004.

\bibitem{libo1999analysisofspr}
Bo~Li and Z.~Zhang.
\newblock Analysis of a class of superconvergence patch recovery techniques for linear and bilinear finite elements.
\newblock {\em Numer. Methods Partial Differential Equations}, 15(2):151--167, 1999.

\bibitem{li2024fvmppr}
Y.~Li, P.~Yang, and Z.~Zhang.
\newblock Polynomial preserving recovery for the finite volume element methods under simplex meshes.
\newblock {\em Math. Comp.}, pages 907--928, 2024.

\bibitem{lin2014femlowerbound}
Q.~Lin, H.~Xie, and J.~Xu.
\newblock Lower bounds of the discretization error for piecewise polynomials.
\newblock {\em Math. Comp.}, 83(285):1--13, 2014.

\bibitem{lin2008naturalsuper}
R.~Lin and Z.~Zhang.
\newblock Natural superconvergence points in three-dimensional finite elements.
\newblock {\em SIAM J. Numer. Anal.}, 46(3):1281--1297, 2008.

\bibitem{lin2015ppifem}
T.~Lin, Y.~Lin, and X.~Zhang.
\newblock Partially penalized immersed finite element methods for elliptic interface problems.
\newblock {\em SIAM J. Numer. Anal.}, 53(2):1121--1144, 2015.

\bibitem{liseikin1999meshgeneration}
V.~D. Liseikin.
\newblock {\em Grid generation methods}.
\newblock Scientific Computation. Springer-Verlag, Berlin, 1999.

\bibitem{maisano2006recovery}
G.~Maisano, S.~Micheletti, S.~Perotto, and C.~L. Bottasso.
\newblock On some new recovery-based a posteriori error estimators.
\newblock {\em Comput. Methods Appl. Mech. Engrg.}, 195(37-40):4794--4815, 2006.

\bibitem{mitchell2017newestvertexbisection}
W.~F. Mitchell.
\newblock 30 years of newest vertex bisection.
\newblock {\em JNAIAM. J. Numer. Anal. Ind. Appl. Math.}, 11(1-2):11--22, 2017.

\bibitem{naga2004pprposteriori}
A.~Naga and Z.~Zhang.
\newblock A posteriori error estimates based on the polynomial preserving recovery.
\newblock {\em SIAM J. Numer. Anal.}, 42(4):1780--1800 (electronic), 2004.

\bibitem{naga2005ppr2d3d}
A.~Naga and Z.~Zhang.
\newblock The polynomial-preserving recovery for higher order finite element methods in 2{D} and 3{D}.
\newblock {\em Discrete Contin. Dyn. Syst. Ser. B}, 5(3):769--798, 2005.

\bibitem{naga2012functionrecovery}
A.~Naga and Z.~Zhang.
\newblock Function value recovery and its application in eigenvalue problems.
\newblock {\em SIAM J. Numer. Anal.}, 50(1):272--286, 2012.

\bibitem{naga2006eigenenhance}
A.~Naga, Z.~Zhang, and A.~Zhou.
\newblock Enhancing eigenvalue approximation by gradient recovery.
\newblock {\em SIAM J. Sci. Comput.}, 28(4):1289--1300, 2006.

\bibitem{nitsche1974negativenorm}
J.~A. Nitsche and A.~H. Schatz.
\newblock Interior estimates for {R}itz-{G}alerkin methods.
\newblock {\em Math. Comp.}, 28:937--958, 1974.

\bibitem{nochetto2009adaptive}
R.~H. Nochetto, K.~G. Siebert, and A.~Veeser.
\newblock Theory of adaptive finite element methods: an introduction.
\newblock In {\em Multiscale, nonlinear and adaptive approximation}, pages 409--542. Springer, Berlin, 2009.

\bibitem{osher2003levelsetbook}
S.~Osher and R.~Fedkiw.
\newblock {\em Level set methods and dynamic implicit surfaces}, volume 153 of {\em Applied Mathematical Sciences}.
\newblock Springer-Verlag, New York, 2003.

\bibitem{picasso2003anisotropic}
M.~Picasso.
\newblock An anisotropic error indicator based on {Z}ienkiewicz-{Z}hu error estimator: application to elliptic and parabolic problems.
\newblock {\em SIAM J. Sci. Comput.}, 24(4):1328--1355, 2003.

\bibitem{picasso2011hessianrecovery}
M.~Picasso, F.~Alauzet, H.~Borouchaki, and P.-L. George.
\newblock A numerical study of some {H}essian recovery techniques on isotropic and anisotropic meshes.
\newblock {\em SIAM J. Sci. Comput.}, 33(3):1058--1076, 2011.

\bibitem{repin2008posteriorbook}
S.~Repin.
\newblock {\em A posteriori estimates for partial differential equations}, volume~4 of {\em Radon Series on Computational and Applied Mathematics}.
\newblock Walter de Gruyter GmbH \& Co. KG, Berlin, 2008.

\bibitem{reusken2013levelsetrecovery}
A.~Reusken.
\newblock A finite element level set redistancing method based on gradient recovery.
\newblock {\em SIAM J. Numer. Anal.}, 51(5):2723--2745, 2013.

\bibitem{rodenas2007improvedspr}
J.~J. Rodenas, M.~Tur, F.~J. Fuenmayor, and A~Vercher.
\newblock Improvement of the superconvergent patch recovery technique by the use of constraint equations: the spr-c technique.
\newblock {\em Internat. J. Numer. Methods Engrg.}, 70(6):705--727, 2007.

\bibitem{schatz1996super}
A.~H. Schatz, I.~H. Sloan, and L.~B. Wahlbin.
\newblock Superconvergence in finite element methods and meshes that are locally symmetric with respect to a point.
\newblock {\em SIAM J. Numer. Anal.}, 33(2):505--521, 1996.

\bibitem{sethian1996levelsetbook}
J.~A. Sethian.
\newblock {\em Level set methods}, volume~3 of {\em Cambridge Monographs on Applied and Computational Mathematics}.
\newblock Cambridge University Press, Cambridge, 1996.
\newblock Evolving interfaces in geometry, fluid mechanics, computer vision, and materials science.

\bibitem{shashkov1996mfdmbook}
M.~Shashkov.
\newblock {\em Conservative finite-difference methods on general grids}.
\newblock Symbolic and Numeric Computation Series. CRC Press, Boca Raton, FL, 1996.
\newblock With 1 IBM-PC floppy disk (3.5 inch; HD).

\bibitem{shen2011spectralbook}
J.~Shen, T.~Tang, and L.-L. Wang.
\newblock {\em Spectral methods: Algorithms, analysis and applications}, volume~41 of {\em Springer Series in Computational Mathematics}.
\newblock Springer, Heidelberg, 2011.

\bibitem{song2015ipdgppr}
L.~Song and Z.~Zhang.
\newblock Polynomial preserving recovery of an over-penalized symmetric interior penalty {G}alerkin method for elliptic problems.
\newblock {\em Discrete Contin. Dyn. Syst. Ser. B}, 20(5):1405--1426, 2015.

\bibitem{song2015superdg}
L.~Song and Z.~Zhang.
\newblock Superconvergence property of an over-penalized discontinuous {G}alerkin finite element gradient recovery method.
\newblock {\em J. Comput. Phys.}, 299:1004--1020, 2015.

\bibitem{thomee1977highorder}
V.~Thom\'ee.
\newblock High order local approximations to derivatives in the finite element method.
\newblock {\em Math. Comp.}, 31(139):652--660, 1977.

\bibitem{ubertini2004patchenergy}
F.~Ubertini.
\newblock Patch recovery based on complementary energy.
\newblock {\em Internat. J. Numer. Methods Engrg.}, 59(11):1501--1538, 2004.

\bibitem{vallet2007hessianrecovery}
M.-G. Vallet, C.-M. Manole, J.~Dompierre, S.~Dufour, and F.~Guibault.
\newblock Numerical comparison of some {H}essian recovery techniques.
\newblock {\em Internat. J. Numer. Methods Engrg.}, 72(8):987--1007, 2007.

\bibitem{verfurth2013aposteribook}
R.~Verf\"{u}rth.
\newblock {\em A posteriori error estimation techniques for finite element methods}.
\newblock Numerical Mathematics and Scientific Computation. Oxford University Press, Oxford, 2013.

\bibitem{wahlbin1995superconvergencebook}
L.~Wahlbin.
\newblock {\em Superconvergence in {G}alerkin finite element methods}, volume 1605 of {\em Lecture Notes in Mathematics}.
\newblock Springer-Verlag, Berlin, 1995.

\bibitem{wahlbin1998generalprinciple}
L.~B. Wahlbin.
\newblock General principles of superconvergence in {G}alerkin finite element methods.
\newblock In {\em Finite element methods ({J}yv\"askyl\"a, 1997)}, volume 196 of {\em Lecture Notes in Pure and Appl. Math.}, pages 269--285. Dekker, New York, 1998.

\bibitem{wang2019ppredgeelement}
L.~Wang, Q.~Zhang, and Z.~Zhang.
\newblock Superconvergence analysis and {PPR} recovery of arbitrary order edge elements for {M}axwell's equations.
\newblock {\em J. Sci. Comput.}, 78(2):1207--1230, 2019.

\bibitem{wei2010sprsurface}
H.~Wei, L.~Chen, and Y.~Huang.
\newblock Superconvergence and gradient recovery of linear finite elements for the {L}aplace-{B}eltrami operator on general surfaces.
\newblock {\em SIAM J. Numer. Anal.}, 48(5):1920--1943, 2010.

\bibitem{wilson1963sa}
E.~L. Wilson.
\newblock {\em Finite Element Analysis of Two-dimensional Structures}.
\newblock PhD thesis, University of California, Berkeley, 1963.

\bibitem{wu2007superconvergenceonadaptivemesh}
H.~Wu and Z.~Zhang.
\newblock Can we have superconvergent gradient recovery under adaptive meshes?
\newblock {\em SIAM J. Numer. Anal.}, 45(4):1701--1722, 2007.

\bibitem{wu2009eigenvalueadaptive}
H.~Wu and Z.~Zhang.
\newblock Enhancing eigenvalue approximation by gradient recovery on adaptive meshes.
\newblock {\em IMA J. Numer. Anal.}, 29(4):1008--1022, 2009.

\bibitem{wu2024pprpic}
S.~Wu, J.~Bai, X.~He, R.~Zhao, and Y.~Cao.
\newblock An immersed selective discontinuous {G}alerkin method in particle-in-cell simulation with adaptive {C}artesian mesh and polynomial preserving recovery.
\newblock {\em J. Comput. Phys.}, 498:Paper No. 112703, 22, 2024.

\bibitem{xu2004analysisofrecovery}
J.~Xu and Z.~Zhang.
\newblock Analysis of recovery type a posteriori error estimators for mildly structured grids.
\newblock {\em Math. Comp.}, 73(247):1139--1152 (electronic), 2004.

\bibitem{xu2019hbfem}
M.~Xu, H.~Guo, and Q.~Zou.
\newblock Hessian recovery based finite element methods for the {C}ahn-{H}illiard equation.
\newblock {\em J. Comput. Phys.}, 386:524--540, 2019.

\bibitem{xu2023nondivergence}
M.~Xu, R.~Lin, and Q.~Zou.
\newblock A {$C^0$} linear finite element method for a second-order elliptic equation in non-divergence form with {C}ordes coefficients.
\newblock {\em Numer. Methods Partial Differential Equations}, 39(3):2244--2269, 2023.

\bibitem{zhang1996ultra}
Z.~Zhang.
\newblock Ultraconvergence of the patch recovery technique.
\newblock {\em Math. Comp.}, 65(216):1431--1437, 1996.

\bibitem{zhang2000ultraconvergenceofpatchrecovery}
Z.~Zhang.
\newblock Ultraconvergence of the patch recovery technique. {II}.
\newblock {\em Math. Comp.}, 69(229):141--158, 2000.

\bibitem{zhang2004pprquadrilaterials}
Z.~Zhang.
\newblock Polynomial preserving gradient recovery and a posteriori estimate for bilinear element on irregular quadrilaterals.
\newblock {\em Int. J. Numer. Anal. Model.}, 1(1):1--24, 2004.

\bibitem{zhang2003ultranconvergencezz}
Z.~Zhang and R.~Lin.
\newblock Ultraconvergence of {ZZ} patch recovery at mesh symmetry points.
\newblock {\em Numer. Math.}, 95(4):781--801, 2003.

\bibitem{zhang2005ppr}
Z.~Zhang and A.~Naga.
\newblock A new finite element gradient recovery method: superconvergence property.
\newblock {\em SIAM J. Sci. Comput.}, 26(4):1192--1213 (electronic), 2005.

\bibitem{zhang1998analysisofspr1}
Z.~Zhang and J.~Z. Zhu.
\newblock Analysis of the superconvergent patch recovery technique and a posteriori error estimator in the finite element method. {I}.
\newblock {\em Comput. Methods Appl. Mech. Engrg.}, 123(1-4):173--187, 1995.

\bibitem{zhang1998analysisofspr2}
Z.~Zhang and J.~Z. Zhu.
\newblock Analysis of the superconvergent patch recovery technique and a posteriori error estimator in the finite element method. {II}.
\newblock {\em Comput. Methods Appl. Mech. Engrg.}, 163(1-4):159--170, 1998.

\bibitem{zhu1989superconvergencebook}
Q.~Zhu and Q.~Lin.
\newblock {\em Superconvergence Theory of the Finite Element Method (in Chinese)}.
\newblock Hunan Science and Technique Press, Changsha, 1989.

\bibitem{zienkiewicz2013fembook}
O.~C. Zienkiewicz, R.~L. Taylor, and J.~Z. Zhu.
\newblock {\em The finite element method: its basis and fundamentals}.
\newblock Elsevier/Butterworth Heinemann, Amsterdam, seventh edition, 2013.

\bibitem{zhu1987simpleerror}
O.~C. Zienkiewicz and J.~Z. Zhu.
\newblock A simple error estimator and adaptive procedure for practical engineering analysis.
\newblock {\em Internat. J. Numer. Methods Engrg.}, 24(2):337--357, 1987.

\bibitem{zhu1992spr1}
O.~C. Zienkiewicz and J.~Z. Zhu.
\newblock The superconvergent patch recovery and a posteriori error estimates. {I}. {T}he recovery technique.
\newblock {\em Internat. J. Numer. Methods Engrg.}, 33(7):1331--1364, 1992.

\bibitem{zhu1992spr2}
O.~C. Zienkiewicz and J.~Z. Zhu.
\newblock The superconvergent patch recovery and a posteriori error estimates. {II}. {E}rror estimates and adaptivity.
\newblock {\em Internat. J. Numer. Methods Engrg.}, 33(7):1365--1382, 1992.

\bibitem{zlamal1977somesuper}
M.~Zl\'amal.
\newblock Some superconvergence results in the finite element method.
\newblock In {\em Mathematical aspects of finite element methods ({P}roc. {C}onf., {C}onsiglio {N}az. delle {R}icerche ({C}.{N}.{R}.), {R}ome, 1975)}, volume Vol. 606 of {\em Lecture Notes in Math.}, pages 353--362. Springer, Berlin-New York, 1977.

\bibitem{zlamal1978super}
M.~Zl\'amal.
\newblock Superconvergence and reduced integration in the finite element method.
\newblock {\em Math. Comp.}, 32(143):663--685, 1978.

\end{thebibliography}

\end{document}